\documentclass[11pt]{article}
\usepackage[top=1in, bottom=1in, left=1.25in, right=1.25in]{geometry}

\usepackage{amsmath,amssymb,amsthm,xcolor,enumerate,bbm}
\usepackage[unicode,breaklinks=true,colorlinks=true]{hyperref}

\usepackage{theoremref}
\usepackage{mathrsfs}
\usepackage[]{authblk}
\usepackage{comment}
\usepackage{cancel}
\usepackage{graphicx}

\usepackage{titletoc}
\titlecontents{section}[1.5em]{\vspace{-2pt}}%
    {\thecontentslabel\enspace\enspace}
    {}
   {\titlerule*[0.5pc]{.}\contentspage}%

\titlecontents{subsection}[3em]{\vspace{-2.5pt}}%
    {\thecontentslabel\enspace\enspace}
    {}
    {\titlerule*[0.5pc]{.}\contentspage}%
    
\titlecontents{subsubsection}[3em]{\vspace{-2.5pt}}%
    {\thecontentslabel\enspace\enspace}
    {}
    {\titlerule*[0.5pc]{.}\contentspage}%

\numberwithin{equation}{section}
\newtheorem{thm}{Theorem}[section]

\newtheorem{lem}[thm]{Lemma}
\newtheorem{prop}[thm]{Proposition}

\theoremstyle{remark}
\newtheorem{remark}[thm]{Remark}
\theoremstyle{definition}

\newcommand{\bke}[1]{\left ( #1 \right )}
\newcommand{\bkt}[1]{\left [ #1 \right ]}
\newcommand{\bket}[1]{\left \{ #1 \right \}}
\newcommand{\norm}[1]{\left \| #1 \right \|}

\newcommand{\R}{\mathbb{R}}
\newcommand{\N}{\mathbb{N}}
\renewcommand{\div}{\mathop{\rm div}\nolimits}
\newcommand{\curl} {\mathop{\rm curl}}
\newcommand{\pd}{\partial}
\newcommand\Ga{\Gamma}

\newcommand\La{\Lambda}
\newcommand\Om{\Omega}
\newcommand{\si}{\sigma}
\newcommand\De{\Delta}
\newcommand\de{\delta}
\newcommand{\nb}{\nabla}
\newcommand{\lec}{{\ \lesssim \ }}
\newcommand{\gec}{{\ \gtrsim \ }}

\newcommand{\bka}[1]{{\langle #1 \rangle}}
\newcommand{\abs}[1]{\left | #1 \right |}

\newcommand\al{\alpha}
\newcommand\be{\beta}
\newcommand\ep{\epsilon}
\newcommand\ve{\varepsilon}
\newcommand\e {\varepsilon}  %

\newcommand\Si{\Sigma}

\newcommand{\NN}{\mathbb{N}}

\newcommand{\LN}  {\mathrm{L\mkern-0.5mu N}}

\newcommand{\na}{\nabla}
\newcommand{\td}{\tilde}
\newcommand{\wt}{\widetilde}

\newcommand{\II}{I\!I}

\newcommand{\EQ}[1]{\begin{equation} #1 \end{equation}}
\newcommand{\EQS}[1]{\begin{equation}\begin{split} #1 \end{split}\end{equation}}

\newcommand{\EQN}[1]{\begin{equation*}\begin{split} #1 \end{split}\end{equation*}}

\newcommand{\uloc}{{\mathrm{uloc}}}
\newcommand{\loc}{{\mathrm{loc}}}

\newcommand{\one}{\mathbbm{1}}

\begin{document}

\title{Applications of the Green tensor estimates of the nonstationary Stokes system in the half space}

\author[1]{\rm Kyungkeun Kang}
\author[2]{\rm Baishun Lai}
\author[3]{\rm Chen-Chih Lai}
\author[4]{\rm Tai-Peng Tsai}
\affil[1]{\footnotesize Department of Mathematics, Yonsei University, Seoul 120-749, South Korea}
\affil[2]{\footnotesize LCSM (MOE) and School of Mathematics and Statistics, Hunan Normal University, Changsha 410081, Hunan, China}
\affil[3]{\footnotesize Department of Mathematics, Columbia University, New York, NY 10027, USA}
\affil[4]{\footnotesize Department of Mathematics, University of British Columbia, Vancouver, BC V6T 1Z2, Canada}

\date{}
\maketitle

\vspace{-1.4cm}
\begin{abstract}

In this paper, we present a series of applications of the pointwise estimates of the (unrestricted) Green tensor of the nonstationary Stokes system in the half space, established in our previous work [CMP 2023].
First, we show the $L^1$-$L^q$ estimates for the Stokes flow with possibly non-solenoidal $L^1$ initial data, generalizing the results of Giga--Matsui--Shimizu [Math.~Z.~1999] and Desch--Hieber--Pr\"uss [J.~Evol.~Equ.~2001]. 
Second, we construct mild solutions of the Navier--Stokes equations in the half space with mixed-type pointwise decay or with pointwise decay alongside boundary vanishing.
Finally, we explore various coupled fluid systems in the half space 
including viscous resistive magnetohydrodynamics equations, a coupled system for the flow and the magnetic field of MHD type, and the nematic liquid crystal flow. For each of these systems, we construct mild solutions in $L^q$, pointwise decay, and uniformly local $L^q$ spaces.

\end{abstract}

\vspace{-0.55cm}

\setcounter{tocdepth}{2}
\tableofcontents

\section{Introduction}

Building upon the foundational study \cite{Green} concerning the Green tensor of the nonstationary Stokes system in the half space, we advance our research by exploring further applications of the pointwise estimates of the Green tensor derived in the aforementioned reference.
This paper is aimed at demonstrating the considerable potential of the Green tensor estimates for a range of diverse applications.
For instance, we delve into its application to Navier-Stokes equations, magnetohydrodynamics equations, and nematic liquid crystal flows, showcasing the versatility and significance of these estimates across various fields of study.

To begin with, we present the background of our study.

\medskip
\noindent{\bf The Green tensor.}
Denote $x=(x_1,\ldots,x_{n-1},x_n) = (x',x_n)$ for $x \in \R^n$, $n \ge 2$,
 $\R^n_+ = \{  (x',x_n)\in \R^n \ | \ x_n>0\}$, and $\Si=\pd\R^n_+$.
 We consider the nonstationary Stokes system in the half space $\R^n_+$, $n\ge2$,
\begin{align}\label{E1.1}
\begin{split}
\left.
\begin{aligned}
\pd_t u -\Delta u+\nabla \pi = f + \na\cdot F\\
 \div u=0
\end{aligned}\ \right\}\ \ \mbox{in}\ \ \R_{+}^{n}\times (0,\infty),
\end{split}
\end{align}
with initial and \emph{no-slip} boundary conditions
\begin{align}\label{E1.2}
u(\cdot,0)=u_0; \qquad  %
u(x',0,t)=0\ \ \mbox{on}\ \ \Si\times (0,\infty).
\end{align}
Here  $u=(u_{1},\ldots,u_{n})$ is the velocity, $\pi$ is the pressure, $f=(f_{1},\ldots,f_{n})$, and $F=(F_{ij})_{i,j=1}^n$. They are defined for $(x,t)\in \R_{+}^{n}\times (0,\infty)$.

The \emph{Green tensor} $ G_{ij}(x,y,t)$ of the Stokes system \eqref{E1.1}
is defined for $(x,y,t) \in\R^n _+ \times \R^n_+ \times \R$ and
 $1\le i,j\le n$ so that,
for suitable $f$, $F$ and $u_0$,
the solution of \eqref{E1.1} is given by
\EQS{
\label{E1.3}
u_i(x,t)& = \sum_{j=1}^n\int_{\R^n_+}G_{ij}(x,y,t)(u_0)_j(y)\,dy+
\sum_{j=1}^n\int_0^t \int_{\R^n_+} G_{ij}(x,y,t-s) f_j(y,s)\,dy\,ds\\
&\quad - \sum_{j,k=1}^n \int_0^t \int_{\R^n_+} \pd_{y_k} G_{ij}(x,y,t-s) F_{kj}(y,s)\, dyds.
}
Here, weird solutions of \eqref{E1.1} that are unbounded at spatial infinity,~e.g. Tikhonov example and parasitic solutions, are excluded.

\medskip
\noindent{\bf Restricted Green tensors.}
A \emph{solenoidal} vector field $u=(u_1,\ldots,u_n)$ in $\R^n_+$ is a vector field satisfying
\begin{align}\label{solenoidal}
\div u=0,\ \ \ u_{n}|_{\Si}=0.
\end{align}
A tensor $\bar G_{ij}(x,y,t)$ is called a \emph{restricted Green tensor} if for any solenoidal $u_0$, the vector field $u_i(x,t) = \sum_{j=1}^n \int_{\R^n_+} \bar G_{ij}(x,y,t)(u_0)_j(y)dy$ is a solution of the Stokes system \eqref{E1.1}-\eqref{E1.2} for $f=0$ and $F=O$.
There are infinitely many restricted Green tensors (see Comment 2 of \cite[Theorem 1.2]{Green}).

In \cite[(3.12)]{MR0415097} and \cite{MR0460931, MR1992567}, Solonnikov obtained the following explicit formula of one particular restricted Green tensor, which we refer to as the \emph{restricted Green tensor of Solonnikov}.
\EQS{\label{E1.6}
\breve G_{ij}(x,y,t)& = \de_{ij}\Ga(x-y,t) + G_{ij}^*(x,y,t),
\\
G_{ij}^*(x,y,t) &=  -\de_{ij}\Ga(x-y^*,t)
 \\
&\quad
- 4(1-\de_{jn})\frac {\pd}{\pd x_j}
 \int_{\Si \times [0,x_n]} \frac {\pd}{\pd x_i} E(x-z) \Ga(z-y^*,t)\,dz,
}
where $y^*=(y',-y_n)$ for $y=(y',y_n)$, and $\Ga(x,t)$ and $E(x) $ are heat kernel and the fundamental solution of the Laplace equation in $\R^n$, respectively, given by
$$
\Gamma(x,t)=\left\{\begin{array}{ll}(4\pi t)^{-\frac{n}{2}}e^{\frac{-|x|^{2}}{4t}}&\ \text{ for }t>0,\\[3pt] 0&\ \text{ for }t\le0,\end{array}\right.\ \text{ and }\ E(x)=\left\{\begin{array}{ll}\frac1{n\,(n-2)|B_1|}\,\frac1{|x|^{n-2}}&\ \text{ for }n\ge3,\\[3pt] -\frac1{2\pi}\,\log|x|&\ \text{ if }n=2.\end{array}\right.
$$
Moreover, $G_{ij}^*$
satisfy the pointwise bound (\cite[(2.38), (2.32)]{MR1992567}) for $n \ge2$,
\EQ{
\label{Solonnikov.est}
|\pd_{x',y'}^l\pd_{x_n}^k \pd_{y_n}^q\pd_t^m G_{ij}^*(x,y,t)| \lesssim \frac{e^{-\frac{cy_n^2}t}}{t^{m+\frac{q}2}(|x^*-y|^2+t)^{\frac{l+n}2}(x_n^2+t)^{\frac{k}2}}.
}
In \cite[Theorem 1.2]{Green}, we showed that the action of the Green tensor on a solenoidal vector field yields the same results as when the restricted Green tensor of Solonnikov is applied.
In particular, if the initial data $u_0$ is solenoidal, the solution formula \eqref{E1.3} becomes
\EQS{
\label{eq-Stokes-sol-formula-divu0=0}
u_i(x,t)& = \sum_{j=1}^n\int_{\R^n_+} \breve G_{ij}(x,y,t)(u_0)_j(y)\,dy+
\sum_{j=1}^n\int_0^t \int_{\R^n_+} G_{ij}(x,y,t-s) f_j(y,s)\,dy\,ds\\
&\quad - \sum_{j,k=1}^n \int_0^t \int_{\R^n_+} \pd_{y_k} G_{ij}(x,y,t-s) F_{kj}(y,s)\, dyds.
}

\medskip
\noindent{\bf Pointwise estimates of the Green tensor.}
In our work \cite[(4.10)]{Green}, we obtained a formula of the Green tensor $G_{ij}(x,y,t)$ that is well-suited for pointwise estimation purposes.
Employing the derived formula, we established the following pointwise estimates for the Green tensor $G_{ij}$ and its derivatives (see \cite[Theorem 1.5]{Green}).

\bigskip
\mbox{}\hfill
\begin{minipage}{0.93\textwidth}
\emph{Let $n\ge 2$, $x,y\in\R^n_+$, $t>0$, $i,j=1,\ldots,n$, and $l,k,q,m \in \N_0$.
We have
\EQS{\label{eq_Green_estimate}
|\pd_{x',y'}^l \pd_{x_n}^k \pd_{y_n}^q \pd_t^m &G_{ij}(x,y,t)|\lec\frac1{(|x-y|^2+t)^{\frac{l+k+q+n}2+m}}
\\&+\frac{\LN_{ijkq}^{mn} + \LN_{jiqk}^{mn}}
{t^{m}(|x^*-y|^2+t)^{\frac{l+k-k_i+q-q_j+n}2}(x_n^2+t)^{\frac{k_i}2} (y_n^2+t)^{\frac{q_j}2} },
}
where $k_i=(k-\de_{in})_+$, $q_j=(q-\de_{jn})_+$, and
\[
\LN_{ijkq}^{mn}:= 1 + \de_{n2} \mu_{ik}^m \bkt{\log\bke{\nu_{ijkq}^m |x'-y'| + x_n+y_n+\sqrt t} - \log(\sqrt t)},
\]
in which $\mu_{ik}^m = 1-(\de_{k0} + \de_{k1}\de_{in})\de_{m0}$ and $\nu_{ijkq}^m = \de_{q0}\de_{jn} \de_{k(1+\de_{in})}\de_{m0} + \de_{m>0}$.
}
\end{minipage}

\bigskip

In view of the fact that the Green tensor $G_{ij}$ is associated with the no-slip boundary condition, $u|_\Si=0$, it has to satisfy $G_{ij}(x',0,y,t) = G_{ij}(x,y',0,t) = 0$.
In light of this property, we also derived the following boundary vanishing estimates in \cite[Theorem 1.6]{Green}.

\bigskip
\mbox{}\hfill
\begin{minipage}{0.93\textwidth}
\emph{Let $n\ge 2$, $x,y\in\R^n_+$, $t>0$, $i,j=1,\ldots,n$, and $l,q,m \in \N_0$.
Let $0\le\alpha\le1$.
If $k=0$, we have
\EQS{\label{eq_thm3_al_yn}
\left|\pd_{x',y'}^l\pd_{y_n}^q\pd_t^mG_{ij}(x,y,t)\right|
&\lesssim\frac{x_n^\al}{(|x-y|^2+t)^{\frac{l+q+n}2+m}(|x-y^*|^2+t)^{\frac\al2}}
\\
&\quad +\frac{x_n^\al\,\LN}
{t^{m+\frac{\al}2}(|x-y^*|^2+t)^{\frac{l+q-q_j+n}2}(y_n^2+t)^{\frac{q_j}2}},
}
with $\LN= {\textstyle\sum_{k=0}^1}( \LN_{ijkq}^{mn} + \LN_{jiqk}^{mn})(x,y,t)$.
A similar estimate for $\pd_{x',y'}^l\pd_{x_n}^k\pd_t^mG_{ij}(x,y,t)$ $($with no $\pd_{y_n})$ can be obtained by the symmetry of $G_{ij}(x,y,t)$.
}
\end{minipage}

\bigskip

When $l+q=1$, we have an improved version of \eqref{eq_thm3_al_yn}:
\EQS{\label{eq-improved-bdry-vanish-est-new}
|\pd_{y_p}G_{ij}(x,y,t)|
&\lesssim\frac{x_n^\al}{(|x-y|+\sqrt t)^{n+1}(|x-y^*|+\sqrt t)^\al} \\
&\quad + \frac{x_n^\al\wt{\rm Ln}}{\sqrt t^\al(|x-y^*|+\sqrt t)^{n+1-\de_{pn}(1-\de_{jn})}(y_n+\sqrt t)^{\de_{pn}(1-\de_{jn})}},
}
where $1\le p \le n$ and
\[
\wt{\rm Ln} = 1 + \de_{n2}(1-\de_{in}\de_{jn}) \bkt{\log\bke{(1-\de_{in}\de_{jn})(1-\de_{pn})|x'-y'|+x_n+y_n+\sqrt t} - \log(\sqrt t) }.
\]
Note that when $p=n$, there is no $|x'-y'|$ in $\wt{\rm Ln}$. It is the key to our proof of the endpoint case $n=a=2$ of Theorem \ref{thm-mild-bdry}. We will prove \eqref{eq-improved-bdry-vanish-est-new} in Subsection \ref{app-improved-bdry-vanish-est}.

\medskip

The pointwise estimates \eqref{eq_Green_estimate}-\eqref{eq_thm3_al_yn} hold significant potential for diverse applications.
In \cite{Green}, we employed the estimates to establish linear theories for the Stokes system in various function spaces, including $L^p$ (\cite[Lemma 9.1]{Green}),
$Y_a$ (\cite[Lemma 9.2]{Green}), $Z_a$ (\cite[Lemma 9.3]{Green}), and
$L^q_\uloc$ (\cite[Lemma 9.4]{Green}).
Here,
\EQS{\label{def-Ya-Za}
Y_a &= \bigg\{ f \in L^\infty_\loc(\R^n_+) \, \bigg|\, \norm{f}_{Y_a} = \sup_{x \in \R^n_+} |f(x)|\bka{x}^a<\infty
\bigg\},\ a\ge0,\\
Z_a &= \bigg\{ f \in L^\infty_\loc(\R^n_+) \, \bigg|\, \norm{f}_{Z_a} = \sup_{x \in \R^n_+} |f(x)|\bka{x_n}^a<\infty,
\bigg\},\ a\ge0,
}
and
\[
L^q_\uloc(\R^n_+) = \bket{f\in L^q_\loc(\R^n_+) : \sup_{x\in\R^n_+} \norm{f}_{L^q(B_1(x)\cap\R^n_+)}<\infty}.
\]
Utilizing these linear theories, we were able to prove the existence of mild solutions for the Navier-Stokes equations within their corresponding function spaces. See \cite[Theorems 1.7--1.10]{Green}.
Furthermore, in our previous work \cite{KLLT-TAMS2022}, the pointwise estimates were used to construct a finite energy Navier-Stokes flow with a localized H\"older continuous boundary flux while its gradient blows up everywhere on the boundary away from the flux at a finite time.

In this paper, we delve into further applications of the pointwise estimates \eqref{eq_Green_estimate}-\eqref{eq_thm3_al_yn} and the linear theories developed in \cite[Lemma 9.1--9.4]{Green}.
Firstly, we use the pointwise estimates to study the Stokes flow with $L^1$ initial data in Theorem \ref{thm0.1-GMS} that extends the results of \cite[Theorem 0.1]{GMS-MZ1999} and \cite[Theorem 2.1]{Maremonti-JMFM2011} to encompass non-solenoidal initial data.
Additionally, we utilize the pointwise estimates to establish linear theories for the Stokes system in a space of functions with mixed pointwise decay
\EQ{\label{eq-def-Yab}
Y_{a,b} = \bket{f\in L^\infty_{\loc}(\R^n_+) : \norm{f}_{Y_{a,b}} = \sup_{\R^n_+} |f(x)|\, \bka{x}^a \bka{x_n}^b <\infty },\quad a,b\ge0,
}
and in a space of functions with pointwise decay and boundary vanishing estimate
\EQ{\label{eq-def-Zaal}
Z_{a,\al} = \bket{f\in L^\infty_{\loc}(\R^n_+) : \norm{f}_{Z_{a,\al}} = \sup_{\R^n_+}  \frac{ |f(x)|\, \bka{x}^a \bka{x_n}^\al}{x_n^\al} <\infty },\quad a,\al\ge0.
}
With the aid of these linear theories, we proceed to construct mild solutions of the Navier-Stokes equations in $Y_{a,b}$ and $Z_{a,\al}$ in Theorems \ref{thm-mix-decay} and \ref{thm-mild-bdry}, respectively.

We also explore applications of the estimates established in \cite[Lemma 9.1--9.4]{Green} to various fluid-coupled systems.
The fluid-coupled systems discussed in this paper include the viscous resistive MHD equations, a system for the flow and magnetic field of MHD type (abbreviated F-M system of MHD type), and the nematic liquid crystal flows (see \S\ref{sec-MHD}--\ref{sec:NLCF-Dirichlet} for a brief literature review of these systems).
For the nematic liquid crystal flows, we consider both Neumann and Dirichlet boundary conditions for orientation fields.
Using the estimates in \cite[Lemma 9.1--9.4]{Green}, we are able to construct mild solutions of these fluid-coupled systems in the half within the $L^q$, $Y_a$, and $L^q_\uloc$ frameworks.

Finally, we conclude this section by giving the structure of this paper.
In Section \ref{sec-results}, we present the main results of this paper and some background on the corresponding topics.
In Section \ref{sec-prelim}, we first recall some preliminary integral estimates and develop a new Lemma \ref{lem-heat-int-decay-est}.
We then establish linear theories, Lemmas \ref{lem-heat-est-Lq}\,--\,\ref{lem-heat-est-Lq-uloc}, for the Dirichlet and the Neumann problems of heat equation in the half space within $L^q$, $Y_a$, and $Z_a$
which are essential for constructing mild solutions of the MHD equations and nematic liquid crystal flow.

Moving on to Section \ref{sec:app-NSE}, we first prove the estimates of the Stokes flow with $L^1$ initial data, as stated in Theorem \ref{thm0.1-GMS}.
Furthermore, We construct mild solutions of Navier--Stokes equations in $Y_{a,b}$ and $Z_{a,\al}$, proving Theorems \ref{thm-mix-decay} and \ref{thm-mild-bdry}.
In Section \ref{sec:mild-MHD-Ya}, we focus on constructing mild solutions of the viscous resistive MHD equations and the F-M system of MHD type in the half space in $L^q$, $Y_a$, $Z_a$, and $L^q_\uloc$.
These give Theorems \ref{mild-MHD-Lq}, \ref{mild-MHD-Ya}\,--\, 
\ref{mild-MHD-Lq-uloc}, and 
\ref{mild-mMHD-Lq}\,--\,\ref{mild-mMHD-Lq-uloc}. Section \ref{sec:mild-NLCF} is dedicated to the construction of mild solutions of nematic liquid crystal flows in the half space, considering either Neumann or Dirichlet boundary orientation fields.
We construct mild solutions for each problem in $L^q$, $Y_a$, and $L^q_\uloc$.
These results are summarized in \ref{mild-NLCF-N-Ya}\,--\,\ref{mild-NLCF-N-Lq} and \ref{mild-NLCF-D-Ya}\,--\,\ref{mild-NLCF-D-Lq-uloc}.


\section{Main results}\label{sec-results}
In this section we present our main results.

\subsection{Estimates of the Stokes flow with $L^1$ initial data}

The first application of the pointwise estimate of (unrestricted) Green tensor is $L^1\to L^q$ estimates 
for solutions for non-solenoidal initial data $u_0\in L^1(\R^n_+)$ of the Stokes system \eqref{E1.1}-\eqref{E1.2} in $\R^n_+$.

In \cite{Maremonti-JMFM2011}, Maremonti considered the Navier-Stokes equations in bounded and exterior domains and constructed, for any possibly non solenoidal $L^1$ initial data, a $L^q$ solutions with $q>1$.
$L^q$ decay estimates of the solution are given in \cite[Theorem 2.1]{Maremonti-JMFM2011}.
In its Introduction, the author mentioned that the result also holds in the half-space case by a suitable modification of the proof.
Below in Section \ref{sec-L1-grad-est}, we give a proof of the following half-space version of \cite[Theorem 2.1]{Maremonti-JMFM2011} using the Green tensor estimates. We remark that our result
also include the case $q=1$ for the temporal decay for $\nabla u$. More precisely, we obtain

\begin{thm}\label{thm0.1-GMS}
Let $n\ge2$.
Then for any $u_0\in L^1(\R^n_+)$, not necessarily solenoidal, there exists a unique solution represented by \eqref{E1.3} of the Stokes system \eqref{E1.1}-\eqref{E1.2} in $\R^n_+$ with $f=0$ and $F=0$ such that the following estimates are satisfied for any $t>0$;
\begin{equation}\label{thm2.1-10}
\norm{u(t)}_{L^q(\R^n_+)} \le Ct^{-\frac{n}2\bke{1-\frac1q}} \norm{u_0}_{L^1(\R^n_+)}, \qquad q\in (1, \infty],
\end{equation}
\begin{equation}\label{thm2.1-20}
\norm{\nb u(t)}_{L^q(\R^n_+)} \le Ct^{-\frac12-\frac{n}2\bke{1-\frac1q}} \norm{u_0}_{L^1(\R^n_+)},\qquad q\in [1, \infty],
\end{equation}
\begin{equation}\label{thm2.1-30}
\norm{\pd_t u(t)}_{L^q(\R^n_+)} \le Ct^{-1-\frac{n}2\bke{1-\frac1q}} \norm{u_0}_{L^1(\R^n_+)},\qquad q\in (1, \infty],
\end{equation}
where the constant $C$ is independent of $u_0$.
\end{thm}

The proof of Theorem \ref{thm0.1-GMS} is in Section \ref{sec-L1-grad-est}.

In case that $u_0 \in L^1_{\sigma}(\R^n_+)$,
the estimate \eqref{thm2.1-20} with $q=1$ was proved by Giga, Matsui, and Shimizu  \cite{GMS-MZ1999} through the rewriting of Ukai's formula using the pseudo-differential operator $\La$ whose symbol equals $|\xi'|$, $(\xi',\xi_n)=\xi\in\R^n$, followed by deriving the estimate
$$
\|\Lambda u(t)\|_{\mathcal{H}^{1}(\R^n)}\lec t^{-\frac{1}{2}}\|u_0\|_{L^1(\R^n)},
$$
where $\mathcal H^1(\R^n)$ is the Hardy space in $\R^n$.
A couple of years later, Desch, Hieber, and Pr\"{u}ss \cite{DHP} gave an alternative proof of the estimate 
of \cite{GMS-MZ1999}.
In particular, their derivation of \eqref{thm2.1-20} is based on a detailed analysis of the resolvent problem for the Stokes system using Fourier transform.

Our establishment of Theorem \ref{thm0.1-GMS} not only extends the results of \cite{DHP, GMS-MZ1999} to non-solenoidal initial data $u_0 \in L^1(\R^n_+)$ but also provide a simple and direct proof to the $L^1-L^q$ estimates using the Green tensor estimates.

On the other hand,
Desch, Hieber, and Pr\"{u}ss \cite{DHP} provided an example showing that the estimate \eqref{thm2.1-10} with $q=1$ isn't valid, unlike the case of whole space $\R^n$.
Nevertheless, it is known that some weighted $L^1-L^1$ estimates are available under some assumptions
on $u_0\in L^1(\R^n_+)$.
For example, Bae proved in \cite[Corollary 3.6]{B2006}, that if 
$u_0$ is divergence free and satisfies $\int_{\R^{n-1}} u_0(x', x_n ) dx'=0$ and $|x'|u_0(x', x_n)\in L^1(\R^n_+)$, then
\[
\norm{u(t)}_{L^1(\R^n_+)} \le Ct^{-\frac{1}{2}} \norm{|x'| u_0}_{L^1(\R^n_+)},
\]
In this directions, there have been many diverse and improved results (see  e.g  \cite{H2014}, \cite{H2018},  \cite{B2022} and  references therein) but we are not pursuing such directions in this paper.

\subsection{Mild solutions of Navier-Stokes equations in the half space}

For applications to Navier-Stokes equations,
\EQ{\label{NS}
\pd_t u-\Delta u+ \nabla \pi=-u \cdot \nb u ,\quad
 \div u=0, \quad \mbox{in}\ \ \R_{+}^{n}\times (0,\infty),
}
with zero boundary condition.
Suppose the initial data $u_0$ is solenoidal, a solution of \eqref{NS} is called a \emph{mild solution} if it satisfies \eqref{eq-Stokes-sol-formula-divu0=0} with suitable estimates for
$f + \na\cdot F = -u \cdot \nb u$, or equivalently,
\EQS{
\label{eq-mild-NS}
u_i(x,t) = \sum_{j=1}^n\int_{\R^n_+} \breve G_{ij}(x,y,t)(u_0)_j(y)\,dy
 - \sum_{j,k=1}^n \int_0^t \int_{\R^n_+} \pd_{y_k} G_{ij}(x,y,t-s) (u_ku_j)(y,s)\, dyds
}
since $u\cdot\nb u = \nb\cdot(u\otimes u)$ for solenoidal $u$, where $(u\otimes v)_{jk}:= u_jv_k$.

\subsubsection{Navier-Stokes flows with mixed pointwise decay}
For $a,b\ge0$, recall $Y_{a,b}$ defined in \eqref{eq-def-Yab}.
Note that $\bka{x}^a \bka{x_n}^b\not\approx \bka{x'}^a \bka{x_n}^{a+b}$. In particular, they are not comparable when $|x'|=x_n$.

\begin{thm}\thlabel{thm-mix-decay}
Let $n\ge2$, $a>0$, $0<b\le 1$, $a+b\le n$ and $(a,b)\neq (n-1,1)$. For any vector field $u_0\in Y_{a,b}$ with $\div u_0=0$ and $u_{0,n}|_\Si=0$, there is a mild solution $u\in L^\infty(0,T;Y_{a,b})$ of \eqref{NS} with initial data $u_0$ for some time interval $(0,T)$.
Moreover, the mild solution is unique in the class $L^\infty(\R^n_+\times (0,T))$.
\end{thm}

The proof of \thref{thm-mix-decay} is in Subsection \ref{sec-mix-nse}. The main task is to establish the linear estimates \thref{lem-mix-decay} for the initial data and the Duhamel term.

The borderline cases $b=0$ or $a=0$ were previously known: The case $b=0<a\le n$ is proved
 in \cite[Theorem 1]{CJ-NA2017} and \cite[Theorem 2.1]{CM2004}, and reproved
 in \cite[Theorem 1.8]{Green}. The case $a=0<b$ is proved in \cite[Theorem 1.9]{Green}.  When $a=b=0$ and $Y_{a,b} = L^\infty$, it is proved in
 \cite{BJ2012}, and reproved in \cite[Theorem 1.7]{Green} assuming further $n\ge3$. The
additional assumption in \cite[Theorem 1.7]{Green} that $u_0$ is in the $L^\infty$-closure of $C^1_{c,\si}(\overline{\R^n_+})$ ensures strong continuity at time zero, and is unnecessary for the existence of $L^\infty$-mild solutions. All cases with $a,b>0$ in Theorem \ref{thm-mix-decay} are new. See Remark \ref{rem4.3} for the restrictions on $a,b$.


\begin{remark}[Strong solution]
To upgrade the mild solution to a strong solution by applying the Serrin's regularity result \cite{Serrin-ARMA1962} as in \cite[Remark 2]{CJ-NA2017}, gradient estimates of the solution is needed.
Moreover, a classical solution has to be strongly continuous at time zero.
Therefore, in non-decaying cases, i.e.~$a=0$ or \cite[Theorem 1.9]{Green}, further assumption, e.g.~$u_0$ being in the $L^\infty$-closure of $C^1_{c,\si}(\overline{\R^n_+})$, may be required to ensure that the mild solution solves the Navier-Stokes equations in the classical sense.

\end{remark}

\subsubsection{Navier-Stokes flows with pointwise decay and boundary vanishing estimate}
For $a,\al\ge0$, 
recall $Z_{a,\al}$ defined in \eqref{eq-def-Zaal}.

\begin{thm}\thlabel{thm-mild-bdry}
Let $n\ge2$, $\al\in[0,1]$, and $\al<a\le n$.
For any vector field $u_0\in Z_{a,\al}$ with $\div u_0=0$ and $u_{0,n}|_\Si=0$, there is a mild solution $u\in L^\infty(0,T;Z_{a,\al})$ of \eqref{NS} with initial data $u_0$ for some time interval $(0,T)$.
Moreover, the mild solution is unique in the class $L^\infty(\R^n_+\times (0,T))$.
\end{thm}

The proof of \thref{thm-mild-bdry} is in Section \ref{sec-bdry-nse}. The case $a=n=2$ is the most delicate.

\subsection{Mild solutions of MHD equations and F-M system of MHD type}

\subsubsection{The viscous resistive MHD equations}\label{sec-MHD}

In a magnetofluid, the motion of an electrically conducted fluid is governed by the system of magnetohydrodynamic (MHD) equations which features the interaction between the fluid flow and the magnetic field.
The dynamics of the fluid particle and the magnetic field is described by the coupling between the Navier-Stokes equations of fluid dynamics and the Maxwell's equations of electromagnetism (see e.g.~\cite{Davidson-book2001}).
Here, we consider the system of viscous resistive incompressible MHD equations in the half space:
\EQ{\label{eq-MHD}
\setlength\arraycolsep{1.5pt}\def\arraystretch{1.2}
\left.\begin{array}{rl}
\pd_tu - \De u + \na \pi&= - u\cdot \na u + b\cdot\na b \\
\pd_tb - \De b &= - u\cdot\na b + b\cdot\na u \\
\div u &= \div b=0\
\end{array}\right\} \text{ in }\R^n_+\times(0,\infty).
}
Here  $u=(u_{1},\ldots,u_{n})$ is the fluid velocity, $\pi$ is the pressure, and
$b=(b_{1},\ldots,b_{n})$ is the magnetic field.
The system is coupled with initial conditions
\EQ{\label{eq-MHD-IC}
(u,b)(\cdot,0) = (u_0,b_0),\qquad \div u_0 = \div b_0 = 0,\quad (u_0)_n|_\Si = (b_0)_n|_\Si = 0,
}
\emph{no-slip} boundary condition for the fluid velocity $u$
\EQ{\label{eq-MHD-uBC}
u(x',0,t)=0\ \text{ on }\ \Si\times(0,\infty),
}
and \emph{slip} boundary condition for the magnetic field $b$
\EQ{\label{eq-MHD-bBC}
b_n(x',0,t)=0,\quad
\pd_nb_i(x',0,t)=0,\ i=1,\ldots,n-1,\quad \text{ on }\ \Si\times(0,\infty).
}
Let $\nu$ denote the outer normal on $\Si$.
In the case of dimension three, the boundary condition \eqref{eq-MHD-bBC} is equivalent to
\EQ{\label{eq-MHD-bBC-explicit}
b(x',0,t)\cdot\nu=0,\quad
\curl b(x',0,t)\times \nu=0\quad \text{ on }\ \Si\times(0,\infty),
}
corresponding to a perfectly conducting wall (see \cite[(5.23)]{DL-ARMA1972} and \cite[(1.3)]{ST-CPAM1983}). It is a special case of the \emph{Navier boundary condition} with  \emph{infinite} slip length.
Note that the boundary condition \eqref{eq-MHD-bBC} preserves the solenoidality of $b$ (see Lemma \ref{lem-b-solenoidal}).

There have been extensive studies on the MHD equations.
In the three-dimensional case, weak solutions with the boundary conditions \eqref{eq-MHD-uBC} and \eqref{eq-MHD-bBC-explicit} were constructed by Ladyzhenskaya and Solonnikov \cite{LS-1960}.
Their result was generalized by Duvaut and Lions in \cite{DL-ARMA1972} to the case of dimension $n$, $n=2,3$,
where they constructed a global weak solution with finite energy and a local strong solution, and showed that weak solutions become regular when $n=2$.
Sermange and Temam \cite{ST-CPAM1983} showed that such solution becomes regular if the solution is additionally in $L^\infty(0,T;H^1)$.
Over the decades, there have been many significant results on the problem of regularity of weak solutions of MHD equations.
We only present a brief list of results in the half space that are related to our consideration.
In \cite{KK-JDE2012}, the first author and J.-M.~Kim gave two regularity criteria of the MHD equations in the three-dimensional half space.
Both regularity criteria use Serrin's type conditions -- one is on $u$ alone, and the other one is on tangential components of $u$ and normal component of $b$.
In Kim \cite{Kim-AdvMathPhys2022}, a regularity criterion for weak solutions is given in terms of the pressure.
Concerning suitable weak solutions in $\R^3_+$, there are several boundary regularity criteria (see \cite{VS-POMI2010, KK-JFA2014, Kim-ActaMathSci2017, NY-JMAA2021}).

Attention here is focused on mild solutions of the viscous resistive incompressible MHD equations in the half space, \eqref{eq-MHD}-\eqref{eq-MHD-bBC}.
In order to give a proper definition of mild solutions, we first give the following solution formula for sufficiently good solutions.

\begin{prop}\thlabel{prop-MHD-sol-formula}
Let $(u,b)$ be a sufficiently smooth and bounded solution of the viscous resistive MHD equations \eqref{eq-MHD}-\eqref{eq-MHD-bBC} in $\R^n_+$, $n \ge 2$. Then we have the following solution formula:
\EQS{\label{u-mild}
u_i(x,t) &= \sum_{j=1}^n \int_{\R^n_+} \breve G_{ij}(x,y,t) (u_0)_j(y)\, dy \\
&\quad + \sum_{j,k=1}^n \int_0^t \int_{\R^n_+} \pd_{y_k}G_{ij}(x,y,t-s)\bke{u_ku_j - b_kb_j}(y,s)\, dyds,\ i=1,\ldots,n,
}
and
\EQS{\label{b-mild_1}
b_i(x,t) &= \int_{\R^n_+} G^N(x,y,t) (b_0)_i(y)\, dy \\
&\quad + \sum_{k=1}^n \int_0^t \int_{\R^n_+} \pd_{y_k} G^N(x,y,t-s)\bke{u_kb_i - b_ku_i}(y,s)\, dyds,\ i=1,\ldots,n-1,
}
\EQS{\label{b-mild_n}
b_n(x,t) = \int_{\R^n_+} G^D(x,y,t) (b_0)_n(y)\, dy  + \sum_{k=1}^n \int_0^t \!\!\int_{\R^n_+} \pd_{y_k}G^D(x,y,t-s)\bke{u_kb_n-b_ku_n}(y,s)\, dyds,
}
where $G^N(x,y,t)$ and $G^D(x,y,t)$ are the Green functions of the heat equation in the half space for Neumann and Dirichlet boundary conditions, respectively, defined by
 \EQ{\label{eq-GN-def}
 G^N(x,y,t) = \Ga(x-y,t)+\Ga(x-y^*,t),
 }
 and
 \EQ{\label{eq-GD-def}
 G^D(x,y,t) = \Ga(x-y,t)-\Ga(x-y^*,t).
 }
\end{prop}

The proof of \thref{prop-MHD-sol-formula} is in Section \ref{sec-MHD-sol-formula}.

For solenoidal initial data $u_0$ and $b_0$, a solution of the viscous resistive MHD equations \eqref{eq-MHD}-\eqref{eq-MHD-bBC} is called a \emph{mild solution} if it satisfies \eqref{u-mild}-\eqref{b-mild_n} with suitable estimates.
Our goal for the viscous resistive MHD equations \eqref{eq-MHD}-\eqref{eq-MHD-bBC} is to construct mild solutions in various function spaces.

For $1\le q\le\infty$, denote
\[
L^q_\si(\R^n_+) = \bket{f\in L^q(\R^n_+;\R^n):\div f = 0,\ f_n(x',0)= 0 }.
\]

\begin{thm}\thlabel{mild-MHD-Lq}
Let $2\le n\le q\le\infty$ and $u_0, b_0\in L^q_\si(\R^n_+)$.
Then there exist $T = T(n,q,u_0,b_0)>0$ and a unique mild solution $(u,b)\in L^\infty(0,T; L^q)$ of the viscous resistive MHD equations \eqref{eq-MHD}-\eqref{eq-MHD-bBC} in the class
\EQN{
\sup_{0<t<T} &\bkt{\norm{u(t)}_{L^q} + \norm{b(t)}_{L^q} + t^{\frac{n}{2q}} \bke{ \norm{u(t)}_{L^\infty} + \norm{b(t)}_{L^\infty}} + t^{1/2} \bke{ \norm{\nb u(t)}_{L^q} + \norm{\nb b(t)}_{L^q} } }\\
&\qquad\qquad\qquad\qquad\qquad\qquad\qquad\qquad\qquad\qquad\qquad\qquad\qquad \le C_* \bke{ \norm{u_0}_{L^q} + \norm{b_0}_{L^q} }.
}
We can take $T=T(n,q,\norm{u_0}_{L^q},\norm{b_0}_{L^q})$ if $n<q\le\infty$.

Moreover, the solution $(u,b)$ is continuous at $t=0$, i.e. $(u,b)\in C([0,T];L^q)$, provided $q<\infty$ or $u_0$ and $b_0$ are in the $L^\infty$-closure of $C^1_{c,\si}(\overline{\R^n_+})$ if $q=\infty$.
\end{thm}

\begin{remark}
(i) Mild solutions of MHD has been constructed in \cite{TWZ-CMS2020}  in the whole space case, and in
\cite{MR4109418} for stochastic MHD in bounded domains.
Theorem \ref{mild-MHD-Lq} seems to be the first mild solution result for the half space.

(ii)
If $u_0,b_0\in L^q_\si(\R^n_+)\cap L^2_\si(\R^n_+)$, then the
unique mild solution in Theorem \ref{mild-MHD-Lq} of the
system \eqref{eq-MHD}-\eqref{eq-MHD-bBC} is also a weak solution as in 
\cite{DL-ARMA1972} and \cite[Theorem 3.1]{ST-CPAM1983}, 
by the same argument of Kato \cite[Theorem 3]{Kato1984}.
This remark also applies to the whole space case.
\end{remark}

Recall, for $a\ge0$, $Y_a$ and $Z_a$ are defined in \eqref{def-Ya-Za}.

\begin{thm}\thlabel{mild-MHD-Ya}
Let $n\ge2$ and $0\le a\le n$. For any pair of vector fields $u_0,b_0\in Y_a$ with $\div u_0 = \div b_0 = 0$ and $(u_0)_n|_\Si=(b_0)_n|_\Si=0$, there is a
 mild solution $(u,b)\in L^\infty(0,T; Y_a^2)$ of the viscous resistive MHD equations \eqref{eq-MHD}-\eqref{eq-MHD-bBC} for some time interval $(0,T)$. Moreover, the mild solution is unique in the class $L^\infty(\R^n_+\times(0,T))$.
\end{thm}

\begin{thm}\thlabel{mild-MHD-Za}
Let $n\ge2$ and $0\le a\le 1$. For any pair of vector fields $u_0,b_0\in Z_a$ with $\div u_0 = \div b_0 = 0$ and $(u_0)_n|_\Si=(b_0)_n|_\Si=0$, there is a
mild solution $(u,b)\in L^\infty(0,T; Z_a^2)$ of the viscous resistive MHD equations \eqref{eq-MHD}-\eqref{eq-MHD-bBC} for some time interval $(0,T)$. Moreover, the mild solution is unique in the class $L^\infty(\R^n_+\times(0,T))$.
\end{thm}

For $1\le q\le\infty$, let
\EQN{
L^q_\uloc(\R^n_+) &= \bket{u\in L^q_\loc(\R^n_+) : \sup_{x\in\R^n_+} \norm{u}_{L^q(B_1(x)\cap\R^n_+)}<\infty},\\
L^q_{\uloc,\si}(\R^n_+) &= \bket{u\in L^q_\uloc(\R^n_+;\R^n) : \div u = 0,\ u_n|_\Si = 0}.
}

\begin{thm}\thlabel{mild-MHD-Lq-uloc}
Let $2\le n\le q\le\infty$ and $u_0,b_0\in L^q_{\uloc,\si}(\R^n_+)$.

(a) If $n<q\le\infty$
, then there exist $T=T(n,q,\norm{u_0}_{L^q_\uloc},\norm{b_0}_{L^q_\uloc})>0$ and a unique mild solution $(u,b)$ of the viscous resistive MHD equations \eqref{eq-MHD}-\eqref{eq-MHD-bBC} with
\EQS{\label{eq-mhd-uloc-class}
&u(t), b(t) \in L^\infty(0,T; L^q_{\uloc,\si})\cap C((0,T); W^{1,q}_{\uloc,0}(\R^n_+)^n),\\
\sup_{0<t<T} &\bigg[ \norm{u(t)}_{L^q_\uloc} + \norm{b(t)}_{L^q_\uloc} + t^{\frac{n}{2q}} \bke{\norm{u(t)}_{L^\infty} + \norm{b(t)}_{L^\infty} }
\\
&\hspace{10mm} + t^{\frac12} \bke{\norm{\nb u(t)}_{L^q_\uloc} + \norm{\nb b(t)}_{L^q_\uloc} } \bigg]
\le C_* \bke{\norm{u_0}_{L^q_\uloc} + \norm{b_0}_{L^q_\uloc} }.
}

(b) If $n=q$, for any $0<T<\infty$, there are $\ep(T), C_*(T)>0$ such that if $\norm{u_0}_{L^n_\uloc} + \norm{b_0}_{L^n_\uloc}\le\ep(T)$, then there is a unique mild solution $(u,b)$ of the viscous resistive MHD equations \eqref{eq-MHD}-\eqref{eq-MHD-bBC} in the class \eqref{eq-mhd-uloc-class}.
\end{thm}

\begin{remark}
In the critical case $n=q$, one can use a density argument to establish the local-in-time existence of the mild solution without the smallness of the initial data if $u_0,b_0\in E^n = \overline{C^\infty_c}^{L^n_\uloc}$.
\end{remark}

\begin{remark}
One can easily establish the same existence and uniqueness as in Theorems \ref{mild-MHD-Lq} to \ref{mild-MHD-Lq-uloc} for the incompressible viscoelastic Navier-Stokes equations with damping (see \cite{Hynd-SIMA2013,Kim-AML2017,LLW-SIMA2017,Lai-JMFM2019,TWZ-CMS2020}) in the half space.
That is, using a similar fixed point argument, we can construct mild solutions of the following system in $L^q$, $Y_a$, $Z_a$, and $L^q_\uloc$.
\EQ{\label{eq-vNSEd}
\setlength\arraycolsep{1.5pt}\def\arraystretch{1.2}
\left.\begin{array}{rl}
\pd_tu - \De u + \na \pi&= - u\cdot \na u + \nb\cdot({\bf F}{\bf F}^\top)\\
\pd_t{\bf F} - \De{\bf F} &= - u\cdot\na{\bf F} + (\nb u){\bf F} \\
\div u &= \div{\bf F}=0\
\end{array}\right\} \text{ in }\R^n_+\times(0,\infty).
}
Here  $u:\R^n_+\times(0,\infty)\to\R^n$ is the velocity, $\pi:\R^n_+\times(0,\infty)\to\R$ is the pressure, and
${\bf F} = (F_{ij})_{i,j=1}^n:\R^n_+\times(0,\infty)\to\R^{n\times n}$ represents the local deformation tensor.
The system is coupled with initial conditions
\EQ{\label{eq-vNSEd-IC}
(u,{\bf F})(\cdot,0) = (u_0,{\bf F}_0),\quad \div u_0 = 0,\quad \div {\bf F}_0 = 0,\quad (u_0)_n|_\Si = (F_0)_{nj}|_\Si = 0,\ j=1,\ldots,n,
}
no-slip boundary condition for the fluid velocity $u$
\EQ{\label{eq-vNSEd-uBC}
u(x',0,t)=0\ \text{ on }\ \Si\times(0,\infty),
}
and slip boundary condition for the local deformation tensor ${\bf F}$
\EQS{\label{eq-vNSEd-FBC}
F_{nj}(x',0,t)&=0,\ \qquad\qquad\qquad\quad\ \,j=1,\ldots,n,\quad \text{ on }\ \Si\times(0,\infty),\\
\pd_nF_{ij}(x',0,t)&=0,\ i=1,\ldots,n-1,\ j=1,\ldots,n,\quad \text{ on }\ \Si\times(0,\infty).
}

\end{remark}

\subsubsection{A system for the flow and the magnetic field of MHD type}\label{sec-FM-system}

As pointed out in the introduction of \cite{DXX-SciChinaMath2022}, there are various types of boundary conditions for the magnetic field $b$, and a no-slip boundary condition may lead to an overdetermined problem. In fact, if $b$ is assumed to satisfy the homogeneous Dirichlet boundary condition, a potential is needed in the magnetic equation (see \cite{Lassner-ARMA1967}) because the divergence-free condition $\div b=0$ is not preserved by the heat evolution under no-slip boundary condition.
In this case, we consider the following system in the half space
\EQ{\label{eq-mMHD}
\setlength\arraycolsep{1.5pt}\def\arraystretch{1.2}
\left.\begin{array}{rl}
\pd_tu - \De u + \na \pi&= - u\cdot \na u + b\cdot\na b \\
\pd_tb - \De b + \na \phi&= - u\cdot\na b + b\cdot\na u \\
\div u &=\div b=0\
\end{array}\right\} \text{ in }\R^n_+\times(0,\infty),
}
Here  $u=(u_{1},\ldots,u_{n})$ is the fluid velocity, $\pi$ is the fluid pressure,
$b=(b_{1},\ldots,b_{n})$ is the magnetic field, $\phi$ is related to the motion of heavy ions.
The system is coupled with initial conditions
\EQ{\label{eq-mMHD-IC}
(u,b)(\cdot,0) = (u_0,b_0),\qquad \div u_0 = \div b_0 = 0,\quad (u_0)_n|_\Si = (b_0)_n|_\Si = 0,
}
and no-slip boundary conditions for both velocity and magnetic fields
\EQ{\label{eq-mMHD-BC}
u(x',0,t)= b(x',0,t) = 0\ \text{ on }\ \Si\times(0,\infty).
}

The system \eqref{eq-mMHD}-\eqref{eq-mMHD-BC} describes the motion of an incompressible electrical conducting fluid in the presence of a magnetic field when Maxwell's displacement currents are considered \cite{Batchelor-book1967,Shercliff-book1965}, i.e.~there is free motion of heavy ions, not directly due to the electric field (see \cite{Schluter-plasma1950,Schluter-plasma1951,Pikelner-book1966}).
The system \eqref{eq-mMHD}-\eqref{eq-mMHD-BC} also arises from geophysics that governs the flow and the magnetic field within the earth \cite{Hide-geophyics1971}.

Some authors called the system \eqref{eq-mMHD}-\eqref{eq-mMHD-BC} \emph{the evolution equations of MHD type} \cite{RB-Proyec1994,DR-Proyec1997,RB-JAustMathSocB1997,Zhao-MMAS2003}, while others named it as \emph{the equations for the flow and the magnetic field (abbreviated F-M equations)} \cite{QKC-AML2000}.
Note that the system is different from the incompressible MHD-type system considered in \cite{LZ-CPAM2014}.
To avoid confusion, we shall call the system \eqref{eq-mMHD}-\eqref{eq-mMHD-BC} the flow and the magnetic field of magnetohydrodynamic type and abbreviated it as \textit{\textbf{F-M system of MHD type}}.

Many significant contributions have been made for the F-M system of MHD type \eqref{eq-mMHD}-\eqref{eq-mMHD-BC} concerning the existence of weak and strong solutions, uniqueness and regularity criteria; see for instance \cite{Lassner-ARMA1967,RB-Proyec1994,QSW-JMAA1994,DR-Proyec1997,RB-JAustMathSocB1997,CWQ-IJMMS1998,Zhao-MMAS2003} and the references therein.

For solenoidal initial data $u_0$ and $b_0$, a solution of the F-M system of MHD type \eqref{eq-mMHD}-\eqref{eq-mMHD-BC} is called a \emph{mild solution} if it satisfies the following identities with suitable estimates.
\EQS{\label{u-mild-m}
u_i(x,t) &= \sum_{j=1}^n \int_{\R^n_+} \breve G_{ij}(x,y,t) (u_0)_j(y)\, dy \\
&\quad + \sum_{j,k=1}^n \int_0^t \int_{\R^n_+} \pd_{y_k}G_{ij}(x,y,t-s)\bke{u_ku_j - b_kb_j}(y,s)\, dyds,\ i=1,\ldots,n,
}
\EQS{\label{b-mild-m}
b_i(x,t) &= \sum_{j=1}^n \int_{\R^n_+} \breve G_{ij}(x,y,t) (b_0)_j(y)\, dy \\
&\quad + \sum_{j,k=1}^n \int_0^t \int_{\R^n_+} \pd_{y_k}G_{ij}(x,y,t-s)\bke{u_kb_j - b_ku_j}(y,s)\, dyds,\ i=1,\ldots,n.
}

\begin{thm}\thlabel{mild-mMHD-Lq}
Let $2\le n\le q\le\infty$ and $u_0, b_0\in L^q_\si(\R^n_+)$.
Then there exist $T = T(n,q,u_0,b_0)>0$ and a unique mild solution $(u,b)\in L^\infty(0,T; L^q)$ of the F-M system of MHD type \eqref{eq-mMHD}-\eqref{eq-mMHD-BC} in the class
\EQN{
\sup_{0<t<T} &\bkt{\norm{u(t)}_{L^q} + \norm{b(t)}_{L^q} + t^{\frac{n}{2q}} \bke{ \norm{u(t)}_{L^\infty} + \norm{b(t)}_{L^\infty}} + t^{1/2} \bke{ \norm{\nb u(t)}_{L^q} + \norm{\nb b(t)}_{L^q} } }\\
&\qquad\qquad\qquad\qquad\qquad\qquad\qquad\qquad\qquad\qquad\qquad\qquad\qquad \le C_* \bke{ \norm{u_0}_{L^q} + \norm{b_0}_{L^q} }.
}
We can take $T=T(n,q,\norm{u_0}_{L^q},\norm{b_0}_{L^q})$ if $n<q\le\infty$.

Moreover, the solution $(u,b)$ is continuous at $t=0$, i.e. $(u,b)\in C([0,T];L^q)$, provided $q<\infty$ or $u_0$ and $b_0$ are in the $L^\infty$-closure of $C^1_{c,\si}(\overline{\R^n_+})$ if $q=\infty$.
\end{thm}

\begin{thm}\thlabel{mild-mMHD-Ya}
Let $n\ge2$ and $0<a\le n$. For any pair of vector fields $u_0,b_0\in Y_a$ with $\div u_0 = \div b_0 = 0$ and $(u_0)_n|_\Si=(b_0)_n|_\Si=0$, there is a 
mild solution $(u,b)\in L^\infty(0,T; Y_a^2)$ of the F-M system of MHD type \eqref{eq-mMHD}-\eqref{eq-mMHD-BC} for some time interval $(0,T)$. Moreover, the mild solution is unique in the class $L^\infty(\R^n_+\times(0,T))$.
\end{thm}

\begin{thm}\thlabel{mild-mMHD-Za}
Let $n\ge2$ and $0<a\le 1$. For any pair of vector fields $u_0,b_0\in Z_a$ with $\div u_0 = \div b_0 = 0$ and $(u_0)_n|_\Si=(b_0)_n|_\Si=0$, there is a 
mild solution $(u,b)\in L^\infty(0,T; Z_a^2)$ of the F-M system of MHD type \eqref{eq-mMHD}-\eqref{eq-mMHD-BC} for some time interval $(0,T)$. Moreover, the mild solution is unique in the class $L^\infty(\R^n_+\times(0,T))$.
\end{thm}

Theorems \ref{mild-mMHD-Ya} and \ref{mild-mMHD-Za} are \cite[Theorem 2.1.11]{lai-thesis2021}, but $\na q$ is missing in \cite[(2.31)-(2.32)]{lai-thesis2021}.
The proof of Theorems \ref{mild-mMHD-Ya} and \ref{mild-mMHD-Za} uses \cite[Lemmas 9.2 and 9.3]{Green} and a fixed point argument. It can be found in \cite[$\S$2.9.3]{lai-thesis2021}.

\begin{thm}\thlabel{mild-mMHD-Lq-uloc}
Let $2\le n\le q\le\infty$ and $u_0,b_0\in L^q_{\uloc,\si}(\R^n_+)$.

(a) If $n<q\le\infty$
, then there exist $T=T(n,q,\norm{u_0}_{L^q_\uloc},\norm{b_0}_{L^q_\uloc})>0$ and a unique mild solution $(u,b)$ of the F-M system of MHD type \eqref{eq-mMHD}-\eqref{eq-mMHD-BC} with
\EQS{\label{eq-mMHD-uloc-class}
&u(t), b(t) \in L^\infty(0,T; L^q_{\uloc,\si})\cap C((0,T); W^{1,q}_{\uloc,0}(\R^n_+)^n\cap BUC_\si(\R^n_+)),\\
\sup_{0<t<T} &\Big[ \norm{u(t)}_{L^q_\uloc} + \norm{b(t)}_{L^q_\uloc} + t^{\frac{n}{2q}} \bke{\norm{u(t)}_{L^\infty} + \norm{b(t)}_{L^\infty} } 
\\
&\quad + t^{\frac12} \bke{\norm{\nb u(t)}_{L^q_\uloc} + \norm{\nb b(t)}_{L^q_\uloc} } \Big]
\le C_* \bke{\norm{u_0}_{L^q_\uloc} + \norm{u_0}_{L^q_\uloc} }.
}

(b) If $n=q$, for any $0<T<\infty$, there are $\ep(T), C_*(T)>0$ such that if $\norm{u_0}_{L^n_\uloc} + \norm{b_0}_{L^n_\uloc}\le\ep(T)$, then there is a unique mild solution $(u,b)$ of the F-M system of MHD type \eqref{eq-mMHD}-\eqref{eq-mMHD-BC} in the class \eqref{eq-mMHD-uloc-class}.
\end{thm}

\begin{remark}
In the critical case $n=q$, we do not need the smallness of the initial data to prove the local-in-time existence of the mild solution if $u_0,b_0\in E^n = \overline{C^\infty_c}^{L^n_\uloc}$.
\end{remark}

\subsection{Mild solutions of nematic liquid crystal flows}\label{sec:NLCF}

The mathematical study of the liquid crystal flow was initiated by Ericksen \cite{Ericksen-ARMA1962} and Leslie \cite{Leslie-ARMA1968} in the 1960s.
Over several decades, their works have attracted considerable interest.
In the late 1980s, Lin and Liu \cite{Lin-CPAM1989, LL-CPAM1995} proposed a simplified model to describe the flow phenomena of nematic liquid crystals in dimension $n=2,3$.
The model of \emph{nematic liquid crystal flow} (abbreviated as NLCF) delineates the time evolution of the materials under the influence of velocity field and the macroscopic description of the microscopic orientation field of liquid crystals.
It can be viewed as a coupled system of the non-homogeneous incompressible Naiver-Stokes equations and the transported harmonic map heat flow to $S^2$.

In various scenarios, one can consider different boundary conditions for the nematic liquid crystal flow.
When the fluid velocity field is taken to be no-slip on the boundary, one can impose either the Dirichlet or the Neumann boundary condition for the orientation vector field of the liquid crystals.
We refer the former case as the \emph{Dirichlet problem of NLCF} and the latter case as the \emph{Neumann problem of NLCF}.

\subsubsection{The nematic liquid crystal flow with Neumann boundary condition for orientation field}\label{sec:NLCF-Neumann}

We focus on the Neumann problem in this subsection, and the Dirichlet problem of NLCF is discussed in Subsection \ref{sec:NLCF-Dirichlet} below.

In \cite{HWW-JDE2019}, the authors studied the Neumann problem of NLCF in the upper-half three-dimensional space where the macroscopic orientation field of liquid crystal molecules aligns with $e_3=(0,0,1)$ at spatial infinity.
Here, we consider a general setting of the Neumann problem of NLCF in the half space when the orientation field $d:\R^n_+\times(0,\infty)\to S^{L-1}$, $n,L\ge2$ ($n=2,3$ and $L=3$ most studied), the velocity field $u:\R^n_+\times(0,\infty)\to\R^n$ and the pressure $\pi:\R^n_+\times(0,\infty)\to\R$ satisfy
\EQ{\label{eq-NLCF-N}
\setlength\arraycolsep{1.5pt}\def\arraystretch{1.2}
\left.\begin{array}{rl}
\pd_t u - \De u + \nb\pi &= -u\cdot\nb u -\nb\cdot(\nb d\odot\nb d)\\
\pd_t d - \De d &= -u\cdot\nb d + |\nb d|^2 d\\
\div u &= 0
\end{array}\right\} \text{ in } \R^n_+\times(0,\infty),
}
where $(\nb d\odot\nb d)_{ij} = \bka{\pd_i d,\pd_j d}$, $i,j=1,\ldots,n$.
The system is supplemented with initial condition
\EQ{\label{eq-NLCF-IC-N}
(u(\cdot,0),d(\cdot,0))= (u_0, d_0),\quad \div u_0 = 0,\quad (u_0)_n|_\Si = 0,\quad |d_0| = 1,
}
boundary conditions
\EQ{\label{eq-NLCF-NBC}
u(x',0,t) = 0,\quad \pd_nd(x',0,t) = 0\ \text{ on }\ \Si\times(0,\infty),
}
and far-field condition for some constant $d_\infty \in S^{L-1}$
\EQ{\label{eq-NLCF-far-field-N}
d(x,t) , d_0(x) \to d_\infty\ \text{ as }\ |x|\to\infty.
}

The Neumann boundary condition for $d$ in \eqref{eq-NLCF-NBC} means that there is no contribution to the surface force from $d$.

In a three-dimensional bounded domain $\Om\subset\R^3$, if $d:\Om\to S^2$, local strong solutions with large data and global strong solution with small data of the Neumann problem of NLCF were constructed in \cite{LW-JDE2012}.
For general dimensions, the existence of a unique local strong solution in a bounded domain was proved in \cite{HNPS-AnnIHPoincareAN2016} by analyzing the system as a quasilinear parabolic equation.
Such solution exists globally in time under certain smallness conditions.

In the half space case, the existence and uniqueness of the global strong solution to \eqref{eq-NLCF-N}-\eqref{eq-NLCF-far-field-N} when $n=L=3$ has been established in \cite{HWW-JDE2019} under certain smallness conditions.
They also obtained optimal time-decay rates of such solution in $L^p(\R^3_+)$, $p\ge1$.

We direct our attention towards investigating mild solutions of the Neumann problem of NLCF, \eqref{eq-NLCF-N}-\eqref{eq-NLCF-far-field-N}.
To properly define mild solutions of such system, we begin by presenting a solution formula applicable to sufficiently well-behaved solutions below.

\begin{prop}\thlabel{prop-NLCF-sol-formula-N}
Let $(u,d)$ be a sufficiently smooth and bounded solution of the Neumann problem of NLCF,
\eqref{eq-NLCF-N}-\eqref{eq-NLCF-far-field-N}.
If $\nb d$ is bounded and $(u,\nb d)(x,t)\to 0$ as $|x|\to \infty$ in $L^\infty$ or $L^q$ sense, then we have the following solution formula:
\EQS{\label{u-mild-NLCF-N}
u_i(x,t) &= \sum_{j=1}^n \int_{\R^n_+} \breve G_{ij}(x,y,t) (u_0)_j(y)\, dy \\
&\quad + \sum_{j,k=1}^n \sum_{\ell=1}^L \int_0^t \int_{\R^n_+} \pd_{y_k}G_{ij}(x,y,t-s)\bkt{u_ku_j + (\pd_kd_\ell) \pd_jd_\ell}(y,s)\, dyds,
}
for $i=1,\ldots,n$, and
\EQS{\label{d-mild-NLCF-N}
d_\ell(x,t) &= (d_\infty)_\ell + \int_{\R^n_+} G^N(x,y,t) \bkt{(d_0)_\ell(y) - (d_\infty)_\ell} dy \\
&\quad + \int_0^t \int_{\R^n_+} G^N(x,y,t-s)\bkt{ -(u\cdot\nb)d_\ell + |\nb d|^2 d_\ell}(y,s)\, dyds,\ \ell=1,\ldots,L.
}

\end{prop}

The solution formula \eqref{u-mild-NLCF-N}-\eqref{d-mild-NLCF-N} is displayed in \cite[(1.8)]{HWW-JDE2019} when $n=3$ and $d_\infty=e_3$.
We give a proof of \thref{prop-NLCF-sol-formula-N} in Section \ref{sec-NLCF-sol-formula}.

For initial data $(u_0, d_0)$ satisfying \eqref{eq-NLCF-IC-N}, a solution of the nematic liquid crystal flow with Neumann boundary condition for orientation field, \eqref{eq-NLCF-N}-\eqref{eq-NLCF-far-field-N}, is called a \emph{mild solution} if it satisfies \eqref{u-mild-NLCF-N}-\eqref{d-mild-NLCF-N} with suitable estimates. 
Note that the far field condition \eqref{eq-NLCF-far-field-N} of a mild solution follows from \eqref{d-mild-NLCF-N} if $u$ and $\nabla d$ have certain spatial decay, but is unclear when they belong to $L^\infty=L^\infty_\uloc$ or $L^q_\uloc$. Hence we will only consider $Y_a$ with $a>0$ in \thref{mild-NLCF-N-Ya} and $L^q$ with $q<\infty$ in \thref{mild-NLCF-N-Lq}. A bounded function $f(x)$ with $\nb f \in Y_a$, $0<a\le1$ or $\nb f \in L^q$, $q \ge n$, may not have a limit as $|x|\to\infty$, for example, $f(x)=\sin \log \log (e+|x|)$. We will show \eqref{eq-NLCF-far-field-N} using \eqref{d-mild-NLCF-N}, see Remark \ref{rem6.1}.

\begin{thm}\thlabel{mild-NLCF-N-Ya}
Let $n,L\ge2$ and $0< a\le n$. For any $u_0\in Y_a$ with $\div u_0 = 0$ and $(u_0)_n|_\Si=0$, any constant unit vector $d_\infty\in\R^L$, any $d_0\in L^\infty(\R^n_+;\R^L)$ with $|d_0| = 1$, $\nb d_0\in Y_a$,
and $\lim_{|x|\to \infty} d_0 (x)= d_\infty$,
there is a
mild solution $(u,d)$ of the Neumann problem of NLCF,
\eqref{eq-NLCF-N}-\eqref{eq-NLCF-far-field-N}, for some time interval $(0,T)$, such that $(u,d,\nb d)\in L^\infty(0,T; Y_a\times L^\infty\times Y_a)$. Moreover, $ d(x,t) -d_\infty  \to 0$ as $|x|\to \infty$, uniformly in $t \in (0,T)$.
\end{thm}

\begin{thm}\thlabel{mild-NLCF-N-Lq}
Let $n,L\ge2$, $n\le q<\infty$, $u_0\in L^q_\si(\R^n_+)$, 
constant unit vector $d_\infty\in\R^L$, $d_0\in L^\infty(\R^n_+;\R^L)$ with $|d_0| = 1$, $\nb d_0\in L^q(\R^n_+)$, 
and $\lim_{|x|\to \infty} d_0 (x)= d_\infty$.
Then there exist $T = T(n,L,q,u_0,d_0,\nb d_0)>0$ and a unique mild solution $(u,d)$ of the Neumann problem of NLCF, \eqref{eq-NLCF-N}-\eqref{eq-NLCF-far-field-N} in the class 
\EQN{
\sup_{0<t<T} &\Big[\norm{u(t)}_{L^q} + \norm{d(t) - d_\infty}_{L^\infty} + \norm{\nb d(t)}_{L^q}\\
&\ + t^{\frac{n}{2q}}\bke{\norm{u(t)}_{L^\infty} + \norm{\nb d(t)}_{L^\infty} } + t^{\frac12} \bke{\norm{\nb u(t)}_{L^\infty} + \norm{\nb^2 d(t)}_{L^\infty} } \Big]\\
 \le&\, C_* \bke{ \norm{u_0}_{L^q} + \norm{d_0-d_\infty}_{L^\infty} + \norm{\nb d_0}_{L^q}}.
}
Moreover, $ d(x,t) -d_\infty  \to 0$ as $|x|\to \infty$, uniformly in $t \in (0,T)$.
\end{thm}

The proofs of Theorems \ref{mild-NLCF-N-Ya} and \ref{mild-NLCF-N-Lq} are given in Section \ref{sec:mild-NLCF-N}.

\subsubsection{The nematic liquid crystal flow with Dirichlet boundary condition for orientation field}\label{sec:NLCF-Dirichlet}

When the surface of a container holding liquid crystals is specially treated, the molecules of the liquid crystals near the container's surface can align with the treatment, allowing their orientation to be specified near the boundary. This phenomenon is commonly referred to as the \emph{strong anchoring condition} \cite{LL-JPDE2001}.
Mathematically, the dynamics of the liquid crystal flow in this scenario can be described as the following Dirichlet boundary value problem.
Consider $d:\R^n_+\times(0,\infty)\to S^{L-1}$ and $u:\R^n_+\times(0,\infty)\to\R^n$ for $n,L\ge2$
\EQ{\label{eq-NLCF-D}
\setlength\arraycolsep{1.5pt}\def\arraystretch{1.2}
\left.\begin{array}{rl}
\pd_t u - \De u + \nb\pi &= -u\cdot\nb u -\nb\cdot(\nb d\odot\nb d)\\
\pd_t d - \De d &= -u\cdot\nb d + |\nb d|^2 d\\
\div u &= 0
\end{array}\right\} \text{ in } \R^n_+\times(0,\infty),
}
where $(\nb d\odot\nb d)_{ij} = \bka{\pd_i d,\pd_j d}$, $i,j=1,\ldots,n$.
Here $u$ is the fluid velocity field, $d$ denotes the macroscopic orientation field of the liquid crystal molecules, and $\pi$ represents the pressure.
The system has to be solved with initial condition
\EQ{\label{eq-NLCF-IC-D}
(u(\cdot,0),d(\cdot,0))= (u_0, d_0),\qquad \div u_0 = 0,\quad (u_0)_n|_\Si = 0,\quad |d_0| = 1,
}
boundary conditions
\EQ{\label{eq-NLCF-DBC}
u(x',0,t) = 0,\quad d(x',0,t) = d_*(x',t)\ \text{ on }\ \Si\times(0,\infty),\qquad |d_*|=1.
}
Assume the compatibility condition on $\Si\times\{t=0\}$:
\EQ{\label{eq-NLCF-D-compatible}
d_0(x',x_n=0) = d_*(x',t=0).
}
For simplicity, we will assume $d_* \in S^{L-1}$ is a constant vector.

Compared to the Neumann problem of the nematic liquid crystal flow discussed in Section \ref{sec:NLCF-Neumann}, the Dirichlet problem has received relatively more research attention.
However, as far as we are aware, the studies and results of the Dirichlet problem of NLCF are currently limited to bounded domains.
Following Lin and Liu's \cite{LL-CPAM1995} establishment of the existence of a unique global classical solution of the Ericksen-Leslie system with a variable length of orientations in dimension two and three, substantial research has been conducted in this area.
We provide below a non-exhaustive list of results for the Dirichlet problem of NLCF in bounded domains.

In the two-dimensional case, the global existence of Leray-Hopf-Struwe-type weak solutions was proved in \cite{LLW-ARMA2010}.
The associated uniqueness problem was shown in \cite{LW-ChineseAnnMath2010}.
In regards to singularity formation, infinite-time singularities of $m$-equivariant solutions were constructed in \cite{CY-ARMA2017}.
Moreover, the third author and collaborators \cite{LLWWZ-CPAM2022} constructed solutions that blow up at prescribed interior points in finite time.

For the case of dimension three,
under an assumption that the initial orientation field $d_0(\Om)\subset S^2_+$, a global existence of weak solutions satisfying the global energy inequality was proved in \cite{LW-CPAM2016}.
Additionally, two examples of finite time singularity have been constructed in \cite{HLLW-ARMA2016}.

The primary focus of this study lies in exploring mild solutions of NLCF in the half space, \eqref{eq-NLCF-D}-\eqref{eq-NLCF-D-compatible}. To provide a proper definition of mild solutions, we begin by presenting the following solution formula.

\begin{prop}\thlabel{prop-NLCF-sol-formula-D}
Let $(u,d)$ be a sufficiently smooth and bounded solution of the Dirichlet problem of NLCF,
\eqref{eq-NLCF-D}-\eqref{eq-NLCF-D-compatible}.
If $\nb d$ is bounded and $d_*$ is constant, then we have the following solution formula:
\EQS{\label{u-mild-NLCF-D}
u_i(x,t) &= \sum_{j=1}^n \int_{\R^n_+} \breve G_{ij}(x,y,t) (u_0)_j(y)\, dy \\
&\quad + \sum_{j,k=1}^n \sum_{\ell=1}^L \int_0^t \int_{\R^n_+} \pd_{y_k}G_{ij}(x,y,t-s)\bkt{u_ku_j + (\pd_kd_\ell) \pd_jd_\ell}(y,s)\, dyds,
}
for $i=1,\ldots,n$,
and for $\ell=1,\ldots,L$,
\EQS{\label{d-mild-NLCF-D}
d_\ell(x,t) &= (d_*)_\ell + \int_{\R^n_+} G^D(x,y,t) \bkt{(d_0)_\ell(y) - (d_*)_\ell} dy\\
&\quad + \int_0^t \int_{\R^n_+} G^D(x,y,t-s)\bkt{ -(u\cdot\nb)d_\ell + |\nb d|^2 d_\ell }(y,s)\, dyds.
}
\end{prop}

The proof of \thref{prop-NLCF-sol-formula-D} is contained in Section \ref{sec-NLCF-sol-formula}.

For initial data $(u_0, d_0)$ satisfying \eqref{eq-NLCF-IC-D}, a solution of the nematic liquid crystal flow with Dirichlet boundary condition for orientation field, \eqref{eq-NLCF-D}-\eqref{eq-NLCF-D-compatible}, is called a \emph{mild solution} if it satisfies \eqref{u-mild-NLCF-D}-\eqref{d-mild-NLCF-D} with suitable estimates.

\begin{thm}\thlabel{mild-NLCF-D-Ya}
Let $n,L\ge2$ and $0\le a\le n$. For any $u_0\in Y_a$ with $\div u_0=0 $ and $(u_0)_n|_\Si=0$, constant unit vector $d_*\in\R^L$, $d_0\in L^\infty(\R^n_+;\R^L)$ with $|d_0| = 1$,  $\nb d_0\in Y_a$, and  $d_0|_{\Si} = d_*$,
 there is a 
 mild solution $(u,d)$ of the Dirichlet problem of NLCF,
\eqref{eq-NLCF-D}-\eqref{eq-NLCF-D-compatible}, assuming the constant boundary condition \eqref{eq-NLCF-D-compatible}, for some time interval $(0,T)$, such that $(u,d,\nb d)\in L^\infty(0,T; Y_a\times L^\infty\times Y_a)$.
\end{thm}

Note that $a=0$ is allowed in Theorem \ref{mild-NLCF-D-Ya}, in contrast to Theorem \ref{mild-NLCF-N-Ya}, because the Dirichlet problem does not require the far-field condition.

\begin{thm}\thlabel{mild-NLCF-D-Lq}
Let $n,L\ge2$ and $n\le q\le\infty$. For any $u_0\in L^q_\si(\R^n_+)$, constant unit vector $d_*\in\R^L$, $d_0\in L^\infty(\R^n_+;\R^L)$ with $|d_0| = 1$,  $\nb d_0\in L^q(\R^n_+)$, and  $d_0|_{\Si} = d_*$,
there exist $T = T(n,q,u_0,d_0,\nb d_0)>0$ and a unique mild solution $(u,d)$ of the Dirichlet problem of NLCF, \eqref{eq-NLCF-D}-\eqref{eq-NLCF-D-compatible}, assuming the constant boundary condition \eqref{eq-NLCF-D-compatible}, in the class
\EQN{
\sup_{0<t<T} &\Big[\norm{u(t)}_{L^q} + \norm{d(t) - d_*}_{L^\infty} + \norm{\nb d(t)}_{L^q}\\
&\ + t^{\frac{n}{2q}}\bke{\norm{u(t)}_{L^\infty} + \norm{\nb d(t)}_{L^\infty} } + t^{\frac12} \bke{\norm{\nb u(t)}_{L^\infty} + \norm{\nb^2 d(t)}_{L^\infty} } \Big]\\
 \le&\, C_* \bke{ \norm{u_0}_{L^q} + \norm{d_0-d_*}_{L^\infty} + \norm{\nb d_0}_{L^q}}.
}
\end{thm}

Note that $q=\infty$ is allowed in Theorem \ref{mild-NLCF-D-Lq}, in contrast to Theorem \ref{mild-NLCF-N-Lq}, because the Dirichlet problem does not require the far-field condition.

\begin{thm}\thlabel{mild-NLCF-D-Lq-uloc}
Let $n,L\ge2$, $n\le q\le\infty$, $u_0\in L^q_{\uloc,\si}(\R^n_+)$, constant unit vector $d_*\in\R^L$, $d_0\in L^\infty(\R^n_+;\R^L)$ with $|d_0| = 1$,  $\nb d_0\in L^q_\uloc(\R^n_+)$, and  $d_0|_{\Si} = d_*$.

(a) If $n<q\le\infty$, then there exist $T=T(n,q,\norm{u_0}_{L^q_\uloc},\norm{d_0 - d_*}_{L^\infty},\norm{\nb d_0}_{L^q_\uloc})>0$ and a unique mild solution $(u,d)$ of the Dirichlet problem of NLCF, \eqref{eq-NLCF-D}-\eqref{eq-NLCF-D-compatible}, assuming the constant boundary condition \eqref{eq-NLCF-D-compatible}, in the class
\EQS{\label{eq-NLCF-D-uloc-class}
\sup_{0<t<T} &\Big[\norm{u(t)}_{L^q_\uloc} + \norm{d(t) - d_*}_{L^\infty} + \norm{\nb d(t)}_{L^q_\uloc}\\
&\ + t^{\frac{n}{2q}}\bke{\norm{u(t)}_{L^\infty} + \norm{\nb d(t)}_{L^\infty} } + t^{\frac12} \bke{\norm{\nb u(t)}_{L^q_\uloc} + \norm{\nb^2 d(t)}_{L^q_\uloc} } \Big]\\
 \le&\, C_* \bke{ \norm{u_0}_{L^q_\uloc} + \norm{d_0-d_*}_{L^\infty} + \norm{\nb d_0}_{L^q_\uloc}}.
}

(b) If $n=q$, for any $0<T<\infty$, there are $\ep(T), C_*(T)>0$ such that if $\norm{u_0}_{L^n_\uloc} + \norm{d_0 - d_*}_{L^\infty} + \norm{\nb d_0}_{L^n_\uloc} \le \ep(T)$, then there is a unique mild solution $(u,d)$ of the Dirichlet problem of NLCF, \eqref{eq-NLCF-D}-\eqref{eq-NLCF-D-compatible}, in the class \eqref{eq-NLCF-D-uloc-class}.
\end{thm}

\begin{remark}
For $n=q$, by using the density argument, we can 
show a zero limit similar to \eqref{eq-Lq-mild-linear-critical-YT} and
establish the local-in-time existence of the mild solution without the smallness of the initial data if $u_0, \nb d_0\in E^n = \overline{C^\infty_c}^{L^n_\uloc}$, the closure of smooth functions of compact supports in $L^n_\uloc$-norm.
\end{remark}


\begin{remark}
Theorem \ref{mild-NLCF-D-Lq-uloc}, as well as
case $a=0$ of Theorem \ref{mild-NLCF-D-Ya}, are 
 for the Dirichlet problem and have no counter part for the Neumann problem. It is because the far-field condition \eqref{eq-NLCF-far-field-N} is not meaningful if $\td d(x,t) := d(x,t) - d_*$ does not decay at spatial infinity.
\end{remark}

The proofs of Theorems \ref{mild-NLCF-D-Ya}, \ref{mild-NLCF-D-Lq}, and \ref{mild-NLCF-D-Lq-uloc} are given in Section \ref{sec:mild-NLCF-D}.

\section{Preliminaries}\label{sec-prelim}

\subsection{Integral estimates}

We first recall a couple useful integral estimates: \cite[Lemma 2.1, Lemma 2.2]{Green}.

\begin{lem}\label{lemma2.1}
For positive $L,a,d$, and $k$ we have
\[\int_0^L\frac{r^{d-1}\,dr}{(r+a)^k}\lesssim\left\{\begin{array}{ll}L^d(a+L)^{-k}& \text{ if } k<d,\\[3pt] L^d(a+L)^{-d}(1+\log_+\frac{L}a) & \text{ if } k=d,\\[3pt] %
L^d(a+L)^{-d}a^{-(k-d)}
& \text{ if } k>d.\end{array}\right.\]
\end{lem}

\begin{lem}\label{lemma2.2}
Let $a>0$, $b>0$, $k\ge0$, $m\ge0$ and $k+m>d$. Let $0\not=x\in \R^d$ and
\[I:=\int_{\R^d}\frac{dz}{(|z|+a)^k(|z-x|+b)^m}.\]
Then, with $R=\max\{|x|,\,a,\,b\}\sim|x|+a+b$,
\EQN{I\lesssim R^{d-k-m} + \de_{kd} R^{-m} \log \frac Ra
+ \de_{md} R^{-k} \log \frac Rb
+ \mathbbm 1_{k>d} R^{-m}a^{d-k}
+ \mathbbm 1_{m>d} R^{-k}b^{d-m}.
}
In particular, for $k>1=d$ and $|x|=0_+$,
\EQ{\label{lemma2.2d1}
\int_0^\infty \frac{dz}{(z+A)^k(z+1)^m} \lec R^{-m}A^{1-k} + \de_{m1}R^{-k}\log R + \one_{m>1}R^{-k},\quad R=A+1.
}
\end{lem}

If $k>0$ and $m>0$, Lemma \ref{lemma2.2} is the same as \cite[Lemma 2.2]{Green}. 
When $k=0$ or $m=0$, Lemma \ref{lemma2.2} follows directly from Lemma \ref{lemma2.1}.

\begin{lem}\thlabel{lem-logxt}
For $k\ge0$, $r>0$ and $t>0$,
\[
t^{-\frac{k}2} e^{-r^2/ct} \log(2+r) \lec \frac1{(r+\sqrt t)^k}\, \log(2+t).
\]
\end{lem}
\begin{proof}
As $t^{-\frac{k}2} e^{-r^2/2ct} \lec t^{-\frac{k}2}  (1+ \frac r{\sqrt t})^{-k} = \frac1{(r+\sqrt t)^k}$,
it suffices to show
\[
e^{-r^2/2ct} \log(2+r^2) \lec \log(2+t).
\]
It is clear when $r^2\le t$. When $r^2>t$,
\[
\text{LHS} = e^{-r^2/2ct} \bke{\log \frac{2+r^2}{2+t} + \log(2+t)}
\lec  \frac t{r^2+t} \frac{r^2}{2+t} + \log(2+t) \lec\log(2+t)  .\qedhere
\]
\end{proof}

\begin{lem}\thlabel{lem-heat-int-decay-est}
Denote by $\Ga_k$ the $k$-dimensional heat kernel, then for $x\in\R^k$
\EQ{\label{eq-lem3.4-ineq}
\int_{\R^k} \frac{\Ga_k(x-y,t)}{(|y|+1)^a}\, dy \lec \frac{ (\one_{a>k}\sqrt t + 1)^{a-k}}{(|x|+\sqrt t +1)^a} \bke{1 + \de_{ak}\log_+t}.
}
\end{lem}

\thref{lem-heat-int-decay-est} is a generalization of \cite[(9.13)]{Green} and \cite[Lemma 1]{MR191213}.

\begin{proof}[Proof of \thref{lem-heat-int-decay-est}]
Denote the integral as $u(x,t)$.
If $\max(|x|,\sqrt t) \le 2$, then
\[
u(x,t) \le 1 \lec (|x|+\sqrt{t}+1)^{-a}.
\]
It suffices to consider $\max(|x|,\sqrt t) > 2$.

If $|x| \ge \max(\sqrt t, 2)$, then
\EQS{\label{eq-lem3.4-pf-1}
u&(x,t) \le\int_{|y|<|x|/2}  c t^{-k/2} e^{-|x|^2/ct} (|y|+1)^{-a} dy +
\int_{|y|>|x|/2} \Ga(x-y,t) (|x| +1)^{-a}  dy
\\
&\le ct^{-k/2} e^{-|x|^2/ct} \int_0^{|x|/2} \frac{\rho^{k-1}}{(\rho + 1)^a}\, d\rho  + (|x| +1)^{-a}
= ct^{-k/2} e^{-|x|^2/ct}  I_0 + (|x| +1)^{-a},
}
where, by using Lemma \ref{lemma2.1},
\[
I_0 := \int_0^{|x|/2} \frac{\rho^{k-1}}{(\rho + 1)^a}\, d\rho
\lec
\begin{cases}
|x|^k(|x|+1)^{-a}&\quad \text{ if }a<k,\\
|x|^k(|x|+1)^{-k}\bke{1 + \log_+ |x|}&\quad \text{ if }a=k,\\
|x|^k(|x|+1)^{-k}&\quad \text{ if }a>k.
\end{cases}
\]
For $a\neq k$, we use the inequality $e^{-|x|^2/ct} \lec \bke{1+\frac{|x|^2}t}^{-\max(a,\,k)/2}$ to continue the computation in \eqref{eq-lem3.4-pf-1} and get
\EQS{\label{eq-lem3.4-pf-a-not-k}
u(x,t) &\lec \sqrt{t}^{\max(a,\,k)-k}(|x|+\sqrt{t})^{-\max(a,\,k)}\, \frac{|x|^k}{(|x|+1)^{\min(a,k)}}  + (|x| +1)^{-a}\\
&\lec \frac { (\one_{a>k} \sqrt t + 1)^{a-k} }  { ( |x| + \sqrt t + 1)^{a} }.
}
If $a=k$, we have $I_0 \lec 1+\log_+|x|$.
Thus, by \thref{lem-logxt},
\[
t^{-\frac{k}2} e^{-|x|^2/ct} I_0 \lec t^{-\frac{k}2} e^{-|x|^2/ct} (1+\log_+|x|)
\lec \frac1{(|x|+\sqrt t)^k}\, \log(2+t).
\]
It then follows from \eqref{eq-lem3.4-pf-1} that, for $a=k$,
\EQ{\label{eq-lem3.4-pf-a=k}
u(x,t)\lec t^{-\frac{k}2} e^{-|x|^2/ct} I_0 + (|x|+1)^{-k}
\lec \frac{\log(2+t)}{(|x|+\sqrt t)^k}\,  + (|x|+1)^{-k}
\lec \frac{1+\log_+t}{(|x|+\sqrt t+1)^k}.
}
Collecting \eqref{eq-lem3.4-pf-a-not-k} and \eqref{eq-lem3.4-pf-a=k}, we derive the estimate \eqref{eq-lem3.4-ineq} for $|x|\ge\max(\sqrt t,2)$.

If $\sqrt t \ge \max (|x|, 2)$, we discuss the two cases: $a<k$ and $a\ge k$.

For $a<k$, with $y=\sqrt t z$ and $x=\sqrt t \hat x$, we have $|\hat x|\le 1$ and
\EQS{\label{eq-lem3.4-pf-a<k}
u(x,t) &= \int_{\R^k}  c t^{-k/2}\exp\bke{- \frac {|\hat x - z|^2}{4}} \sqrt t ^{k-a} \bke{|z|+\sqrt t^{-1}}^{-a} dz\\
&\lec \sqrt t^{-a}\int_{\R^k} \frac1{(|\hat x-z| + 1)^{k+1}\bke{|z|+\sqrt t^{-1}}^a}\, dz\\
&\lec \sqrt t^{-a} \frac1{\bke{|\hat x|+1+\sqrt t^{-1}}^a}
= \frac1{(|x|+\sqrt t+1)^a}.
}
where we've used Lemma \ref{lemma2.2} in the last inequality.

If $a\ge k$, we directly use Lemma \ref{lemma2.2} to get
\EQS{\label{eq-lem3.4-pf-a>k}
u(x,t) &= \int_{\R^k}  \frac{\Ga_k(x-y,t)}{(|y|+1)^a}\, dy
\lec \int_{\R^k} \frac1{(|x-y|+\sqrt t)^k(|y|+1)^a}\, dy\\
&\lec \frac{\de_{ak}\log(|x|+\sqrt t+1) + \one_{a>k}}{(|x|+\sqrt t+1)^k} + \frac1{(|x|+\sqrt t+1)^a}\log\frac{|x|+\sqrt t+1}{\sqrt t}\\
&\lec \de_{ak}\, \frac{\log(2+t)}{(|x|+\sqrt t+1)^a} + \one_{a>k} \frac{(1+\sqrt t)^{a-k}}{(|x|+\sqrt t+1)^a} + \frac{\log(|x|+\sqrt t+1)}{(|x|+\sqrt t+1)^a}\\
&\lec \de_{ak}\, \frac{\log(2+t)}{(|x|+\sqrt t+1)^a} + \one_{a>k} \frac{(1+\sqrt t)^{a-k}}{(|x|+\sqrt t+1)^a} + \frac{\log(2+t)}{(|x|+\sqrt t+1)^a}.
}
Combining \eqref{eq-lem3.4-pf-a<k} and \eqref{eq-lem3.4-pf-a>k}, we establish \eqref{eq-lem3.4-ineq} for $\sqrt t\ge\max(|x|,2)$ and complete the proof of the lemma.
\end{proof}

\subsection{Linear and bilinear estimates for construction of mild solutions}

In this subsection, we establish linear and nonlinear estimates to be used to construct mild solutions. Recall
$G^N(x,y,t)$ and $G^D(x,y,t)$ are the Green functions of the heat equation in the half space for Neumann and Dirichlet boundary conditions, respectively, defined in \eqref{eq-GN-def} and \eqref{eq-GD-def}. 
We will also use the semigroup notation that
\EQ{\label{DeN.def}
e^{t\De^N} f (x)= \int _{\R^n_+}  G^N(x,y,t) f(y)\,dy,
\quad
e^{t\De^D} f (x)= \int _{\R^n_+}  G^D(x,y,t) f(y)\,dy.
}
We also let  $A=-\mathbb P\De$ be the Stokes operator, so that $e^{-tA}u_0$ is given by the first integral in \eqref{E1.3}. When $u_0$ is  
solenoidal, $e^{-tA}u_0$ is also given by the first integral in \eqref{eq-Stokes-sol-formula-divu0=0}.

\begin{lem}\thlabel{lem-heat-est-Lq}
Let $n\ge 2$, $1\le p\le q\le\infty$ and $q>1$.
If $f\in L^p(\R^n_+)$, $k,l\in\NN_0$. Then
\EQN{
\norm{\int _{\R^n_+} \pd_x^k\pd_y^l G^D(x,y,t) f(y)\,dy }_{L^q(\R^n_+)}
+ \norm{\int _{\R^n_+} \pd_x^k\pd_y^l G^N(x,y,t) f(y)\,dy }_{L^q(\R^n_+)}
\\
\le C t^{-\frac{k+l}2-\frac{n}2\bke{\frac1p-\frac1q}} \norm{f}_{L^p(\R^n_+)}.
}
\end{lem}

\begin{proof}
By the basic estimates of heat kernel,
\EQ{\label{eq-heat-Green-est}
|\pd_x^k\pd_y^l G^D(x,y,t)| + |\pd_x^k\pd_y^l G^N(x,y,t)| \lec \frac1{(|x-y|^2+t)^{\frac{n+m}2}},\qquad m=k+l.
}
\thref{lem-heat-est-Lq} then follows from the usual Young's convolution inequality.
\end{proof}

\begin{lem}\thlabel{lem-heat-est-Ya}
Let $n\ge 2$ and $0 \le a \le n$. %
For $f\in Y_a$,
\EQ{\label{eq-linear-Ya}
\norm{\int _{\R^n_+} G^D(x,y,t) f(y) dy }_{Y_a}
+ \norm{\int _{\R^n_+} G^N(x,y,t) f(y) dy }_{Y_a}
\le C (1 + \delta_{an} \log_+t)  \norm{f}_{Y_a}.
}
For $F\in Y_{2a}$ and $k,l\in\N_0$ with $k+l=1$, 
\EQ{\label{eq-bilinear-Ya}
\norm{\int _{\R^n_+} \pd_x^k\pd_{y_p}^l G^D(x,y,t) F_{pj}(y)\,dy }_{Y_a}
+ \norm{\int _{\R^n_+} \pd_x^k\pd_{y_p}^l G^N(x,y,t) F_{pj}(y)\,dy }_{Y_a}
\le C t^{-1/2} \norm{F}_{Y_{2a}}.
}
\end{lem}
\begin{proof}
The linear estimate \eqref{eq-linear-Ya} follows from the classical estimate
\EQ{\label{eq-KLLT-9.13}
\norm{\int_{\R^n_+} \Ga(x-y,t) f(y)\, dy }_{Y_a} \lec (1+\de_{an}\log_+t) \norm{f}_{Y_a}.
}
See e.g. \cite[Lemma 1]{MR191213} for the case of $n=3$. Its statement corresponds to $1\le a\le n$ but its proof works for $0\le a<1$.

Consider now \eqref{eq-bilinear-Ya}. By \eqref{eq-heat-Green-est} we have
\EQ{\label{eq-heat-Green-est-q=1}
|\pd_x^k\pd_y^l G^D(x,y,t)| + |\pd_x^k\pd_y^l G^N(x,y,t)| \lec \frac1{(|x-y|^2+t)^{\frac{n+1}2}}.
}
It suffices to prove
\[
\int_{\R^n_+} \frac1{(|x-y|+\sqrt t)^{n+1}} \frac1{\bka{y}^{2a}}\, dy \lec t^{-\frac12} \frac1{\bka{x}^a}.
\]
Indeed, via Lemma \ref{lemma2.2}, we have
\EQN{
\int_{\R^n_+} \frac1{(|x-y|+\sqrt t)^{n+1}} \frac1{\bka{y}^{2a}}\, dy
&\lec \frac1{(|x|+\sqrt t+1)^{2a}\sqrt t}
+ \frac{\de_{(2a)n}\log(|x|+\sqrt t+1)+\one_{2a>n}}{(|x|+\sqrt t+ 1)^{n+1} } .
}
Thus, if $0\le a\le n$,
\[
\int_{\R^n_+} \frac1{(|x-y|+\sqrt t)^{n+1}} \frac1{\bka{y}^{2a}}\, dy \lec \frac1{(|x|+\sqrt t+1)^a\sqrt t},
\]
proving the lemma.
\end{proof}

\begin{lem}\thlabel{lem-heat-est-Za}
Let $n\ge 2$ and $0 \le a \le 1$. %
For $f\in Z_a$,
\EQ{\label{eq-linear-Za}
\norm{\int _{\R^n_+} G^D(x,y,t) f(y) dy }_{Z_a}
+ \norm{\int _{\R^n_+} G^N(x,y,t) f(y) dy }_{Z_a}
\le C (1 + \delta_{a1} \log_+t)  \norm{f}_{Z_a}.
}
For $F\in Z_{2a}$,
\EQ{\label{eq-bilinear-Za}
\norm{\int _{\R^n_+} \pd_{y_p} G^D(x,y,t) F_{pj}(y)\,dy }_{Z_a}
+ \norm{\int _{\R^n_+} \pd_{y_p} G^N(x,y,t) F_{pj}(y)\,dy }_{Z_a}
\le C t^{-1/2} \norm{F}_{Z_{2a}}.
}
\end{lem}

\begin{proof}
We first prove \eqref{eq-linear-Za}.
Denoting $\Ga_k$ the $k$-dimensional heat kernel, we have for $|f(y)|\le\bka{y_n}^{-a}$ that
\EQN{
&\abs{\int_{\R^n_+} G^D(x,y,t) f(y) dy} + \abs{\int_{\R^n_+} G^N(x,y,t) f(y) dy}\\
&\qquad\qquad\qquad \lec \int_0^\infty \frac{\Ga_1(x_n-y_n,t)}{(y_n+1)^a} \int_\Si \Ga_{n-1}(x'-y',t)\, dy'dy_n\\
&\qquad\qquad\qquad \lec \int_0^\infty \frac{\Ga_1(x_n-y_n,t)}{(y_n+1)^a}\, dy_n
\lec (1+\de_{a=1}\log_+t)(x_n+1)^{-a},
}
where we used the one-dimensional version of \eqref{eq-KLLT-9.13} for the last inequality and $0<a\le1$.
This proves \eqref{eq-linear-Za}.

As for \eqref{eq-bilinear-Za}, in view of the estimate \eqref{eq-heat-Green-est-q=1}, it suffices to show
\[
\int_{\R^n_+} \frac1{(|x-y|+\sqrt t)^{n+1}} \frac1{\bka{y_n}^{2a}}\, dy \lec t^{-\frac12} \frac1{\bka{x_n}^a}.
\]
By Lemma \ref{lemma2.2}, we indeed have
\EQN{
\int_{\R^n_+} \frac1{(|x-y|+\sqrt t)^{n+1}} \frac1{\bka{y_n}^{2a}}\, dy
&\lec \int_0^\infty \frac1{(y_n+1)^{2a}} \int_\Si \frac1{(|x-y|+\sqrt t)^{n+1}}\, dy'dy_n\\
&\lec \int_0^\infty \frac1{(y_n+1)^{2a}(|x_n-y_n|+\sqrt t)^2}\, dy_n\\
&\lec R^{-1-2a} + \de_{2a=1} R^{-2}\log R + \one_{2a>1} R^{-2} + R^{-2a} t^{-\frac12}\\
&\lec t^{-\frac12} \frac1{\bka{x_n}^a},
}
where $R=x_n+\sqrt t+1$. We have used $a\le1$ to bound $\one_{2a>1}R^{-2}\lec t^{-1/2}\bka{x_n}^{-a}$.
This completes the proof of the lemma.
\end{proof}

\begin{lem}\thlabel{lem-heat-est-Lq-uloc}
Let $n\ge2$. Let $1\le p\le q\le\infty$.
For $f\in L^p_\uloc$, $k,l\in\N_0$. Then
\EQS{
\norm{\int _{\R^n_+} \pd_x^k\pd_y^l G^D(x,y,t) f(y)\,dy }_{L^q_\uloc(\R^n_+)}
&+ \norm{\int _{\R^n_+} \pd_x^k\pd_y^l G^N(x,y,t) f(y)\,dy }_{L^q_\uloc(\R^n_+)} \\
&\le C t^{-\frac{k+l}2}\bke{ 1+t^{-\frac{n}2\bke{\frac1p-\frac1q}} } \norm{f}_{L^p_\uloc(\R^n_+)}.
}
\end{lem}
\begin{proof}
Denote $m=k+l$.
By the estimate \eqref{eq-heat-Green-est}, we get
\[
\abs{\int _{\R^n_+} \pd_x^k\pd_y^l G^D(x,y,t) f(y)\,dy } + \abs{\int _{\R^n_+} \pd_x^k\pd_y^l G^N(x,y,t) f(y)\,dy }
\lec \int_{\R^n} H_t^0(x-y) |f(y)|\, dy,
\]
where
\[
H_t^0(x) = t^{-\frac{n+m}2} H_1^0\bke{\frac{x}{\sqrt t}},\quad
H_1^0(x) = \frac1{(|x|^2+1)^{\frac{n+m}2}}\in L^1\cap L^\infty(\R^n).
\]
By the estimate of Maekawa-Terasawa \cite[Theorem 3.1]{MaTe} with $\frac1q=\frac1r+\frac1p-1$, we derive
\EQN{
&\norm{\int _{\R^n_+} \pd_x^k\pd_y^l G^D(x,y,t) f(y)\,dy }_{L^q_\uloc(\R^n_+)}
+ \norm{\int _{\R^n_+} \pd_x^k\pd_y^l G^N(x,y,t) f(y)\,dy }_{L^q_\uloc(\R^n_+)} \\
&\qquad\qquad\qquad\qquad\qquad \lec t^{-\frac{m}2}\bke{t^{-\frac{n}2\bke{\frac1p-\frac1q}}\norm{H_1^0}_{L^r(\R^n)} + \norm{H_1^0}_{L^1(\R^n)} } \norm{f}_{L^p_\uloc(\R^n_+)},
}
completing the proof of the lemma.
\end{proof}

Here we quote an improved version of \cite[Lemma 9.1(b)]{Green} shown in \cite{Green-errata}.

\begin{lem}\label{lem-9.1}
Let $n \ge 2$, $1 \le p \le q \le \infty$ and $1<q$.
Let $F\in L^p(\R^n_+)$, $a,b\in \N_0$ and $1\le a+b$.
Assume $b\ge1$ if $p=q=\infty$.
Then
\[ 
\norm{\int _{\R^n_+} \pd_{x}^a\pd_{y}^b  G_{ij}(x,y,t) F(y) dy }_{L^q(\R^n_+)} \le C t^{-\frac{a+b}{2}-\frac{n}2(\frac1p-\frac1q)} \norm{F}_{L^p(\R^n_+)}.
\]
\end{lem}
It improves \cite[Lemma 9.1(b)]{Green} by including the case $p=q=\infty$ and $n=2$.

\subsection{Improved boundary-vanishing estimates}\label{app-improved-bdry-vanish-est}

In this subsection, we prove the improved boundary-vanishing estimate
 \eqref{eq-improved-bdry-vanish-est-new}. It will follow from \eqref{eq_thm3_al_yn}  
  and the following estimate
\EQS{\label{eq-improved-bdry-vanish-est}
|\pd_{y_p}G_{ij}(x,y,t)|
&\lesssim\frac{x_n^\al}{(|x-y|+\sqrt t)^{n+1}(|x-y^*|+\sqrt t)^\al} \\
&\quad + \frac{x_n^\al{\rm Ln}}{\sqrt t^\al(|x-y^*|+\sqrt t)^{n+1-\de_{pn}}(y_n+\sqrt t)^{\de_{pn}}},
}
where
\[
{\rm Ln} = 1 + \de_{n2} \bkt{\log\bke{(1-\de_{pn})|x'-y'|+x_n+y_n + \sqrt t} - \log(\sqrt t) }.
\]



Note that \eqref{eq_thm3_al_yn} yields
\EQS{\label{eq_thm3_al_yn-new}
|\pd_{y_p}G_{ij}(x,y,t)|
&\lesssim\frac{x_n^\al}{(|x-y|+\sqrt t)^{n+1}(|x-y^*|+\sqrt t)^\al} \\
&\quad + \frac{x_n^\al{\rm LN}}{\sqrt t^\al(|x-y^*|+\sqrt t)^{n+1-\de_{pn}(1-\de_{jn})}(y_n+\sqrt t)^{\de_{pn}(1-\de_{jn})}}.
}
When $j=n$, among the four summands in $\LN$, only the term $\LN^{0n}_{ij1q}$ has a log factor when $n=2$, and its $\mu = \de_{i<n}$ no matter $q=0$ or $q=1$. Thus when $j=n$, 
\[
{\rm LN}\sim1+\de_{n2}(1-\de_{in})\bkt{\log\bke{(1-\de_{in})(1-\de_{pn})|x'-y'| + x_n + y_n+\sqrt t} - \log(\sqrt t)}.
\]

Therefore, using \eqref{eq_thm3_al_yn-new} for the case $j=n$, and using \eqref{eq-improved-bdry-vanish-est} for $j<n$, we obtain \eqref{eq-improved-bdry-vanish-est-new}.


\bigskip 

We now proceed to prove the improved boundary-vanishing estimates \eqref{eq-improved-bdry-vanish-est}.
The idea is to use the first estimates of Green tensor in \cite[Proposition 1.1]{Green}.
First of all, note that \cite[Proposition 1.1]{Green} gives
\EQ{\label{eq-pdy'G}
|\pd_{y'}G_{ij}(x,y,t)| \lec \frac1{(|x-y|+\sqrt t)^{n+1}},
}
\EQ{\label{eq-pdynG}
|\pd_{y_n}G_{ij}(x,y,t)| \lec \frac1{(|x-y|+\sqrt t)^{n+1}} + \frac1{(|x-y^*|+\sqrt t)^{n}(y_n+\sqrt t)},
}
\EQ{\label{eq-pdy'xnG}
|\pd_{y'}\pd_{x_n}G_{ij}(x,y,t)| \lec \frac1{(|x-y|+\sqrt t)^{n+2}} +  \frac{1+\de_{n2}\log\bke{\frac{|x-y^*|+\sqrt t}{\sqrt t}}}{(|x-y^*|+\sqrt t)^{n+1}(x_n+\sqrt t)},
}
and
\EQ{\label{eq-pdynxnG}
|\pd_{y_n}\pd_{x_n}G_{ij}(x,y,t)| \lec \frac1{(|x-y|+\sqrt t)^{n+2}} + \frac{1+\de_{n2}\log\bke{\frac{x_n+y_n+\sqrt t}{\sqrt t}}}{(|x-y^*|+\sqrt t)^n(x_n+\sqrt t)(y_n+\sqrt t)}.
}

\medskip

\noindent{\bf Boundary estimate of $\pd_{y'}G_{ij}$.}
It follows from \eqref{eq_thm3_al_yn} with $l=1$, $q=m=0$ that
\EQ{\label{eq-pdy'-est}
|\pd_{y'}G_{ij}(x,y,t)|
\lesssim\frac{x_n}{(|x-y|+\sqrt t)^{n+1}(|x-y^*|+\sqrt t)}  + \frac{x_n\bke{ 1+\de_{n2}\log\frac{|x-y^*|+\sqrt t}{\sqrt t} }}{\sqrt t(|x-y^*|+\sqrt t)^{n+1}}.
}

\medskip

\noindent{\bf Boundary estimate of $\pd_{y_n}G_{ij}$.}
By $\pd_{y_n}G_{ij}|_{x_n=0}=0$ and \eqref{eq-pdynxnG}, we have
\EQS{\label{eq_thm3_pf-yn}
&\abs{\pd_{y_n}G_{ij}(x,y,t)} 
\le \int_0^{x_n} \abs{\pd_{y_n}\pd_{x_n}G_{ij}(x',z_n,y,t)} dz_n\\
&\lec \! \int_0^{x_n}\!\! \bkt{ \frac1{(|x'-y'|+|z_n-y_n|+\sqrt t)^{n+2}} +  \frac{1+\de_{n2}\log\bke{\frac{z_n+y_n+\sqrt t}{\sqrt t}}}{(|x'-y'|+z_n+y_n+\sqrt t)^n(z_n+\sqrt t)(y_n+\sqrt t)} } dz_n\\
&\le \int_0^{x_n}\!\! \bkt{ \frac1{(|x'-y'|+|z_n-y_n|+\sqrt t)^{n+2}} +  \frac{1+\de_{n2}\log\bke{\frac{x_n+y_n+\sqrt t}{\sqrt t}}}{\sqrt t (|x'-y'|+z_n+y_n+\sqrt t)^n(y_n+\sqrt t) } } dz_n.
}

Case 1.
If $3x_n<y_n$, then $z_n+y_n>|z_n-y_n|>\frac12 (x_n + y_n)$ for $0<z_n<x_n$. Thus,  \eqref{eq_thm3_pf-yn} gives
\EQN{
\abs{\pd_{y_n}G_{ij}(x,y,t)} 
&\lesssim \frac{x_n}{(|x-y^*|+\sqrt t)^{n+2}} + \frac{x_n\bke{ 1+\de_{n2}\log\frac{x_n+y_n+\sqrt t}{\sqrt t} }}{\sqrt t(|x-y^*|+\sqrt t)^n(y_n+\sqrt t)}.
}

Case 2.
If $y_n<3x_n < \frac12 \bke{|x'-y'|+y_n+\sqrt{t}}$, then $x_n+y_n<\frac43(|x'-y'|+\sqrt t)$, which implies $|x-y^*|+\sqrt t\lec |x'-y'|+\sqrt t$.
We drop $|z_n-y_n|$ and $z_n+y_n$ in the integrand of \eqref{eq_thm3_pf-yn} to get
\EQN{
\abs{\pd_{y_n}G_{ij}(x,y,t)} 
&\lesssim\frac{x_n}{(|x'-y'|+\sqrt t)^{n+2}} + \frac{x_n\bke{ 1+\de_{n2}\log\frac{x_n+y_n+\sqrt t}{\sqrt t} }}{\sqrt t(|x'-y'|+\sqrt t)^n(y_n+\sqrt t)} \\
&\lesssim\frac{x_n}{(|x-y^*|+\sqrt t)^{n+2}} + \frac{x_n\bke{ 1+\de_{n2}\log\frac{x_n+y_n+\sqrt t}{\sqrt t} }}{\sqrt t(|x-y^*|+\sqrt t)^n(y_n+\sqrt t)}.
}

Case 3.
If $3 x_n > y_n > \frac12 \bke{|x'-y'|+y_n+\sqrt{t}}$ or $3 x_n > \frac12 \bke{|x'-y'|+y_n+\sqrt{t}} > y_n$, then $x_n\approx |x-y^*|+\sqrt{t}$. By \eqref{eq-pdynG},
\EQN{
|\pd_{y_n}G_{ij}(x,y,t)|
&\lesssim\frac1{(|x-y|+\sqrt t)^{n+1}} + \frac1{(|x-y^*|+\sqrt t)^n (y_n+\sqrt t)}\\
&\lesssim\frac{x_n}{(|x-y|+\sqrt t)^{n+1}(|x-y^*|+\sqrt t)} + \frac{x_n}{(|x-y^*|+\sqrt t)^{n+1} (y_n+\sqrt t)}.
}

Combing the above three cases, we obtain
\EQ{\label{eq-pdyn-est}
|\pd_{y_n}G_{ij}(x,y,t)|
\lesssim\frac{x_n}{(|x-y|+\sqrt t)^{n+1}(|x-y^*|+\sqrt t)}  + \frac{x_n\bke{ 1+\de_{n2}\log\frac{x_n+y_n+\sqrt t}{\sqrt t} }}{\sqrt t(|x-y^*|+\sqrt t)^n(y_n+\sqrt t)}.
}

The estimate \eqref{eq-improved-bdry-vanish-est} follows from \eqref{eq-pdy'-est} and \eqref{eq-pdyn-est} and an interpolation argument with \eqref{eq-pdy'G} and \eqref{eq-pdynG}.
\qed

\section{Applications to Stokes and Navier-Stokes equations}\label{sec:app-NSE}

In this section, we prove Theorems \ref{thm0.1-GMS}, \ref{thm-mix-decay}, and \ref{thm-mild-bdry}.

\subsection{$L^1$ gradient bounds for solutions to the Stokes system}\label{sec-L1-grad-est}

This subsection is devoted to the proof of Theorem \ref{thm0.1-GMS}.

\begin{proof}[Proof of Theorem \ref{thm0.1-GMS}]

By \eqref{E1.3} we have
\[
u_i(x,t) = \sum_{j=1}^n \int_{\R^n_+} G_{ij}(x,y,t) (u_0)_j(y)\, dy .
\]
We first estimate $u$ in $L^q$.
Applying the Minkowski integral inequality and using Green tensor estimates \eqref{eq_Green_estimate}, we obtain
\EQS{\label{eq-thm2.1-maremonti-u-est}
\norm{u(t)}_{L^q(\R^n_+)}
&\lec \bke{\int_{\R^n_+}\abs{\int_{\R^n_+} G_{ij}(x,y,t)u_0(y)\, dy}^q dx }^{\frac1q}\\
&\lec \int_{\R^n_+} \bke{ \int_{\R^n_+} |G_{ij}(x,y,t)|^q\, dx}^{\frac1q} |u_0(y)|\, dy\\
&\lec \int_{\R^n_+} \bke{ \int_{\R^n_+} \frac1{(|x-y|+\sqrt t)^{nq}}\, dx}^{1/q} |u_0(y)|\, dy\\
&\lec t^{-\frac{n}2\bke{1-\frac1q}} \norm{u_0}_{L^1(\R^n_+)}.
}
For $\nb u$, in case that $q>1$, the estimate in \cite[(9.3)]{Green} yields
\[
\norm{\nb u(t)}_{L^q(\R^n_+)} \lec t^{-\frac12-\frac{n}2\bke{1-\frac1q}} \norm{u_0}_{L^1(\R^n_+)}.
\]
It remains to consider the case that $q=1$.
Suppose that $n\ge3$.  Using Green tensor estimates \eqref{eq_Green_estimate}, we have
\EQN{
&\norm{\na u(t)}_{L^1(\R^n_+)}
\le \int_{\R^n_+} \int_{\R^n_+} |\pd_x G_{ij}(x,y,t)|\, dx\, |u_0(y)|\, dy\\
&\lec \int_{\R^n_+} \int_{\R^n_+} \bkt{\frac1{(|x-y|+\sqrt{t})^{n+1}} + \frac1{(|x^*-y|+\sqrt{t})^n (x_n+\sqrt{t})}} dx\, |u_0(y)|\, dy=I_1+I_2.
}
We have
\[
I_1 \lec \int_{\R^n_+} \int_0^\infty \frac1{(r+\sqrt{t})^{n+1}}\, r^{n-1}\,dr\, |u_0(y)|\, dy
\lec \int_{\R^n_+} \frac1{\sqrt{t}} \, |u_0(y)|\, dy
\]
and
\EQN{
I_2 &\lec  \int_{\R^n_+} \int_0^\infty \int_0^\infty \frac1{(r+x_n+y_n+\sqrt{t})^n}\, r^{n-2}\, dr\, \frac1{x_n+\sqrt{t}}\, dx_n  |u_0(y)|\, dy\\
&\lec \int_{\R^n_+}  \int_0^\infty \frac1{(x_n+y_n+\sqrt{t})(x_n+\sqrt t)}\, dx_n\, |u_0(y)|\, dy\\
&\le \int_{\R^n_+}  \int_0^\infty \frac1{(x_n+\sqrt t)^2}\, dx_n\, |u_0(y)|\, dy
\lec t^{-1/2} \norm{u_0}_{L^1(\R^n_+)} .
}
When $n=2$, the first term $I_1$ remains the same and the second term $I_2$ becomes
\EQN{
I_2 &= \int_{\R^2_+} \int_{\R^2_+} \frac{\log\bke{|x_1-y_1|+x_2+y_2+\sqrt{t}} - \log(\sqrt{t})}{(|x^*-y|+\sqrt{t})^2 (x_2+\sqrt{t})}\, dx\, |u_0(y)|\, dy\\
&\lec \int_{\R^2_+} \int_0^\infty \int_0^\infty \frac1{(r+x_2+y_2+\sqrt{t})^2} \bke{\frac{r+x_2+y_2}{\sqrt{t}}}^\al dr\, \frac1{x_2+\sqrt{t}}\, dx_2\, |u_0(y)|\, dy,\quad \al>0,\\
&\lec \frac1{\sqrt{t}^\al} \int_{\R^2_+} \int_0^\infty \int_0^\infty \frac1{(r+x_2+y_2+\sqrt{t})^{2-\al}}\, dr\, \frac1{x_2+\sqrt{t}}\, dx_2\, |u_0(y)|\, dy\\
&\lec \frac1{\sqrt{t}^\al} \int_{\R^2_+} \int_0^\infty \frac1{(x_2+y_2+\sqrt{t})^{1-\al}(x_2+\sqrt{t})}\, dx_2\, |u_0(y)|\, dy\\
&\le \frac1{\sqrt{t}^\al} \int_{\R^2_+} \int_0^\infty \frac1{(x_2+\sqrt{t})^{2-\al}}\, dx_2\, |u_0(y)|\, dy
\lec t^{-1/2} \norm{u_0}_{L^1(\R^n_+)}.
}

To estimate $\pd_tu$, we divide the proof into two cases: $n\ge3$ and $n=2$.
For $n\ge3$, following the same approach as in \eqref{eq-thm2.1-maremonti-u-est} using Minkowski integral inequality and using Green tensor estimates \eqref{eq_Green_estimate} gives
\EQN{
\norm{\pd_tu(t)}_{L^q(\R^n_+)}
&\lec \int_{\R^n_+} \bkt{ \int_{\R^n_+} \bke{\frac1{(|x-y|+\sqrt t)^{(n+2)q}} + \frac1{t^q(|x^*-y|+\sqrt t)^{nq}} } dx}^{1/q} |u_0(y)|\, dy\\
&\lec t^{-1} \int_{\R^n_+} \bke{\int_{\R^n_+} \frac1{(|x-y|+\sqrt t)^{nq}}\, dx}^{\frac1q} |u_0(y)|\, dy
\lec t^{-1-\frac{n}2\bke{1-\frac1q}} \norm{u_0}_{L^1(\R^n_+)}.
}
When $n=2$, due to the logarithmic term $\textup{LN}^{mn}_{ijkq}$ in the Green tensor estimate \eqref{eq_Green_estimate}, there is an additional term
\EQN{
\int_{\R^2_+} &\bke{\int_{\R^2_+} \frac{\log\bke{|x_1-y_1|+x_2+y_2+\sqrt t} - \log(\sqrt t)}{t^q(|x^*-y|+\sqrt t)^{2q}}\, dx }^{\frac1q} |u_0(y)|\, dy\\
&\quad\lec t^{-1} \int_{\R^n_+} \bke{\int_0^\infty \frac1{(r+\sqrt t)^{2q}} \bke{\frac{r}{\sqrt t} }^\al r\,dr }^{\frac1q} |u_0(y)|\, dy,\quad \al>0,\\
&\quad\lec t^{-1} \int_{\R^n_+} \bke{ \frac1{\sqrt t^{2q-2}} }^{\frac1q} |u_0(y)|\, dy
= t^{-1-\bke{1-\frac1q}} \norm{u_0}_{L^1(\R^2_+)}.
}
This completes the proof of Theorem \ref{thm0.1-GMS}.
\end{proof}

\subsection{Mild solutions of Navier-Stokes equations with mixed pointwise decay}\label{sec-mix-nse}

In this subsection we construct a mild solution of the Navier-Stokes equations in $Y_{a,b}$, the space of functions with mixed decay, proving \thref{thm-mix-decay}.
The idea of the proof is to use the estimates in $Y_{a,b}$ blow in Lemma \ref{lem-mix-decay} and a fixed point argument. We skip the proof of \thref{thm-mix-decay} and focus on proving the following lemma.

\begin{lem}\thlabel{lem-mix-decay}
Let $n\ge 2$, $a\ge0$, $0\le b\le 1$, $a+b\le n$. For $u_0\in Y_{a,b}$ with $\div u_0=0$ and $u_{0,n}|_\Si=0$, and $(a,b)\neq (n-1,1)$,
\EQ{\label{eq-mix-linear}
\norm{ \sum_{j=1}^n \int_{\R^n_+} \breve G_{ij}(x,y,t) u_{0,j}(y)\, dy }_{Y_{a,b}} \le C
\bkt{ 1+(\de_{b1}+\de_{(a+b)n} )\log_+t}
\norm{u_0}_{Y_{a,b}}.
}
For $F\in Y_{2a,2b}$,
\EQ{\label{eq-mix-bilinear}
\norm{ \int_{\R^n_+} \pd_{y_p} G_{ij}(x,y,t) F_{pj}(y)\, dy }_{Y_{a,b}} \le C t^{-1/2} \norm{F}_{Y_{2a,2b}}.
}
\end{lem}

\begin{remark}
We use the restricted Green tensor $\breve G_{ij}$ in \eqref{eq-mix-linear}, and the unrestricted Green tensor $G_{ij}$ in \eqref{eq-mix-bilinear}. Estimate
\eqref{eq-mix-linear} holds for all $(a,b)$ in the trapezoid bounded by $a\ge 0$, $0\le b\le 1$, $a+b\le n$ including the four boundaries except the vertex $(a,b)=(n-1,1)$, while \eqref{eq-mix-bilinear} holds for all $(a,b)$ in the closure of the trapezoid. See Remark \ref{rem4.3} for the restrictions on $a,b$.
\end{remark}

\begin{proof}
If $a=0$ or $b=0$, the lemma follows from Lemma 9.2 and Lemma 9.3 of \cite{Green}. Thus, we consider $a,b>0$. We may suppose $\norm{u_0}_{Y_{a,b}} = 1$ without loss of generality. Write
\EQN{
\sum_{j=1}^n\int _{\R^n_+} \breve G_{ij}(x,y,t) u_{0,j}(y) dy &=\int _{\R^n_+} \Ga(x-y,t) u_{0,i}(y) dy + \int _{\R^n_+} G_{ij}^*(x,y,t) u_{0,j}(y) dy\\
& = :
u^{heat}_i(x,t) + u_i^*(x,t).
}

We first estimate $u^{heat}_i(x,t)$. If $\max (|x'|,x_n)\le 2$, then
\[
|u_i^{heat}(x,t)| \le 1 \lec (|x|+1)^{-a-b}.
\]

Consider the case $\max (|x'|,x_n)>2$. We compute
\EQN{
|u_i^{heat}(x,t)| \lec \int_{\R^n_+} \frac{\Ga(x-y,t)}{(|y|+1)^a (y_n+1)^b}\, dy = \int_{y_n>|y'|} \cdots\, dy + \int_{y_n<|y'|} \cdots\, dy
=: I + \II,
}
where, by using \thref{lem-heat-int-decay-est},
\[
I \lec \int_{\R^n_+} \frac{\Ga(x-y,t)}{(|y|+1)^{a+b}}\, dy \lec \frac{1}{(|x|+\sqrt t+1)^{a+b}}\, (1+\de_{(a+b)n} \log_+ t).
\]

If $x_n \ge \max(|x'|,2)$,
we decompose $\II$ for $y_n<|y'|$ into
\[
\II = \int_{|y'|<x_n/2} \int_0^{|y'|} \cdots\, dy_ndy' + \int_{|y'|>x_n/2} \int_0^{|y'|} \cdots\, dy_ndy'
=: \II_1 + \II_2.
\]
For $\II_1$ we have $0<y_n<|y'|<x_n/2$, so
\EQN{
\II_1 &\lec \int_{|y'|<x_n/2} \frac1{(|y'|+1)^a} \int_0^{|y'|} \frac{t^{-\frac n2} e^{-x_n^2/ct}}{(y_n+1)^b}\, dy_ndy'\\
&\lec t^{-\frac n2} e^{-x_n^2/ct} \int_{|y'|<x_n/2} \frac1{(|y'|+1)^{a+b-1}}\, (1+\de_{b1}\log_+ |y'|)\, dy'\\
&\lec t^{-\frac{n}2} e^{-x_n^2/ct} (1+\de_{b1}\log_+ x_n) \int_{|y'|<x_n/2} \frac1{(|y'|+1)^{a+b-1}}\, dy'\\
&\lec t^{-\frac{n}2} e^{-x_n^2/ct} (1+\de_{b1}\log_+ x_n)\, \frac{x_n^{n-1}}{(x_n+1)^{a+b-1}}\, (1+\de_{(a+b)n}\log_+ x_n)\\
&\lec \frac{(1+\de_{b1}\log_+ t)(1+\de_{(a+b)n}\log_+ t)}{x_n^{a+b}}
 \lec \frac{(1+\de_{b1}\log_+ t)(1+\de_{(a+b)n}\log_+ t)}{(|x|+1)^a (x_n+1)^b}.
}
where we have used \thref{lem-logxt} in the second last inequality.
For $\II_2$ we have
\EQN{
\II_2 &\lec \frac1{(x_n+1)^a} \int_{|y'|>x_n/2} \Ga_{n-1}(x'-y',t) \int_0^{|y'|} \frac{\Ga_1(x_n-y_n,t)}{(y_n+1)^b}\, dy_ndy'\\
&\le \frac1{(x_n+1)^a} \int_0^\infty \frac{\Ga_1(x_n-y_n,t)}{(y_n+1)^b}\, dy_n\\
&\lec \frac{(\one_{b>1}\sqrt t+1)^{b-1} (1+\de_{b1}\log_+ t)}{(x_n+1)^a (x_n+\sqrt t+1)^b }
\lec \frac{(\one_{b>1}\sqrt t+1)^{b-1} (1+\de_{b1}\log_+ t)}{(|x|+1)^a (x_n+1)^b}
}
where we have use \thref{lem-heat-int-decay-est} in the second last inequality.
Therefore, we have for $x_n>\max(|x'|,2)$, $b\le 1$, $a+b\le n$, that
\EQ{\label{mix-linear-heat-a}
|u^{heat}_i(x,t)| \lec \frac{(1+\de_{b1}\log_+ t)(1+\de_{(a+b)n}\log_+ t)}{(|x|+1)^a (x_n+1)^b}.
}

If $|x'| \ge \max(x_n, 2)$,
we decompose $\II$ for $y_n<|y'|$ into
\[
\II = \int_0^{|x'|/2} \int_{|y'|>y_n} \cdots\, dy'dy_n+ \int_{|x'|/2}^\infty \int_{|y'|>y_n} \cdots\, dy' dy_n
=: \II_1 + \II_2,
\]
where
\[
\II_2 \lec \frac1{(|x'|+1)^{a+b}} \int_{y_n>|x'|/2} \Ga(x-y,t)\, dy
\lec \frac1{(|x'|+1)^{a+b}}.
\]
Using Fubini's theorem, we further decompose $\II_1$ into
\EQN{
\II_1 = \int_{|y'|<|x'|/2} \int_0^{|y'|} \cdots\, dy_ndy' + \int_{|y'|>|x'|/2} \int_0^{|x'|/2} \cdots\, dy_ndy'
=: \II_{11} + \II_{12}.
}
For $\II_{11}$,
\EQN{
\II_{11}
&\lec t^{-\frac{n}2} e^{-|x'|^2/ct} \int_{|y'|<|x'|/2} \int_0^{|y'|} \frac1{(y_n+1)^b (|y'|+y_n+1)^a}\, dy_ndy'\\
&\lec t^{-\frac{n}2} e^{-|x'|^2/ct} \int_{|y'|<|x'|/2} \frac1{(|y'|+1)^{a+b-1}}\, (1+\de_{b1}\log_+|y'|)\, dy'\\
&\lec t^{-\frac{n}2} e^{-|x'|^2/ct} (1+\de_{b1}\log_+|x'|)\int_{|y'|<|x'|/2} \frac1{(|y'|+1)^{a+b-1}}\, dy'\\
&\lec t^{-\frac{n}2} e^{-|x'|^2/ct} (1+\de_{b1}\log_+ |x'|)\, \frac{|x'|^{n-1}}{(|x'|+1)^{a+b-1}}\, (1+\de_{(a+b)n}\log_+ |x'|)\\
&\lec \frac{(1+\de_{b1}\log_+ t)(1+\de_{(a+b)n}\log_+ t)}{|x'|^{a+b}}
 \lec \frac{(1+\de_{b1}\log_+ t)(1+\de_{(a+b)n}\log_+ t)}{(|x|+1)^a (x_n+1)^b},
}
where we have used \thref{lem-logxt} in the second last inequality.
For $\II_{12}$,
\EQN{
\II_{12} &\lec \frac1{(|x'|+1)^a} \int_{|y'|>|x'|/2} \Ga_{n-1}(x'-y',t) \int_0^{|x'|/2} \frac{\Ga(x_n-y_n,t)}{(y_n+1)^b}\, dy_ndy'\\
&\le \frac1{(|x'|+1)^a} \int_0^\infty \frac{\Ga(x_n-y_n,t)}{(y_n+1)^b}\, dy_n\\
&\lec \frac{(\one_{b>1}\sqrt t+1)^{b-1}(1+\de_{b1}\log_+ t)}{(|x'|+1)^a(x_n+\sqrt t+1)^b}
\lec \frac{(\one_{b>1}\sqrt t+1)^{b-1} (1+\de_{b1}\log_+ t)}{(|x|+1)^a (x_n+1)^b}
}
where we have use \thref{lem-heat-int-decay-est} in the second last inequality.
Therefore, we have for $|x'|>\max(x_n,2)$, $b\le 1$, $a+b\le n$, that
\EQ{\label{mix-linear-heat-b}
|u^{heat}_i(x,t)| \lec \frac{(1+\de_{b1}\log_+ t)(1+\de_{(a+b)n}\log_+ t)}{(|x|+1)^a (x_n+1)^b}.
}
Combing \eqref{mix-linear-heat-a}-\eqref{mix-linear-heat-b}, we obtain
\EQ{\label{mix-linear-heat-c}
\norm{u^{heat}_i}_{Y_{a,b}} \le (1+\de_{b1}\log_+ t)(1+\de_{(a+b)n}\log_+ t) \norm{u_0}_{Y_{a,b}},\quad b\le 1,\quad a+b\le n.
}

For $u^*$ with $|u_0(y)| \le \bka{y}^{-a} \bka{y_n}^{-b}$, by \eqref{Solonnikov.est}, (for both $n\ge 3$ and $n=2$), we get
\EQ{\label{J.def}
|u^*(x,t)| \lec J(x,t) = \int_{\R^n_+} \frac{e^{-\frac{cy_n^2}t}}{(|x^*-y|^2+t)^{\frac{n}2} \bka{y}^a \bka{y_n}^b}\, dy.
}
Suppose $0<a<n-1$. By Lemma \ref{lemma2.2}
\EQN{
J \lec \int_0^\infty \frac{e^{-\frac{cy_n^2}t}}{(y_n+1)^b}\,\frac1{(|x|+\sqrt t+y_n+1)^a}\, \frac1{x_n+y_n+\sqrt t}\, dy_n
\le \frac{J_1}{(|x|+\sqrt t+1)^a},
}
where
\EQN{
J_1&=\int_0^\infty \frac{e^{-\frac{cy_n^2}t}}{(y_n+1)^b}\, \frac1{x_n+y_n+\sqrt t}\, dy_n\\
&\le \frac1{x_n+\sqrt t} \int_0^{x_n+\sqrt t} \frac1{(y_n+1)^b}\, dy_n + \frac1{(x_n+\sqrt t+1)^b}\int_{x_n+\sqrt t}^\infty \frac{e^{-\frac{cy_n^2}t}}{y_n+\sqrt t}\, dy_n.
}
Using Lemma \ref{lemma2.1} to bound the first integral, we have
\EQN{
J_1&\lec \frac1{x_n+\sqrt{t}}\, \frac{(x_n+\sqrt{t})\bke{1 + \de_{b1} \log_+(x_n+\sqrt{t}) } }{(1+x_n+\sqrt{t})^{\min(b,1)}} + \frac1{(x_n+\sqrt{t}+1)^b} \int_0^\infty \frac{e^{-u^2}}{u+1}\, du\\
&\lec \frac{1 + \de_{b1} \log_+(x_n+\sqrt{t})}{(x_n+\sqrt{t}+1)^{\min(b,1)}}.
}

When $b=1$,
we want to improve the above numerator $1 + \de_{b1} \log_+(x_n+\sqrt{t})$ to a function of $t$ independent of $x_n$. It suffices to consider the case $x_n > 10+\sqrt t$. In this case,
\EQS{\label{eq4.7}
J_1&\lec \frac1{x_n} \int_{0}^\infty \frac{e^{-c\frac{y_n^2}t}}{y_n+1}\, dy_n
\lec  \frac1{x_n} \bke{ \int_{0}^{\sqrt t} \frac{1}{y_n+1}\, dy_n+  \int_{\sqrt t}^\infty  \frac{e^{-c\frac{y_n^2}t}}{y_n}\, dy_n}
\\
&\lec  \frac1{x_n} \log(\sqrt t +2).
}
We conclude when $b=1$, either $x_n > 10+\sqrt t$ or not,
\[
(|x|+\sqrt t + 1)^a J \lec  \frac {\log(2+\sqrt{t})}{x_n+\sqrt{t}+1}.
\]
Thus,
\EQN{
J&\lec \frac{1+\de_{b1}\log(2+\sqrt t)}{(|x|+\sqrt t+1)^a (x_n+\sqrt t+1)^{\min(b,\,1)}}.
}
This proves
\EQ{\label{mix-linear-u*-1}
\norm{u^*}_{Y_{a,b}} \lec (1+\de_{b1}\log_+ t)\norm{u_0}_{Y_{a,b}},\qquad 0<a<n-1,\ 0<b\le1.
}

If $a=n-1$, we have an additional term from Lemma \ref{lemma2.2},
\EQN{
\int_0^\infty& \frac{e^{-\frac{cy_n^2}t}}{(y_n+1)^b}\, \frac1{(|x|+\sqrt t+y_n+1)^n}\, \log\bke{1+\frac{|x|+\sqrt t}{y_n+1}}\, dy_n\\
&\le \frac1{(|x|+\sqrt t+1)^n} \int_0^\infty \frac{e^{-\frac{cy_n^2}t}}{(y_n+1)^b} \, \log\bke{1+\frac{|x|+\sqrt t}{y_n+1}}\, dy_n.
}

If $b<1$ then for $\ve\in(0,1-b)$
\EQN{
\int_0^\infty &\frac{e^{-\frac{cy_n^2}t}}{(y_n+1)^b}\, \log\bke{1+\frac{|x|+\sqrt t}{y_n+1}}\, dy_n
\lec \int_0^\infty \frac{e^{-\frac{cy_n^2}t}}{(y_n+1)^b} \bke{\frac{|x|+\sqrt t}{y_n+1}}^\ve dy_n\\
&\lec (|x|+\sqrt t)^\ve \int_0^{|x|+\sqrt t}\frac1{(y_n+1)^{b+\ve}}\, dy_n + \frac{(|x|+\sqrt t)^\ve}{(|x|+\sqrt t+1)^{b+\ve}}\int_{|x|+\sqrt t}^\infty e^{-\frac{cy_n^2}t}\, dy_n\\
&\lec (|x|+\sqrt t + 1)^{1-b} + \frac{\sqrt t}{(|x|+\sqrt t+1)^b}
\lec (|x|+\sqrt t + 1)^{1-b}.
}
So the additional term is bounded by $(|x|+\sqrt t + 1)^{1-b-n} = (|x|+\sqrt t + 1)^{-a-b}\le(|x|+\sqrt t + 1)^{-a}(x_n+\sqrt t + 1)^{-b}$. This proves \eqref{mix-linear-u*-1} for $a=n-1$, $b<1$.

If $n-1<a<n$, we have an additional term from Lemma \ref{lemma2.2},
\EQN{
\int_0^\infty& \frac{e^{-\frac{cy_n^2}t}}{(|x|+\sqrt t+y_n+1)^n (y_n+1)^{a+b-n+1}}\, dy_n\\
&\lec \frac1{(|x|+\sqrt t+1)^n} \bkt{\int_0^{|x|+\sqrt t} \frac1{(y_n+1)^{a+b-n+1}}\, dy_n + \int_{|x|+\sqrt t}^\infty \frac{e^{-\frac{cy_n^2}t}}{(|x|+\sqrt t+1)^{a+b-n+1}}\, dy_n }
}
The second integral is bounded by $(|x|+\sqrt t +1)^{-a-b+n}$. The first integral is bounded by $(|x|+\sqrt t +1)^{-a-b+n}$ if $a+b<n$, by $1+\log_+(|x|+\sqrt t)$ if $a+b=n$, and by $1$ if $a+b>n$. Thus, if $a+b<n$ the additional term is bounded by $(|x|+\sqrt t+1)^{-a-b}$. This proves \eqref{mix-linear-u*-1} for $n-1<a<n$, $a+b<n$.

If $n-1<a<n$, $a+b=n$, we have the same additional term from Lemma \ref{lemma2.2},
\EQN{
\int_0^\infty& \frac{e^{-\frac{cy_n^2}t}}{(|x|+\sqrt t+y_n+1)^n (y_n+1)^{a+b-n+1}}\, dy_n
= \int_0^\infty \frac{e^{-\frac{cy_n^2}t}}{(|x|+\sqrt t+y_n+1)^n (y_n+1)}\, dy_n \\
& \lec \frac{1}{(|x|+\sqrt{t}+1)^{n}}
 \int_0^\infty \frac{e^{-\frac{cy_n^2}t}}{y_n+1}\, dy_n
\lec \frac{1+\log_+ t}{(|x|+\sqrt t+1)^n}.
}
The last inequality is by the same integral estimate in \eqref{eq4.7}.

Therefore, for $b\le1$, $a+b\le n$, $(a,b)\neq(n-1,1)$,
\EQ{\label{mix-linear-u*-3}
\norm{u^*}_{Y_{a,b}} \lec \bkt{1+ (\de_{an} + \de_{b1} + \de_{(a+b)n})\log_+ t}\norm{u_0}_{Y_{a,b}}.
}
The estimate \eqref{eq-mix-linear} follows from \eqref{mix-linear-heat-c} and \eqref{mix-linear-u*-3}.

\medskip

We next consider \eqref{eq-mix-bilinear}. By the first estimates of Green tensor \cite[Proposition 1.1]{Green} with $k=0$ so that $\mu_{ik}^m=0$ and thus $\LN=1$,
\[
|\pd_{y_p} G_{ij}(x,y,t)| \lec \frac1{(|x-y|+\sqrt t)^{n+1}} + \frac1{(|x^*-y|+\sqrt t)^n (y_n+\sqrt t)}.
\]
(If we use the final
Green tensor estimate \eqref{eq_Green_estimate} we have an undesired $\LN$ factor when $n=2$.)
To show \eqref{eq-mix-bilinear},
It suffices to show
\[
I_1 + I_2 \lec t^{-1/2}\, \frac1{\bka{x}^a \bka{x_n}^b}
\]
where
\[
I_1 = \int_{\R^n_+} \frac1{(|x-y|+\sqrt t)^{n+1}}\, \frac1{\bka{y}^{2a}\bka{y_n}^{2b}}\, dy,
\]
\[
I_2 = \int_{\R^n_+} \frac1{(|x^*-y|+\sqrt t)^n (y_n+\sqrt t)}\, \frac1{\bka{y}^{2a} \bka{y_n}^{2b}}\, dy.
\]

For $I_1$, we consider the two cases: $0<a<n-1$ and $n-1\le a\le n$.

If $0<a<n-1$,
we drop $\bka{y}^{-a}$ and use Lemma \ref{lemma2.2} and that $x_n \le |x_n-y_n|+y_n$ to get
\EQN{
I_1 &\lec \int_0^\infty \int_\Si \frac1{(|x'-y'|+|x_n-y_n|+\sqrt t)^{n+1} (|y'|+y_n+1)^a}\, dy' \frac1{(y_n+1)^{2b}}\, dy_n\\
&\lec \int_0^\infty \frac1{(|x'|+|x_n-y_n|+y_n+\sqrt t+1)^a}\, \frac1{(|x_n-y_n|+\sqrt t)^2}\, \frac1{(y_n+1)^{2b}}\, dy_n\\
&\lec \frac1{(|x|+\sqrt t+1)^a}\, I_0,\quad \text{where}\
I_0 = \int_0^\infty \frac1{(|x_n-y_n|+\sqrt t)^2 (y_n+1)^{2b}}\, dy_n.
}
Using Lemma \ref{lemma2.2} again,
\begin{align}\nonumber
I_0 
&\lec  \frac{\de_{(2b)1}}{(x_n+\sqrt t+1)^2}\, \log(x_n+\sqrt t+1) +  \frac{\one_{2b>1}}{(x_n+\sqrt t+1)^2} + \frac1{(x_n+\sqrt t+1)^{2b}}\, \frac1{\sqrt t} \\
&\lec \frac1{(x_n+\sqrt t+1)^b \sqrt t}.
\label{eq-est-I0-yn}
\end{align}
So, for $0<a<n-1$,
$\displaystyle
I_1 \lec \frac1{(|x|+\sqrt t+1)^a(x_n+\sqrt t+1)^b \sqrt t}$.

If $n-1\le a\le n$, we have $n\le 2(n-1) \le 2a$ since $n\ge2$.
Note that $2a>n$ if $(a,n)\neq(1,2)$.
Then, for $(a,n)\neq(1,2)$, dropping $\bka{y_n}^{-2b}$ and applying Lemma \ref{lemma2.2} yield
\EQN{
I_1 &\lec \int_{\R^n} \frac1{(|x-y|+\sqrt t)^{n+1} (|y|+1)^{2a}}\, dy
\lec \frac1{(|x|+\sqrt t+1)^{n+1}} + \frac1{(|x|+\sqrt t+1)^{2a}\sqrt t}\\
&\lec \frac1{(|x|+\sqrt t+1)^n\sqrt t}
\lec \frac1{(|x|+\sqrt t+1)^{a+b}\sqrt t}.
}
If $a=1$ and $n=2$, we have
\EQN{
I_1 &\lec \int_0^\infty \int_{-\infty}^\infty \frac1{(|x_1-y_1|+|x_2-y_2|+\sqrt t)^3 (|y_1|+y_2+1)^2}\, dy_1 \frac1{(y_2+1)^{2b}}\, dy_2\\
&\lec \int_0^\infty \Bigg[ \frac1{(|x_1|+|x_2-y_2|+y_2+\sqrt t+1)^3(y_2+1)} \\
&\qquad\qquad +\frac1{(|x_1|+|x_2-y_2|+y_2+\sqrt t+1)^2}\, \frac1{(|x_2-y_2|+\sqrt t)^2}\Bigg] \frac1{(y_2+1)^{2b}}\, dy_2\\
&\lec \frac1{(|x|+\sqrt t+1)^3} \int_0^\infty \frac1{(y_2+1)^{1+2b}}\, dy_2 + \frac1{(|x|+\sqrt t+1)^2}\, I_0,
}
where $I_0$ is estimated in \eqref{eq-est-I0-yn}.
Hence,
\[
I_1\lec \frac1{(|x|+\sqrt t+1)(x_2+\sqrt t+1)^b\sqrt t} = \frac1{(|x|+\sqrt t+1)^a(x_2+\sqrt t+1)^b\sqrt t}.
\]

Thus, if $0<a\le n$, $0< b\le 1$ and $a+b\le n$, we have
\EQ{\label{mix-bilinear-I1}
I_1 \lec \frac1{(|x|+\sqrt t + 1)^a (x_n+\sqrt t+1)^b \sqrt t}.
}

For $I_2$, let $A=x_n+y_n+\sqrt t$. We have, by Lemma \ref{lemma2.2},
\EQN{
I_2 &\lec \int_0^\infty \bke{\int_\Si \frac1{(|x'-y'|+A)^n (|y'|+y_n+1)^{2a}}\, dy'} \frac1{(y_n+1)^{2b} (y_n+\sqrt t)}\, dy_n\\
&\lec \int_0^\infty \bke{R^{-2a}A^{-1} + R^{-n} \bke{\one_{2a=n-1}\log\frac{R}{y_n+1} + \frac{\one_{2a>n-1}}{(y_n+1)^{2a+1-n}}} }\frac{dy_n}{(y_n+1)^{2b} (y_n+\sqrt t)}\\
&= I_3 + I_4 + I_5,
}
where $R=|x|+y_n+\sqrt t+1$.
We have
\EQN{
I_3 
&\lec \frac1{(|x|+\sqrt t+1)^{2a}} \int_0^\infty \frac1{(y_n+1)^{2b} (y_n+\sqrt t) (x_n+y_n+\sqrt t)}\, dy_n.
}
If $x_n\le 1$, we have
\[
\int_0^\infty \frac1{(y_n+1)^{2b} (y_n+\sqrt t) (x_n+y_n+\sqrt t)}\, dy_n
\lec \int_0^1 \frac1{(y_n+\sqrt t)^2}\, dy_n + \int_1^\infty \frac1{y_n^{2b+1}\sqrt t}\, dy_n
\lec \frac1{\sqrt t}.
\]
If $x_n\ge 1$, using $0<b\le1$ we have
\EQN{
\int_0^\infty \frac1{(y_n+1)^{2b} (y_n+\sqrt t) (x_n+y_n+\sqrt t)}\, dy_n
&\lec \int_0^\infty \frac1{(y_n+1)^{2b}(\sqrt t)x_n^b(y_n+1)^{1-b}}\, dy_n\\
&= \frac1{x_n^b\sqrt t} \int_0^\infty \frac1{(y_n+1)^{1+b}}\, dy_n
= \frac{c}{x_n^b \sqrt t}.
}
Thus,
\[
I_3\lec \frac1{(|x|+\sqrt t+1)^{2a}}\, \frac1{(x_n+1)^b \sqrt t}.
\]

If $2a = n-1$ and $0\le b\le 1$, for any $\max(1-2b,0)<\ve<a-b+1$, we have $n-\ve>a+b$, $2b+\ve>1$, and
\EQN{
I_4 &\lec \int_0^\infty \frac1{(y_n+|x|+1+\sqrt t)^n (y_n+1)^{2b} \sqrt t}\, \frac{(|x|+\sqrt t)^\ve}{(y_n+1)^\ve}\, dy_n\\
&\lec \frac1{(|x|+1+\sqrt t)^{n-\ve}\sqrt t} \int_0^\infty \frac1{(y_n+1)^{2b+\ve}}\, dy_n
\lec \frac1{(|x|+1+\sqrt t)^{a+b} \sqrt t}.
}

If $\frac{n-1}2< a \le n$, we use the estimate \eqref{lemma2.2d1} in Lemma \ref{lemma2.2}
with $k=n$, $A=|x|+1+\sqrt t$ and $m=2a+2b+1-n>0$,
\EQN{
I_5 &\lec \frac1{\sqrt t} \int_0^\infty  \frac{dy_n}{(y_n+|x|+1+\sqrt t)^n (y_n+1)^{2a+2b+1-n}}\\
&\lec \frac1{\sqrt t} \bke{\frac1{(|x|+1+\sqrt t)^{2a+2b}} + \frac{\one_{2a+2b=n} \log(|x|+1+\sqrt t)}{(|x|+1+\sqrt t)^{n}} + \frac{\one_{2a+2b>n}}{(|x|+1+\sqrt t)^n} }\\
&\lec \frac1{\sqrt t}\, \frac1{(|x|+1+\sqrt t)^{a+b}}
}
provided $a+b\le n$.

Thus, if $0<a\le n$, $0\le b\le 1$, $a+b\le n$, we have
\EQ{\label{mix-bilinear-I2}
I_2 \lec I_3 + I_4 + I_5
\lec \frac1{(|x|+\sqrt t+1)^a (x_n+1)^b \sqrt t}.
}
This and the $I_1$ estimate \eqref{mix-bilinear-I1} show \eqref{eq-mix-bilinear}. This completes the proof of \thref{lem-mix-decay}.
\end{proof}

\begin{remark}\label{rem4.3}
For \eqref{eq-mix-linear} of Lemma \ref{lem-mix-decay}, we require $a+b\le n$,
$ b\le 1$, and $(a,b)\not=(n-1,1)$. It is optimal for our proof because the function $J(x,t)$ in our proof does not have the desired upper bound. Recall $J$ is defined in \eqref{J.def},
\[
 J(x,t) = \int_{\R^n_+} \frac{e^{-\frac{cy_n^2}t}}{(|x^*-y|^2+t)^{\frac{n}2} \bka{y}^a \bka{y_n}^b}\, dy.
\]
In the case $t\sim 1$ and $1 \ll x_n \sim |x'|$, we have
\EQN{
J &\gec \int_{{ |y'|< \tfrac 12 |x'|}}\int_0^1  \frac{ dy_n\,dy'}{|x|^n \bka{y}^a \bka{y_n}^b}
\gec  \frac{1}{ |x|^n} \int_{{ |y'|< \tfrac 12 |x'|}} \frac{dy'}{ \bka{y'}^a }
\gec\left\{ \begin{aligned}
|x| ^{-n}  \qquad \quad & a>n-1\\
|x| ^{-n} \ln |x|  \quad & a=n-1\\
|x| ^{-1-a}  \qquad & a<n-1
\end{aligned} \right. .
}
We do not have $J \lec |x|^{-a-b}$ when (i) $a+b>n$, (ii) $a+b\le n$ and $b>1$ (hence $a<n-1$), or
(iii) $a=n-1$ and $b=1$.
\end{remark}

\subsection{Mild solutions of Navier-Stokes equations with pointwise decay and boundary vanishing estimates}\label{sec-bdry-nse}

In this subsection we prove \thref{thm-mild-bdry}, constructing a mild solution of the Navier-Stokes equations in $Z_{a,\al}$, the space of functions with pointwise decay and boundary vanishing estimate defined by \eqref{eq-def-Zaal}.
It is standard to prove \thref{thm-mild-bdry} using the linear and bilinear estimates in $Z_{a,\al}$ below in Lemma \ref{lem-decay-bdry} and a fixed point argument. We omit the proof of \thref{thm-mild-bdry} and focus on the following lemma.

\begin{lem}\thlabel{lem-decay-bdry}
Let $n\ge 2$, $0\le a \le n$, and $\al\in[0,1]$. For $u_0\in Z_{a,\al}$ with $\div u_0=0$ and $u_{0,n}|_\Si=0$,
\EQ{\label{eq-bdry-linear}
\norm{ \sum_{j=1}^n \int_{\R^n_+} \breve G_{ij}(x,y,t) u_{0,j}(y)\, dy }_{Z_{a,\al}} \le C(1+\de_{an}\log_+ t) \norm{u_0}_{Z_{a,\al}}.
}
Let $F\in Z_{2a,2\al}$. Then 
\EQ{\label{eq-bdry-bilinear}
\norm{ \int_{\R^n_+} \pd_{y_p} G_{ij}(x,y,t) F_{pj}(y)\, dy }_{Z_{a,\al}} \le 
 C(t^{-\frac12} + t^{-\frac{1+\mu}2} ) 
\norm{F}_{Z_{2a,2\al}},
}
with $\mu=0$ if $n\ge3$ or if $n=2$ with $\al<1$, and $\mu>0$ arbitrarily small when $n=2$ with $\al=1$.
\end{lem}

The case $\al=0$ overlaps Lemma \ref{lem-mix-decay} as $Z_{a,0}=Y_{a,0}$.
\begin{proof}
We may assume $\al>0$ as the case $\al=0$ is covered by Lemma \ref{lem-mix-decay}.

We first establish \eqref{eq-bdry-linear}.
By Lemma \ref{lem-mix-decay} we have
\EQS{\label{eq-bdry-linear-pf0}
&\bka{x}^a\abs{\int_{\R^n_+} \breve G_{ij}(x,y,t) u_{0,j}(y)\, dy} \lec (1+\de_{an}\log_+t)\norm{u_0}_{Y_{a,0}}\\
&\quad\approx (1+\de_{an}\log_+t)\norm{u_0}_{Z_{a,0}}
\lec (1+\de_{an}\log_+t)\norm{u_0}_{Z_{a,\al}}.
}
In particular, the estimate holds for $x_n\ge1$.

Suppose $x_n<1$.
Recall the decomposition
$\breve G_{ij}(x,y,t)=\de_{ij}(\Ga(x-y,t) - \Ga(x-y^*,t))+ G_{ij}'(x,y,t)$, where $G'_{ij}(x,y,t)$ is the second term on the right hand side of the formula of $G_{ij}^*$ in \eqref{E1.6}.
By Solonnikov's estimate \cite[(2.34)]{Solonnikov-RMS2003},
\[
|\pd_{x',y'}^l\pd_{x_n}^k\pd_{y_n}^q\pd_t^m G_{ij}'(x,y,t)| \lec
\frac{e^{-\frac{cy_n^2}t}}{t^{m+\frac{q}2}(|x-y^*|^2+t)^{\frac{l+n}2}(x_n^2+t)^{\frac{k}2}}.
\]
Since $G'_{ij}(x,y,t)|_{x_n=0}=0$, the same argument as in the proof of \cite[Theorem 2.8.1]{lai-thesis2021} (using the mean value theorem and $k=1$ estimate)
yields the boundary-vanishing estimate
\[
|\pd_{x',y'}^l\pd_{y_n}^q\pd_t^m G_{ij}'(x,y,t)| \lec
\frac{x_n^\al e^{-\frac{cy_n^2}t}}{t^{m+\frac{q+\al}2}(|x-y^*|^2+t)^{\frac{l+n}2}},\qquad 0\le\al\le1.
\]
Therefore, (we may and shall assume $\norm{u_0}_{Z_{a,\al}} \le 1$)
\EQS{
\abs{\int_{\R^n_+} \breve G_{ij}(x,y,t) u_{0,j}(y)\, dy}
\lec &\int_{\R^n_+} | \Ga(x-y,t) - \Ga(x-y^*,t) | \frac{y_n^\al}{\bka{y}^a}\, dy
\\
&+ \int_{\R^n_+} \frac{x_n^\al e^{-\frac{cy_n^2}t}}{t^{\frac{\al}2}(|x-y^*|+\sqrt t)^n}\, \frac{y_n^\al}{\bka{y}^a}\, dy
=: I + \II.
}

For $\II$, using $t^{-\frac\al 2} y_n^{\al} e^{-\frac{cy_n^2}{2t}}\lec 1$,
\[
\II
\lec x_n^\al \int_{\R^n_+} \frac{e^{-\frac{cy_n^2}{2t} }}{(|x-y^*|+\sqrt t)^n \bka{y}^a}\, dy
=: x_n^\al J.
\]
The integral $J$ is estimated in the proof of \cite[Lemma 9.2]{Green} by
\EQ{\label{eq-bdry-linear-pf1}
J \lec  \frac{1 + \de_{an} \log_+ t}{(|x|+\sqrt t + 1)^a},\quad 0\le a \le n.
}

We decompose $I$ it into
\[
I = \int_{y_n\le 3x_n} \ldots  + \int_{y_n> 3x_n} \ldots =: J_1 + J_2.
\]
For $J_1$, by \thref{lem-heat-int-decay-est},
\EQ{\label{eq-bdry-linear-pf2}
J_1 \lec x_n^\al \int_{\R^n} \Ga(x-y,t) \frac1{\bka{y}^a}\, dy
\lec x_n^\al \frac{ (\one_{a>n}\sqrt t + 1)^{a-n}}{(|x|+\sqrt t +1)^a} (1 + \de_{an}\log_+t).
}
For $J_2$, note that for $x_n,y_n\ge0$
\EQN{
\abs{\Ga(x-y,t) - \Ga(x-y^*,t)}
&\lec t^{-\frac{n}2} \bke{e^{\frac{x_ny_n}{2t}} - e^{-\frac{x_ny_n}{2t}}} e^{-\frac{|x'-y'|^2 + x_n^2 + y_n^2}{4t}}\\
&\lec t^{-\frac{n}2 - \al} x_n^\al y_n^\al e^{\frac{x_ny_n}{2t}} e^{-\frac{|x'-y'|^2 + x_n^2 + y_n^2}{4t}}
}
where we've used $e^a - e^{-a} \lec ae^a$.
Moreover, since $y_n>3x_n$, we have $e^{\frac{x_ny_n}{2t}} e^{-\frac{y_n^2}{4t}} < e^{-\frac{y_n^2}{12t}}$.
Then
\[
J_2 \lec t^{-\frac{n}2 - \al} x_n^\al \int_{\R^n_+} y_n^\al e^{-\frac{y_n^2}{16t}} e^{-\frac{|x'-y'|^2 + x_n^2}{4t}}\, \frac{y_n^\al}{\bka{y}^a}\, dy.
\]
Using the fact that $y_n^{2\al} e^{-\frac{y_n^2}{16t}} \lec y_n^{2\al} \bke{1+\frac{y_n^2}{32t}}^{-\al} e^{-\frac{y_n^2}{32t}} \lec t^\al e^{-\frac{y_n^2}{32t}}$,
\EQS{\label{eq-bdry-linear-pf3}
J_2 &\lec t^{-\frac{n}2} x_n^\al \int_{\R^n_+} e^{-\frac{y_n^2}{32t}} e^{-\frac{|x'-y'|^2 + x_n^2}{4t}}\, \frac1{\bka{y}^a}\, dy
\lec x_n^\al \int_{\R^n} \Ga(x-y,8t) \frac1{\bka{y}^a}\, dy \\
&\lec x_n^\al \frac{ (\one_{a>n}\sqrt t + 1)^{a-n}}{(|x|+\sqrt t +1)^a} (1 + \de_{an}\log_+t)
}
by \thref{lem-heat-int-decay-est} again.
Combining \eqref{eq-bdry-linear-pf1}-\eqref{eq-bdry-linear-pf3}, we obtain
\EQ{\label{eq-bdry-linear-pf4}
\abs{\int_{\R^n_+} \breve G_{ij}(x,y,t) u_{0,j}(y)\, dy} \lec (1 + \de_{an} \log_+ t)\, \frac{x_n^\al}{(|x|+\sqrt t + 1)^a},\quad 0\le a \le n.
}
Using \eqref{eq-bdry-linear-pf0} for $x_n<1$ and \eqref{eq-bdry-linear-pf4} for $x_n\ge1$, the estimate \eqref{eq-bdry-linear} follows.

\bigskip

Next, we prove \eqref{eq-bdry-bilinear}.
By Lemma \ref{lem-mix-decay} we have
\EQS{\label{eq-bdry-bilinear-pf0}
\bka{x}^a\abs{\int_{\R^n_+} \pd_{y_p} G_{ij}(x,y,t) F_{pj}(y)\, dy} &\lec t^{-\frac12} \norm{F}_{Y_{2a,0}}\\
&\approx t^{-\frac12} \norm{F}_{Z_{2a,0}}
\lec t^{-\frac12}\norm{F}_{Z_{2a,2\al}}.
}
In particular, the estimate holds for $x_n\ge1$.

Suppose now $0<x_n<1$.
By the improved boundary vanishing estimate \eqref{eq-improved-bdry-vanish-est}, 
it suffices to show (we may and shall assume $\norm{F}_{Z_{2a,2\al}} \le 1$)
\EQ{\label{I1+I2+I3}
I_1 + I_2 + I_3  \le   C(t^{-\frac12} + t^{-\frac{1+\mu}2} )  \frac{x_n^\al}{\bka{x}^a},
}
where
\EQN{ 
I_1 &= \int_{\R^n_+} \frac{x_n^\al}{(|x-y|+\sqrt t)^{n+1}(|x^*-y|+\sqrt t)^\al} \frac{y_n^{2\al}}{\bka{y}^{2a}\bka{y_n}^{2\al}}\,  dy,
}
\EQN{
I_2 &= \one_{p<n} \int_{\R^n_+} \frac{x_n^\al\bke{1+\de_{n2}\log\frac{|x-y^*|+\sqrt t}{\sqrt t}}}{t^{\frac{\al}2} (|x-y^*|+\sqrt t)^{n+1}} \frac{y_n^{2\al}}{\bka{y}^{2a}\bka{y_n}^{2\al}}\, dy\\
&\lec \frac{x_n^\al}{\sqrt t^{\al+\ve}}  \int_{\R^n_+} \frac1{(|x-y^*|+\sqrt t)^{n+1-\ve}} \frac{y_n^{2\al}}{\bka{y}^{2a}\bka{y_n}^{2\al}}\, dy=:I_2',
}
and
\EQN{
I_3 &= \de_{pn} \int_{\R^n_+} \frac{x_n^\al\bke{1+\de_{n2}\log\frac{x_n+y_n+\sqrt t}{\sqrt t} }}{t^{\frac{\al}2}(|x-y^*|+\sqrt t)^n (y_n+\sqrt t)} \frac{y_n^{2\al}}{\bka{y}^{2a}\bka{y_n}^{2\al}}\, dy
\\
&\lec \frac{x_n^\al}{\sqrt t^{\al+\ve}} \int_{\R^n_+} \frac{(x_n+y_n+\sqrt t)^{\ve}}{(|x-y^*|+\sqrt t)^n (y_n+\sqrt t)} \frac{y_n^{2\al}}{\bka{y}^{2a}\bka{y_n}^{2\al}}\, dy=:I_3',
}
with $\ve=0$ if $n\ge3$, and $\ve>0$ can be chosen arbitrarily small when $n=2$.
The term $I_2$ is present when $p<n$. The term $I_3$ is present when $p=n$. 
Observe that $I_2' \lec I_3'$. Hence we only need to estimate $I_1$ and $I_3'$ when $x_n<1$.

For $I_1$, using $\frac {y_n^{2\al}} {\bka{y_n}^{2\al}}\le y_n^\al$
and \cite[(9.15)]{Green},
\EQ{\label{eq-I1-est}
I_1 
\lec x_n^\al \int_{\R^n_+} \frac1{(|x-y|+\sqrt t)^{n+1}} \frac1{\bka{y}^{2a}}\, dy
\lec x_n^\al \frac1{(|x|+\sqrt t+1)^a\sqrt t} .
}

We next estimate $I_3'$. Using $(x_n+y_n+\sqrt t)^{\ve} \lec 1 + (y_n+\sqrt t)^{\ve}$ for $x_n<1$,
\EQS{\label{I45.def}
I_3'&\lec  K(0)+K(\e),
\\
K(\be)&:=  \frac{x_n^{\al}}{\sqrt t^{\al+\ve}} \int_{\R^n_+} \frac1{(|x-y^*|+\sqrt t)^n (y_n+\sqrt t)^{1-\be}} \frac{y_n^{2\al}}{\bka{y}^{2a}\bka{y_n}^{2\al}}\, dy .
}

We will estimate $K(\be)$ for $0 \le \be\le\e<1$ simultaneously.
Let $A=x_n+y_n+\sqrt t$ and $R=|x'|+A+(y_n+1) \sim |x|+y_n + 1+\sqrt t$. Assume for the moment $2a\not=n-1$. By Lemma \ref{lemma2.2},
\EQS{\label{Kbe.dec}
K &\lec \frac{x_n^\al}{\sqrt t^{\al+\ve}} \int_0^\infty \bke{\int_\Si \frac1{(|x'-y'|+A)^n(|y'|+y_n+1)^{2a}}\, dy' } \frac{y_n^{2\al}\,dy_n}{(y_n+\sqrt t)^{1-\be}(y_n+1)^{2\al}} \\
&\lec \frac{x_n^\al}{\sqrt t^{\al+\ve}} \int_0^\infty \bket{R^{-2a}A^{-1} + R^{-n} \frac{\one_{2a>n-1}}{(y_n+1)^{2a-n+1}} } \frac{y_n^{2\al}\, dy_n}{(y_n+\sqrt t)^{1-\be}(y_n+1)^{2\al}} \\
&=: K_{1}(\be) + K_{2}(\be).
}
For $K_{1}$,
\[
K_{1} (\be)\lec \frac{x_n^\al}{\sqrt t^{\al+\ve}} \int_0^\infty \frac{y_n^{2\al}\, dy_n}{(y_n+|x|+1+\sqrt t)^{2a}(y_n+\sqrt t)^{2-\be}(y_n+1)^{2\al}}.
\]
If $0<\al<1$, we have for $\ve<1-\al$ that
\EQN{
K_{1} &\lec \frac{x_n^\al}{\sqrt t^{\al+\ve}} \int_0^\infty \frac{dy_n}{(y_n+|x|+1+\sqrt t)^{2a}(y_n+\sqrt t)^{2-\ve-\al}(y_n+1)^{\al+\e-\be}}
\\
&\lec \frac{x_n^\al}{\sqrt t^{\al+\ve} (|x|+1)^{2a}} \int_0^\infty \frac{dy_n}{(y_n+\sqrt t)^{2-\ve-\al}}
\lec \frac{x_n^\al}{\sqrt t (|x|+1)^{2a}}.
}
In the first inequality we arranged the exponent of $(y_n+\sqrt t)$ to be the same for all $\beta$, so that we have the same exponent of $\sqrt t$ in the end.
If $\al=1$, then
\EQN{
K_{1} &\lec \frac{x_n}{\sqrt t^{1+\ve}} \int_0^\infty \frac{y_n^2\, dy_n}{(y_n+|x|+1+\sqrt t)^{2a}(y_n+\sqrt t)^{2-\be}(y_n+1)^2}\\
&\lec \frac{x_n}{\sqrt t^{1+\ve}(|x|+1)^{2a}} \int_0^\infty \frac{dy_n}{(y_n+1)^{2-\be}}
\lec \frac{x_n}{\sqrt t^{1+\ve}(|x|+1)^{2a}}.
}
Thus, we have 
for $0< \al \le 1$, $0\le a\le n$ and $0 \le \be \le \e$,
\EQ{\label{eq-I321-est}
K_{1} (\be)\lec \frac{x_n^\al}{\sqrt t^{1+\de_{\al1}\ve} (|x|+1)^{2a}}.
}

Consider now the case $2a=n-1$. Denote $K(\be)=K(\be,a)$. By the definition of $K(\be)$, we have for $s\in (0,a/2)$,
\EQ{\label{eq-I322-est}
K(\be,a=\frac{n-1}2) \le K(\be,\frac{n-1}2-s) \lec \frac{x_n^\al}{\sqrt t^{1+\de_{\al1}\ve} (|x|+1)^{2a-2s}}
}
using \eqref{Kbe.dec} and \eqref{eq-I321-est}. Since $2a-2s \ge a$, it has enough spatial decay.

Suppose $\frac{n-1}2<a\le n$ so $K_{2}(\be)$ is also present. Recall 
\[
K_{2}(\be) = \frac{x_n^\al}{\sqrt t^{\al+\ve}} \int_0^\infty \frac{y_n^{2\al}\, dy_n}{(y_n+|x|+1+\sqrt t)^n (y_n+\sqrt t)^{1-\be} (y_n+1)^{2a-n+1+2\al}}.
\]
If $0<\al<1$, we have for $\ve<1-\al$ and $m=2a-n+1+\al+ \e-\be$ that
\EQN{
K_{2} &\lec \frac{x_n^\al}{\sqrt t^{\al+\ve}} \int_0^\infty \frac{dy_n}{(y_n+|x|+1+\sqrt t)^n (y_n+\sqrt t)^{1-\ve-\al}(y_n+1)^{m}}\\
&\lec \frac{x_n^\al}{\sqrt t^{\al+\ve}} \frac1{\sqrt t^{1-\ve-\al}} \int_0^\infty \frac{dy_n}{(y_n+|x|+1+\sqrt t)^n (y_n+1)^{m}}.
}
By  \eqref{lemma2.2d1} of Lemma \ref{lemma2.2},
\[
K_2 
\lec \frac{x_n^\al}{\sqrt t} \bkt{ \frac1{(|x|+1+\sqrt t)^{2a+\al}} + \de_{m1} \frac{\log(|x|+1+\sqrt t)}{(|x|+1+\sqrt t)^n} + \frac{\one_{m>1}}{(|x|+1+\sqrt t)^n} }.
\]
The second term above appears only when $m=1$, i.e., $n=2a+\al+\e-\be$. Since $\al>0$ and $\e\ge \be$, we have $n>a$. Thus, $K_{2}(\be) \lec \frac{x_n^\al}{\sqrt t(|x|+1)^a}$ for $0<\al<1$.
If $\al=1$, then
\EQN{
K_{2} &= \frac{x_n}{\sqrt t^{1+\ve}} \int_0^\infty \frac{y_n^2\, dy_n}{(y_n+|x|+1+\sqrt t)^n (y_n+\sqrt t)^{1-\be}(y_n+1)^{2a-n+3}}\\
&\lec \frac{x_n}{\sqrt t^{1+\ve}} \int_0^\infty \frac{dy_n}{(y_n+|x|+1+\sqrt t)^n (y_n+1)^{2a-n+2-\be}}\\
&\lec \frac{x_n}{\sqrt t^{1+\ve}} \bkt{ \frac1{(|x|+1+\sqrt t)^{2a+1-\be}} + \de_{(2a-n+2-\be)1} \frac{\log(|x|+1+\sqrt t)}{(|x|+1+\sqrt t)^n} + \frac{\one_{2a-n+2-\be>1}}{(|x|+1+\sqrt t)^n} },
}
where we've used \eqref{lemma2.2d1} of Lemma \ref{lemma2.2} in the last inequality.
The second term above appears only when $n=2a+1-\be>a$ if $\be<1$. Thus, $K_{2}\lec\frac{x_n}{\sqrt t^{1+\ve}(|x|+1)^a}$ when $\al=1$.
Therefore,
\EQ{\label{eq-I323-est}
K_{2}(\be) \lec \frac{x_n^\al}{\sqrt t^{1+\de_{\al1}\ve} (|x|+1)^a}.
}
Combining \eqref{eq-I321-est}, \eqref{eq-I322-est}, \eqref{eq-I323-est}, for $0< \al \le 1$, $0\le a\le n$ and $0 \le \be \le \e$,
\EQ{\label{eq-I32-est}
K(\be) \lec \frac{x_n^\al}{\sqrt t^{1+\de_{\al1}\ve} (|x|+1)^a}.
}
By \eqref{I45.def} and \eqref{Kbe.dec}, $I_3'$ has the same estimate. Together with the estimate \eqref{eq-I1-est} for $I_1$, we have shown \eqref{I1+I2+I3} for $x_n<1$,
with $\mu= \de_{\al 1}\e$, i.e., $\mu=0$ if $n\ge3$ or if $n=2$ with $\al<1$, and $\mu>0$ arbitrarily small when $n=2$ with $\al=1$.
Using \eqref{eq-bdry-bilinear-pf0} for $x_n\ge1$ and \eqref{I1+I2+I3} for $x_n<1$ yields
\eqref{eq-bdry-bilinear}.
\end{proof}

\begin{remark} 
For the proof of \eqref{eq-bdry-bilinear} when $0<x_n<1$, if one tries to bound $I_3\lec I_4+I_5$ in \eqref{I45.def} together by
\EQ{\label{rem4.5}
I_3
\lec \frac{x_n^\al}{\sqrt t^{\al+\ve}} \int_{\R^n_+} \frac{1}{(|x-y^*|+\sqrt t)^{n-\ve} (y_n+\sqrt t)} \frac{y_n^{2\al}}{\bka{y}^{2a}\bka{y_n}^{2\al}}\, dy,
}
one cannot get enough spatial decay $C(t)\frac{x_n^\al}{\bka{x}^a}$ when $n=2=a$ (and hence $\e>0$). If the log argument of $I_3$ contains $|x'-y'|$ in addition to $x_n+y_n+\sqrt t$, as what we get if we use 
 \eqref{eq_thm3_al_yn} instead of the improved \eqref{eq-improved-bdry-vanish-est}, then the above \eqref{rem4.5} is the only possible estimate of $I_3$, and we miss the case $n=a=2$. This was the motivation for proving \eqref{eq-improved-bdry-vanish-est}.
\end{remark}

%
%
\section{Mild solutions of MHD equations and F-M system of MHD type}\label{sec:mild-MHD-Ya}

In this section we prove Theorems \ref{mild-MHD-Lq}\,--\,
\ref{mild-MHD-Lq-uloc} for the viscous resistive MHD equations \eqref{eq-MHD}-\eqref{eq-MHD-bBC}, and Theorems \ref{mild-mMHD-Lq}\,--\,
\ref{mild-mMHD-Lq-uloc} for the F-M system of MHD type \eqref{eq-mMHD}-\eqref{eq-mMHD-BC}.

\subsection{A solution formula of the MHD equations in the half space}\label{sec-MHD-sol-formula}

Consider the vector heat equation
\EQ{\label{eq-b-equation}
\pd_t b - \De b = g,\ \text{ in }\R^n_+\times(0,\infty),
}
subject to the initial condition
\[
b(\cdot,0) = b_0,
\]
and the slip boundary condition \eqref{eq-MHD-bBC-explicit}.
So $b_i$ satisfies homogeneous Neumann boundary condition for $i=1,\ldots,n-1$, and $b_n$ satisfies homogeneous Dirichlet boundary condition on $\Si$.
Hence
\EQ{\label{b1.formula}
b_i(x,t) = \int_{ \R^n_+} 	G^N(x,y,t) (b_0)_i(y)\,dy + \int_0^t \!\int_{ \R^n_+} G^N(x,y,t-\tau) g_i(y,\tau)\,dy\,d\tau,\quad (i<n),
}
\EQ{\label{b3.formula}
b_n(x,t) = \int_{ \R^n_+} 	G^D(x,y,t) (b_0)_n(y)\,dy + \int_0^t \!\int_{ \R^n_+} G^D(x,y,t-\tau) g_n(y,\tau)\,dy\,d\tau,
}
where $G^N(x,y,t)$ and $G^D(x,y,t)$ are defined in \eqref{eq-GN-def} and \eqref{eq-GD-def}, respectively.

\medskip

For applications to the viscous resistive MHD equations in the half space, \eqref{eq-MHD}-\eqref{eq-MHD-bBC}, the right hand side of \eqref{eq-b-equation} takes the form $g = -u\cdot\nb b + b\cdot\nb u$.

\begin{lem}\label{th5.1}
Suppose both $u$ and $b$ are solenoidal vector fields in $\R^n_+$, i.e., they satisfy \eqref{solenoidal}.
Then $-u\cdot\nb b + b\cdot\nb u$ is solenoidal.
\end{lem}
\begin{proof}
We compute directly that
\[
\div\bke{-u\cdot\nb b + b\cdot\nb u}
= -\nb u:(\nb b)^\top - u\cdot\nb(\div b) + \nb b:(\nb u)^\top + b\cdot\nb(\div u) = 0,
\]
where $A:B = \sum_{i,j=1}^n A_{ij}B_{ij}$ for $A = (A_{ij})_{i,j=1}^n$, $B = (B_{ij})_{i,j=1}^n$. Note that we used $\div u = \div b = 0$ in the last equation above.
Moreover,
\[
\bke{-u\cdot\nb b + b\cdot\nb u}_n \big|_\Si = \sum_{i=1}^{n-1} (-u_i\pd_ib_n + b_i\pd_iu_n) \Big|_\Si + (-u_n\pd_nb_n + b_n\pd_nu_n) \Big|_\Si = 0,
\]
since $b_n|_\Si = u_n|_\Si= 0$. This shows $-u\cdot\nb b + b\cdot\nb u$ is solenoidal.
\end{proof}

In order to apply \eqref{b1.formula}-\eqref{b3.formula} to the viscous resistive MHD equations \eqref{eq-MHD}-\eqref{eq-MHD-bBC}, we need to check that the vector field $b$ given by the formulas satisfy $\div b = 0$ and $b_n|_\Si = 0$ if $b_0$ and $g$ are solenoidal.

\begin{lem}\thlabel{lem-b-solenoidal}
Suppose $b_0$ and $g$ are solenoidal vector fields, i.e., they satisfy \eqref{solenoidal}.
Let $b$ be the vector field given by \eqref{b1.formula}-\eqref{b3.formula}.
Then $b$ is also solenoidal.
\end{lem}
\begin{proof}
First we note that $b_n|_\Si = 0$ by the definition of Dirichlet Green function $G^D(x,y,t)$.

To prove $\div b = 0$, we first consider the case $g=0$ and compute
\EQN{
\div b &= \sum_{i=1}^{n-1} \int_{\R^n_+} \pd_{x_i} G^N(x,y,t) (b_0)_i(y)\, dy + \int_{\R^n_+} \pd_{x_n} G^D(x,y,t) (b_0)_n(y)\, dy\\
&= \sum_{i=1}^{n} \int_{\R^n_+} \bkt{\pd_{x_i}\Ga(x-y,t) +\e_i  \pd_{x_i}\Ga(x-y^*,t)} (b_0)_i(y)\, dy 
\quad (\ep_i = 1 -2 \de_{in})
\\
&= \sum_{i=1}^{n} \int_{\R^n_+} \bkt{ - \pd_{y_i}\Ga(x-y,t) - \pd_{y_i}\Ga(x-y^*,t)} (b_0)_i(y)\, dy .
}
Integrating by parts and using $(b_0)_n|_\Si = 0$ and $\div b_0 = 0$, we obtain
\[
\div b = \sum_{i=1}^{n} \int_{\R^n_+} \bkt{ \Ga(x-y,t) + \Ga(x-y^*,t)} \pd_i(b_0)_i(y)\, dy = 0.
\]
The general case $g\neq0$ follows the same computation above since $g$ is solenoidal by Lemma \ref{th5.1} and
the time integrals do not play a role.
\end{proof}

\begin{proof}[Proof of \thref{prop-MHD-sol-formula}]
Using the solution formulas
\eqref{eq-Stokes-sol-formula-divu0=0} for $u$  with $f=0$ and $F = -u\otimes u + b\otimes b$, and \eqref{b1.formula}-\eqref{b3.formula} for $b$ with $g=-u\cdot\nb b + b\cdot\nb u$,
thanks to \thref{lem-b-solenoidal}, we derive the following representation formula for the viscous resistive MHD equations \eqref{eq-MHD}-\eqref{eq-MHD-bBC}:
\EQS{
u_i(x,t) &= \sum_{j=1}^n \int_{\R^n_+} \breve G_{ij}(x,y,t) (u_0)_j(y)\, dy \\
&\quad + \sum_{j,k=1}^n \int_0^t \int_{\R^n_+} \pd_{y_k}G_{ij}(x,y,t-s)\bke{u_ku_j - b_kb_j}(y,s)\, dyds,\ i=1,\ldots,n,
}
and
\EQS{\label{b-mild_1-pre}
b_i(x,t) &= \int_{\R^n_+} G^N(x,y,t) (b_0)_i(y)\, dy \\
&\quad - \sum_{k=1}^n \int_0^t \int_{\R^n_+} G^N(x,y,t-s)\bke{u_k\pd_kb_i - b_k\pd_ku_i}(y,s)\, dyds,
}
for $i=1,\ldots,n-1$, and
\EQS{\label{b-mild_n-pre}
b_n(x,t) &= \int_{\R^n_+} G^D(x,y,t) (b_0)_n(y)\, dy \\
&\quad - \sum_{k=1}^n \int_0^t \int_{\R^n_+} G^D(x,y,t-s)\bke{u_k\pd_kb_n-b_k\pd_ku_n}(y,s)\, dyds.
}
Integrating the second integrals in \eqref{b-mild_1-pre} and \eqref{b-mild_n-pre} by parts and using the fact that $u$ and $b$ are solenoidal, we get
\EQN{
b_i(x,t) &= \int_{\R^n_+} G^N(x,y,t) (b_0)_i(y)\, dy \\
&\quad + \sum_{k=1}^n \int_0^t \int_{\R^n_+} \pd_{y_k} G^N(x,y,t-s)\bke{u_kb_i - b_ku_i}(y,s)\, dyds,\ (i<n),\\
b_n(x,t) &= \int_{\R^n_+} G^D(x,y,t) (b_0)_n(y)\, dy \\
&\quad + \sum_{k=1}^n \int_0^t \int_{\R^n_+} \pd_{y_k}G^D(x,y,t-s)\bke{u_kb_n-b_ku_n}(y,s)\, dyds.
}
This proves \thref{prop-MHD-sol-formula}.
\end{proof}

\subsection{Construction of mild solutions of the viscous resistive MHD equations}
Here we consider the proofs of Theorems \ref{mild-MHD-Lq}, \ref{mild-MHD-Ya}, \ref{mild-MHD-Za}, and \ref{mild-MHD-Lq-uloc}.
For the viscous resistive MHD equations in the half space, \eqref{eq-MHD}-\eqref{eq-MHD-bBC},
we write \eqref{u-mild}-\eqref{b-mild_n} in the form
\EQ{\label{eq-MHD-abstract}
(u(t),b(t)) = (e^{-tA}u_0, e^{t\De^*}b_0) + B((u,b),(u,b)),
}
where $A=-\mathbb P\De$ is the Stokes operator and $\De^* = (\De^N,\ldots,\De^N,\De^D)$ defined in \eqref{DeN.def}  so that
\EQ{\label{eq-semigp-formula}
(e^{-tA}u_0)_i(x) = \sum_{j=1}^n \int_{\R^n_+} \breve G_{ij}(x,y,t) (u_0)_j(y)\, dy,
}
\[
(e^{t\De^*}b_0)_i(x,t) = \int_{\R^n_+} G^N(x,y,t) (b_0)_i(y)\, dy,\ i=1,\ldots,n-1,
\]
\[
(e^{t\De^*}b_0)_n(x,t) = \int_{\R^n_+} G^D(x,y,t) (b_0)_n(y)\, dy,
\]
and $B=(B_1,B_2)$ is a bilinear form in which $B_1$ and $B_2$ are given by
\[
(B_1((u,b),(v,c)))_i = \sum_{j,k=1}^n \int_0^t \int_{\R^n_+} \pd_{y_k}G_{ij}(x,y,t-s)\bke{u_kv_j - b_kc_j}(y,s)\, dyds,\ i=1,\ldots n,
\]
\[
(B_2((u,b),(v,c)))_i = \sum_{k=1}^n \int_0^t \int_{\R^n_+} \pd_{y_k} G^N(x,y,t-s)\bke{u_kc_i - b_kv_i}(y,\tau)\, dyds,\ i=1,\ldots,n-1,
\]
and
\[
(B_2((u,b),(v,c)))_n = \sum_{k=1}^n \int_0^t \int_{\R^n_+} \pd_{y_k}G^D(x,y,t-s)\bke{u_kc_n-b_kv_n}(y,\tau)\, dyds.
\]

Applying the classical Picard contraction theorem (see e.g.~\cite[Theorem 13.2]{LR-book2002}) to the abstract integral equation \eqref{eq-MHD-abstract},
it is straightforward to construct a unique mild solution of the MHD equations \eqref{eq-MHD}-\eqref{eq-MHD-bBC} in $L^q$ (\thref{mild-MHD-Lq}) with the estimates \thref{lem-heat-est-Lq} and \cite[Lemma 9.1]{Green}, in $Y_a$ (\thref{mild-MHD-Ya}) with the estimates \thref{lem-heat-est-Ya} and \cite[Lemma 9.2]{Green}, in $Z_a$ (\thref{mild-MHD-Za}) with the estimates \thref{lem-heat-est-Za} and \cite[Lemma 9.3]{Green}, in $L^q_\uloc$ (\thref{mild-MHD-Lq-uloc}) with the estimates \thref{lem-heat-est-Lq-uloc} and \cite[Lemma 9.4]{Green}.
When $n=2$ and $q=\infty$ (or $a=0$), we also use Lemma \ref{lem-9.1}.
We omit the details of the proofs.

\subsection{Construction of mild solutions of the F-M system of MHD type}
Here we consider the proofs of Theorems \ref{mild-mMHD-Lq}, \ref{mild-mMHD-Ya}, \ref{mild-mMHD-Za}, and \ref{mild-mMHD-Lq-uloc}.
For the F-M system of MHD type in the half space, \eqref{eq-mMHD}-\eqref{eq-mMHD-BC},
we write \eqref{u-mild-m}-\eqref{b-mild-m} in the form
\EQ{\label{eq-mMHD-abstract}
(u(t),b(t)) = (e^{-tA}u_0, e^{-tA}b_0) + \mathcal B((u,b),(u,b)),
}
where $e^{-tA}$ is given in \eqref{eq-semigp-formula}, and $\mathcal B=(\mathcal B_1,\mathcal B_2)$ in which $\mathcal B_1$ and $\mathcal B_2$ are bilinear forms defined by
\[
(\mathcal B_1((u,b),(v,c)))_i = \sum_{j,k=1}^n \int_0^t \int_{\R^n_+} \pd_{y_k}G_{ij}(x,y,t-s)\bke{u_kv_j - b_kc_j}(y,s)\, dyds,\ i=1,\ldots,n,
\]
and
\[
(\mathcal B_2((u,b),(v,c)))_i = \sum_{j,k=1}^n \int_0^t \int_{\R^n_+} \pd_{y_k}G_{ij}(x,y,t-s)\bke{u_kc_j - b_kv_j}(y,s)\, dyds,\ i=1,\ldots,n.
\]
Similarly, carrying out the Picard iteration argument to the abstract equation \eqref{eq-mMHD-abstract}, we obtain a unique solution of the F-M system of MHD type \eqref{eq-mMHD}-\eqref{eq-mMHD-BC} in the half space  in $L^q$ (\thref{mild-mMHD-Lq}) with the estimates in \cite[Lemma 9.1]{Green}, in $Y_a$ (\thref{mild-mMHD-Ya}) with the estimates in \cite[Lemma 9.2]{Green}, in $Z_a$ (\thref{mild-mMHD-Za}) with the estimates in \cite[Lemma 9.3]{Green}, in $L^q_\uloc$ (\thref{mild-mMHD-Lq-uloc}) with the estimates in \cite[Lemma 9.4]{Green}.
When $n=2$ and $q=\infty$ (or $a=0$), we also use Lemma \ref{lem-9.1}.
We skip the details.

\section{Mild solutions of nematic liquid crystal flows}\label{sec:mild-NLCF}

\subsection{Solution formulas of Neumann and Dirichlet problems of NLCF}\label{sec-NLCF-sol-formula}

In this subsection, we prove the solution formulas in \thref{prop-NLCF-sol-formula-N} and \thref{prop-NLCF-sol-formula-D} for the nematic liquid crystal flow with Neumann and Dirichlet boundary condition for orientation field.

We first consider the Neumann problem \eqref{eq-NLCF-N}-\eqref{eq-NLCF-far-field-N} and prove \thref{prop-NLCF-sol-formula-N}.

\begin{proof}[Proof of \thref{prop-NLCF-sol-formula-N}]
It is obvious to see that \eqref{u-mild-NLCF-N} follows from the solution formula
\eqref{eq-Stokes-sol-formula-divu0=0}.
To prove \eqref{d-mild-NLCF-N}, let $\td d(x,t) = d(x,t) - d_\infty$. Then $\td d$ solves $(\pd_t-\De)\td d = -u\cdot\nb d + |\nb d|^2 d$ in $\R^n_+\times(0,\infty)$ with initial condition $\td d(x,0) = d_0(x) - d_\infty$, boundary condition $\pd_n\td d(x',0,t) = 0$ on $\Si\times(0,\infty)$, and far-field condition $\td d(x,t) \to 0$ as $|x|\to\infty$.
Applying the solution formula for the Neumann problem of inhomogeneous heat equation, we have
\EQS{\label{td-d-mild-NLCF-N}
\td d_\ell(x,t) &= \int_{\R^n_+} G^N(x,y,t) \bkt{(d_0)_\ell(y) - (d_\infty)_\ell} dy \\
&\quad + \int_0^t \int_{\R^n_+} G^N(x,y,t-s)\bkt{ -(u\cdot\nb)d_\ell + |\nb d|^2 d_\ell}(y,s)\, dyds,\ \ell=1,\ldots,L,
}
which is equivalent to \eqref{d-mild-NLCF-N}.

It remains to show $|d| = 1$. Consider $w = |d|^2 - 1$. Then $w$ satisfies
\[
\pd_t w - \De w = -u\cdot\nb w + 2|\nb d|^2 w\ \text{ in }\ \R^n_+\times(0,\infty),
\]
with initial condition $w|_{t=0}=0$, boundary condition $\pd_nw|_\Si = 0$, and $w\to0$ as $|x|\to\infty$.
Since $u$ and $\nb d$ are bounded, the parabolic maximum principle (see Theorem 5 in \cite[Section 3.4]{PW-book1967}) implies that the maximum of $w$ must occur either at $t=0$ or on $\Si$.
But the Hopf lemma (see Theorem 6 in \cite[Section 3.4]{PW-book1967}) indicates that if the maximum of $w$ occur on $\Si$ then $\pd_nw|_\Si>0$, violating the Neumann boundary condition.
Thus, $w\le0$ for all $t$.
Applying the same reasoning to $-w$, we find $w\ge0$. Therefore $w=0$  for all $t$.
This completes the proof of the proposition.
\end{proof}

We now consider the Dirichlet problem \eqref{eq-NLCF-D}-\eqref{eq-NLCF-D-compatible} and prove \thref{prop-NLCF-sol-formula-D}.

\begin{proof}[Proof of \thref{prop-NLCF-sol-formula-D}]
It is obvious to see that \eqref{u-mild-NLCF-D} follows from the solution formula
\eqref{eq-Stokes-sol-formula-divu0=0}.
To prove \eqref{d-mild-NLCF-D}, let $\td d(x,t) = d(x,t) - d_*$. Then $\td d$ solves $(\pd_t-\De)\td d = -u\cdot\nb d + |\nb d|^2 d$ in $\R^n_+\times(0,\infty)$ with initial condition $\td d(x,0) = d_0(x) - d_*$, and boundary condition $\td d(x',0,t) = 0$ on $\Si\times(0,\infty)$.
Applying the representation formula for 
the Dirichlet problem of inhomogeneous heat equation, we have for $\ell=1,\ldots,L$ that
\EQS{\label{td-d-mild-NLCF-D}
\td d_\ell(x,t) &= \int_{\R^n_+} G^D(x,y,t) \bkt{(d_0)_\ell(y) - (d_*)_\ell} dy\\
&\quad + \int_0^t \int_{\R^n_+} G^D(x,y,t-s)\bkt{ -(u\cdot\nb)d_\ell + |\nb d|^2 d_\ell }(y,s)\, dyds,
}
which is equivalent to \eqref{d-mild-NLCF-D}.

It remains to show $|d| = 1$. Consider $w = |d|^2 - 1$. Then $w$ is a bounded function that satisfies
\[
\pd_t w - \De w = -u\cdot\nb w + 2|\nb d|^2 w\ \text{ in }\ \R^n_+\times(0,\infty),
\]
with initial condition $w|_{t=0}=0$, boundary condition $w|_\Si = 0$.
Since $u$ and $\nb d$ are bounded, the parabolic maximum principle (see Theorem 5 in \cite[Section 3.4]{PW-book1967})
then implies that $w=0$ for all $t$.
This completes the proof of the proposition.
\end{proof}

\subsection{Construction of mild solutions for the Neumann problem}\label{sec:mild-NLCF-N}

In this subsection, we consider the nematic liquid crystal flow with Neumann boundary condition for orientation field, \eqref{eq-NLCF-N}-\eqref{eq-NLCF-far-field-N}.
We prove Theorems \ref{mild-NLCF-N-Ya} and \ref{mild-NLCF-N-Lq} by constructing mild solutions in the corresponding function spaces.

Let $\td d(x,t) = d(x,t) - d_\infty$ and $\td d_0(x) = d_0(x) - d_\infty$. By \eqref{d-mild-NLCF-N}, $\td d$ satisfies
\EQN{
\td d_\ell(x,t) &= \int_{\R^n_+} G^N(x,y,t)(\td d_0)_\ell(y)\, dy\\
&\quad + \int_0^t \int_{\R^n_+} G^N(x,y,t-s)\bkt{ -(u\cdot\nb)\td d_\ell + |\nb \td d|^2 (\td d + d_\infty)_\ell}(y,s)\, dyds,\ \ell=1,\ldots,L.
}

\subsubsection{Construction of mild solution in $Y_a$, $0<a\le n$}
We start by considering the problem in the $Y_a$ framework and prove \thref{mild-NLCF-N-Ya}.

For $0< a\le n$,
define the Banach space
\[
X = \bket{(u,\td d)
:\R^n_+ \to \R^n \times S^{L-1} \ \big| \ 
u\in Y_a,\ \td d\in L^\infty,\, \nb\td d\in Y_a},
\]
with the norm
\[
\norm{(u,\td d)}_X = \norm{u}_{Y_a} + \norm{\td d}_{L^\infty} + \norm{\nb\td d}_{Y_a}.
\]
For a small $0<T\le1$ to be determined, denote $L^\infty_TX=L^\infty(0,T;X)$.
For a given $(u_0,\td d_0)\in X$ define inductively a sequence $(u^m,\td d^m)$ by $(u^0,\td d^0) = (0,0)$,
\EQ{\label{eq-NLCF-iteration}
(u^{m+1},\td d^{m+1})(t) = (e^{-tA}u_0, e^{t\De^N}\td d_0) + F(u^m,\td d^m)(t),\ m\ge0,
}
where $A=-\mathbb P\De$ is the Stokes operator, $F = (F_1,F_2)$ and for $i=1,\ldots,n$,
\EQ{\label{eq-F1-def}
(F_1(u,\td d))_i(x,t) =  \sum_{j,k=1}^n \sum_{\ell=1}^L \int_0^t\int_{\R^n_+} \pd_{y_k}G_{ij}(x,y,t-s)\bke{u_ku_j + \pd_k\td d_\ell \pd_j\td d_\ell}(y,s)\, dyds,
}
and for $\ell=1,\ldots,L$,
\EQ{\label{eq-F2-def}
(F_2(u,\td d))_\ell(x,t) = \int_0^t\int_{\R^n_+} G^N(x,y,t-s)\bkt{ -(u\cdot\nb)\td d_\ell + |\nb \td d|^2 (\td d + d_\infty)_\ell}(y,s)\, dyds.
}

By \cite[Lemma 9.2]{Green} and $T\le1$, we have, for $t\in(0,T)$, that
\EQ{\label{eq-NLCF-N-Ya-linear-1}
\norm{e^{-tA}u_0}_{Y_a} \lec \norm{u_0}_{Y_a},
}
and
\EQS{\label{eq-NLCF-N-Ya-nonlinear-1}
\norm{F_1(u,\td d)(t)}_{Y_a} 
&\lec \int_0^t (t-\tau)^{-\frac12} \norm{ \bke{u\otimes u + \nb\td d\odot\nb\td d}(\tau)}_{Y_{2a}} d\tau\\
&\lec t^{\frac12} \sup_{0\le \tau\le T} \bke{ \norm{u(\tau)}_{Y_a}^2 + \norm{\nb\td d(\tau)}_{Y_a}^2 }
\le CT^{\frac12} \norm{(u,\td d)}_{L^\infty_TX}^2.
}
Moreover, by \thref{lem-heat-est-Ya}, we have, for $t\in(0,T)$, that
\EQ{\label{eq-NLCF-N-Ya-linear-2.1}
\norm{e^{t\De^N}\td d_0}_{L^\infty} \lec \norm{\td d_0}_{L^\infty}.
}
To bound $\nb e^{t\De^N}\td d_0$ in $Y_a$, note that
\begin{align}
\nonumber
\pd_{x_j} e^{t\De^N}\td d_0(x)
&= \int_{\R^n_+} \pd_{x_j} 
\bket{\Ga(x-y,t) + \Ga(x-y^*,t)} \td d_0(y)\, dy\\ \nonumber
&= \int_{\R^n_+} \bket{- \pd_{y_j}\Ga(x-y,t) -\ep_j \pd_{y_j}\Ga(x-y^*,t)} \td d_0(y)\, dy\\
&= \int_{\R^n_+} \bket{\Ga(x-y,t) + \ep_j \Ga(x-y^*,t)} \pd_j\td d_0(y)\, dy,
\label{eq-nb-d-linear}
\end{align}
using the cancellation of boundary terms on $\Si$ when $j=n$.
Note that $\Ga(x-y,t) + \ep_j \Ga(x-y^*,t)$ is $G^N(x,y,t)$ for $j=1,\ldots,n-1$, and $\Ga(x-y,t) + \ep_n \Ga(x-y^*,t) = G^D(x,y,t)$.
Thus, by \eqref{eq-linear-Ya} and $T\le1$,
\EQ{\label{eq-NLCF-N-Ya-linear-2.3}
\norm{\nb e^{t\De^N}\td d_0}_{Y_a} \lec \norm{\nb\td d_0}_{Y_a}.
}
By \thref{lem-heat-est-Ya} with $a=0$, we have, for $t\in(0,T)$, that
\EQS{\label{eq-NLCF-N-Ya-nonlinear-2.1}
\norm{F_2(u,\td d)(t)}_{L^\infty} 
& \lec \int_0^t \norm{\bke{u\cdot\nb\td d + |\nb\td d|^2(\td d + d_\infty)}(\tau)}_{L^\infty} d\tau\\
& \le t \sup_{0\le \tau\le T} \bkt{\norm{u(\tau)}_{Y_a}\norm{\nb\td d(\tau)}_{Y_a} + \norm{\nb\td d(\tau)}_{Y_a}^2\bke{\norm{\td d(\tau)}_{L^\infty} + 1} }\\
&\lec T \norm{(u,\td d)}_{L^\infty_TX}^2 \bke{\norm{(u,\td d)}_{L^\infty_TX} + 1}.
}
By \thref{lem-heat-est-Ya} again, we have that
\EQS{\label{eq-NLCF-N-Ya-nonlinear-2.3}
\norm{\nb F_2(u,\td d)(t)}_{Y_a}
&\lec \int_0^t (t-\tau)^{-\frac12} \norm{\bke{u\cdot\nb\td d + |\nb\td d|^2(\td d + d_\infty)}(\tau)}_{Y_{2a}} d\tau\\
&\lec t^{\frac12} \sup_{0\le t\le T} \bke{\norm{u(t)}_{Y_a}\norm{\nb\td d}_{Y_a} + \norm{\nb\td d}_{Y_a}^2\norm{\td d + d_\infty}_{L^\infty}}\\
&\lec T^{\frac12} \norm{(u,\td d)}_{L^\infty_TX}^2 \bke{1 + \norm{(u,\td d)}_{L^\infty_TX}}.
}
It follows from \eqref{eq-NLCF-N-Ya-linear-1}-\eqref{eq-NLCF-N-Ya-nonlinear-2.3} that, by $T\le 1$,
\EQS{\label{eq-NLCF-N-Ya-induction}
&\norm{(u^{m+1},\td d^{m+1})(t)}_X \le \norm{(e^{-tA}u_0, e^{t\De^N}\td d_0)}_X + \norm{F(u^m,\td d^m)(t)}_X\\
&\quad \le C_0 \norm{(u_0,\td d_0)}_X+ C_1T^{1/2} \norm{(u^m,\td d^m)}_{L^\infty_TX}^2 \bke{1 + \norm{(u^m,\td d^m)}_{L^\infty_TX}}.
}
Suppose $\norm{(u^m,\td d^m)}_{L^\infty_TX}\le 2\e_T$, where
\[
\e_T := C_0 \norm{(u_0,\td d_0)}_X.
\]
Then, by \eqref{eq-NLCF-N-Ya-induction} and the induction hypothesis, 
\[
\norm{(u^{m+1},\td d^{m+1})}_{L^\infty_TX} \le \e_T + 4C_1T^{\frac12}\e_T^2(1+2\e_T)
\le 2\e_T
\]
provided $T<\bke{1+4C_1\e_T(1+2\e_T)}^{-2}$, which can be achieved by taking $T$ sufficiently small.
Therefore, we have $\norm{(u^m,\td d^m)}_{L^\infty_TX}\le 2\e_T$ for all $m$.

By a similar computation as in \eqref{eq-NLCF-N-Ya-nonlinear-1} and in \eqref{eq-NLCF-N-Ya-nonlinear-2.1}-\eqref{eq-NLCF-N-Ya-nonlinear-2.3}, we get
\EQS{\label{eq-NLCF-N-Ya-nonlinear-diff-1}
&\norm{F_1(u,\td d)(t) - F_1(v,\td e)(t)}_{Y_a} \\
&\quad\lec T^{\frac12} \bke{\norm{(u,\td d)}_{L^\infty_TX} + \norm{(v,\td e)}_{L^\infty_TX}} \norm{(u,\td d) - (v,\td e)}_{L^\infty_TX},
}
and
\EQS{\label{eq-NLCF-N-Ya-nonlinear-diff-2}
&\norm{F_2(u,\td d)(t)-F_2(v,\td e)(t)}_{L^\infty}
 + \norm{\nb F_2(u,\td d)(t)-\nb F_2(v,\td e)(t)}_{Y_a}\\
&\quad\lec T^{\frac12} \bke{\norm{(u,\td d)}_{L^\infty_TX} + \norm{(v,\td e)}_{L^\infty_TX}}\\
&\quad\qquad \times\bke{1 + \norm{(u,\td d)}_{L^\infty_TX} + \norm{(v,\td e)}_{L^\infty_TX}}\norm{(u,\td d) - (v,\td e)}_{L^\infty_TX}.
}
Hence,
\EQN{
&\norm{(u^{m+1},\td d^{m+1})(t) - (u^m,\td d^m)(t)}_X
= \norm{F(u^m,\td d^m)(t) - F(u^{m-1},\td d^{m-1})(t)}_X\\
&\qquad\le C_2T^{\frac12}(4\e_T)(1+4\e_T)\norm{(u^m,\td d^m) - (u^{m-1},\td d^{m-1})}_{L^\infty_TX}.
}
By choosing $T$ even smaller such that $4C_2T^{1/2}\e_T(1+4\e_T)<1$, we find that $(u^m,\td d^m)$ converges to a limit $(u,\td d)\in L^\infty_TX$, which is the required solution. For the uniform-in-$t$ limit $\td d(x,t)  \to 0$ as $|x|\to \infty$, see Remark \ref{rem6.1}.
This proves \thref{mild-NLCF-N-Ya}. \qed

\begin{remark}\label{rem6.1}
To show $\td d(x,t)  \to 0$ as $|x|\to \infty$, denote 
$\td d_\ell(x,t)=I(x,t)+J(x,t)$ where $I$ and $J$ are
the two integrals in \eqref{d-mild-NLCF-N}.  For any small $\e>0$, there is $R>1$ such that $|\td d_0(y)|<\e$ for $|y|>R$. Then by \eqref{eq-heat-Green-est}, for $|x|>2R$,
\[
|I(x,t)| \le \int_{|y|<R} G^N(x,y,t)2dy +  \int_{|y|>R} G^N(x,y,t)\e dy \le 
 \int_{|y|<R} \frac{Cdy}{|x-y|^n} + 2 \e \le \frac {CR^n}{|x|^n}  + 2\e,
\]
which is less than $3\e$ for $|x|$ sufficiently large. For $J$, by $(u,\td d)\in L^\infty_T X$ and Lemma \ref{lem-heat-int-decay-est},
\[
|J(x,t)|\lec \int_0^t \int_{\R^n} \Ga(x-y,t-s) {\bka{y}^{-2a}} \, dyds \le C_T\bka{x}^{-2a}.
\]
Hence $\td d(x,t)  \to 0$ as $|x|\to \infty$, uniformly in $t \in (0,T)$.\hfill \qed
\end{remark}

\subsubsection{Construction of mild solution in $L^q$ for $q>n$}\label{sec:NLCF-N-Lq}
We now consider the problem in the $L^q$ framework and prove \thref{mild-NLCF-N-Lq}.
We first consider the subcritical case $q>n$. Let $\td d= d - d_\infty$. Define the Banach space
\EQS{\label{eq-banach-Lq}
X_T &= \Big\{(u,\td d)\in L^\infty(0,T;L^q\times L^\infty): 
\nb\td d \in L^\infty(0,T;L^q),\\ 
&\qquad t^{\frac{n}{2q}}u,\,t^{\frac{n}{2q}}\nb\td d \in L^\infty(0,T;L^\infty),\ 
t^{\frac12} \nb u,\, t^{\frac12}\nb^2\td d\in L^\infty(0,T;L^q)
\Big\} ,
}
with the norm
\EQN{
&\norm{(u,\td d)}_{X_T} 
:= \sup_{0\le t\le T} \bke{\norm{u(\cdot,t)}_{L^q} + \norm{\td d(\cdot,t)}_{L^\infty} + \norm{\nb \td d(\cdot,t)}_{L^q} }\\
&\qquad + \sup_{0\le t\le T}
\bket{ t^{\frac{n}{2q}} \bke{\norm{u(\cdot,t)}_{L^\infty} + \norm{\nb\td d(\cdot,t)}_{L^\infty}  } +  t^{\frac12} \bke{\norm{\nb u(\cdot,t)}_{L^q} + \norm{\nb^2 \td d(\cdot,t)}_{L^q} }}.
}

Applying \cite[Lemma 9.1]{Green} and Lemma \ref{lem-heat-est-Lq}, we derive
\EQN{
\norm{e^{-tA}u_0}_{L^q} \lec \norm{u_0}_{L^q},\quad
&\norm{e^{-tA}u_0}_{L^\infty} \lec t^{-\frac{n}{2q}} \norm{u_0}_{L^q},\quad
\norm{\nb e^{-tA}u_0}_{L^q} \lec t^{-\frac12} \norm{u_0}_{L^q},\\
\norm{e^{t\De^N}\td d_0}_{L^\infty} \lec \norm{\td d_0}_{L^\infty},\
&\norm{\nb e^{t\De^N}\td d_0}_{L^q} \lec \norm{\nb\td d_0}_{L^q},\
\norm{\nb e^{t\De^N}\td d_0}_{L^\infty} \lec t^{-\frac{n}{2q}} \norm{\nb\td d_0}_{L^q},\\
&\norm{\nb^2 e^{t\De^N}\td d_0}_{L^q} \lec t^{-\frac12} \norm{\nb\td d_0}_{L^q}.
}
where the formula \eqref{eq-nb-d-linear} for $\nb e^{t\De^N}\td d_0$ has been used.
These estimates yield
\EQ{\label{eq-Lq-mild-linear}
\norm{(e^{-tA}u_0, e^{t\De^N}\td d_0)}_{X_T} \lec \norm{u_0}_{L^q} + \norm{\td d_0}_{L^\infty} + \norm{\nb d_0}_{L^q}.
}

Let $F = (F_1,F_2)$, where $F_1$ and $F_2$ are defined as in \eqref{eq-F1-def} and \eqref{eq-F2-def}.
Note that
\EQN{
(F_1(u,\td d))_i(x,t) &=  \int_0^t\int_{\R^n_+} \pd_{y_k}G_{ij}(x,y,t-s)\bke{u_ku_j + \pd_k\td d_\ell \pd_j\td d_\ell}(y,s)\, dyds\\
&= -\int_0^t\int_{\R^n_+} G_{ij}(x,y,t-s) \pd_k\bke{u_ku_j + \pd_k\td d_\ell \pd_j\td d_\ell}(y,s)\, dyds,
}
implying
\[
\nb_x (F_1(u,\td d))_i(x,t) = -\int_0^t\int_{\R^n_+} \nb_xG_{ij}(x,y,t-s) \pd_k\bke{u_ku_j + \pd_k\td d_\ell \pd_j\td d_\ell}(y,s)\, dyds.
\]
Moreover, as $\pd_{x_j}G^N(x,y,t) = -\pd_{y_j}G^N(x,y,t)$ for $j<n$ and $\pd_{x_n}G^N(x,y,t) = -\pd_{y_n}G^D(x,y,t)$, one has
\EQN{
\pd_{x_j}(F_2&(u,\td d))_\ell(x,t) 
= \int_0^t\int_{\R^n_+} \pd_{x_j}G^N(x,y,t-s)\bkt{ -(u\cdot\nb)\td d_\ell + |\nb \td d|^2 (\td d + d_\infty)_\ell}(y,s)\, dyds\\
&= 
\begin{cases} 
\displaystyle
\int_0^t\int_{\R^n_+} G^N(x,y,t-s)\, \pd_j\bkt{ -(u\cdot\nb)\td d_\ell + |\nb \td d|^2 (\td d + d_\infty)_\ell}(y,s)\, dyds,\ j<n,\\
\displaystyle
\int_0^t\int_{\R^n_+} G^D(x,y,t-s)\, \pd_j\bkt{ -(u\cdot\nb)\td d_\ell + |\nb \td d|^2 (\td d + d_\infty)_\ell}(y,s)\, dyds,\ j=n.
\end{cases}
}
Applying \cite[Lemma 9.1]{Green} and Lemma \ref{lem-heat-est-Lq} again, we derive the following estimates
\EQN{
\norm{F_1(u,\td d)(t)}_{L^q} 
&\lec \int_0^t (t-\tau)^{-\frac12-\frac{n}{2q}} \norm{ \bke{u\otimes u + \nb\td d\odot\nb\td d}(\tau)}_{L^{\frac{q}2}} d\tau\\
&\le \int_0^t (t-\tau)^{-\frac12-\frac{n}{2q}} \bkt{ \norm{u(\tau)}_{L^q}^2 + \norm{\nb\td d(\tau)}_{L^q}^2 } d\tau
\lec t^{\frac12-\frac{n}{2q}} \norm{(u,\td d)}_{X_T}^2,
}
\EQN{
&\norm{F_1  (u,\td d)(t)}_{L^\infty} 
\lec \int_0^t (t-\tau)^{-\frac12-\frac{n}{2q}} \norm{ \bke{u\otimes u + \nb\td d\odot\nb\td d}(\tau)}_{L^q} d\tau\\
&\quad \le \int_0^t (t-\tau)^{-\frac12-\frac{n}{2q}} \tau^{-\frac{n}{2q}} \bkt{ \tau^{\frac{n}{2q}} \norm{u(\tau)}_{L^\infty} \norm{u(\tau)}_{L^q} + \tau^{\frac{n}{2q}} \norm{\nb\td d(\tau)}_{L^\infty} \norm{{\nb\td d}(\tau)}_{L^q} } d\tau\\
&\quad \lec t^{\frac12-\frac{n}q} \norm{(u,\td d)}_{X_T}^2,
}
\EQN{
&\norm{\nb F_1(u,\td d)(t)}_{L^q} 
\lec \int_0^t (t-\tau)^{-\frac12} \bkt{ \norm{u\cdot\nb u}_{L^q} + \norm{\De\td d\cdot\nb\td d}_{L^q} + \norm{\nb\td d\cdot\nb^2\td d}_{L^q} }(\tau)\, d\tau\\
& \quad \lec \int_0^t (t-\tau)^{-\frac12} \tau^{-\frac12-\frac{n}{2q}} \Big[
\tau^{\frac{n}{2q}}\norm{u(\tau)}_{L^\infty} \tau^{\frac12} \norm{\nb u(\tau)}_{L^q} +
\tau^{\frac{n}{2q}} \norm{\nb\td d(\tau)}_{L^\infty} \tau^{\frac12}  \norm{\nb^2\td d(\tau)}_{L^q} \Big] d\tau \\
&\quad \lec t^{-\frac{n}{2q}} \norm{(u,\td d)}_{X_T}^2
}
for $F_1$. For $F_2$,
\begin{align}
\nonumber
\norm{F_2(u,\td d)(t)}_{L^\infty} 
&\lec \int_0^t (t-\tau)^{-\frac{n}q} \norm{\bke{u\cdot\nb\td d+|\nb\td d|^2(\td d + d_\infty)}(\tau)}_{L^{\frac{q}2}} d\tau\\ \nonumber
&\le \int_0^t (t-\tau)^{-\frac{n}q} \bkt{\norm{u(\tau)}_{L^q}\norm{\nb\td d(\tau)}_{L^q} + \norm{\nb\td d(\tau)}_{L^q}^2\bke{\norm{\td d(\tau)}_{L^\infty} + 1} } d\tau\\
&\lec t^{1-\frac{n}q} \norm{(u,\td d)}_{X_T}^2 \bke{\norm{(u,\td d)}_{X_T} + 1},
\label{F2-Loo}
\end{align}
\begin{align}
\nonumber
\norm{\nb F_2(u,\td d)(t)}_{L^q}
&\lec \int_0^t (t-\tau)^{-\frac12-\frac{n}{2q}} \norm{\bke{u\cdot\nb\td d+|\nb\td d|^2(\td d + d_\infty)}(\tau)}_{L^{\frac{q}2}} d\tau\\ \nonumber
&\lec \int_0^t (t-\tau)^{-\frac12-\frac{n}{2q}} \bkt{\norm{u(\tau)}_{L^q}\norm{\nb\td d(\tau)}_{L^q} + \norm{\nb\td d(\tau)}_{L^q}^2\bke{\norm{\td d(\tau)}_{L^\infty} + 1} } d\tau\\
&\le t^{\frac12-\frac{n}{2q}} \norm{(u,\td d)}_{X_T}^2 \bke{\norm{(u,\td d)}_{X_T} + 1},
\label{F2-Lq}
\end{align}
\EQN{
&\norm{\nb F_2(u,\td d)(t)}_{L^\infty}
\lec \int_0^t (t-\tau)^{-\frac12-\frac{n}{2q}} \norm{\bke{u\cdot\nb\td d+|\nb\td d|^2(\td d + d_\infty)}(\tau)}_{L^q} d\tau\\
&\ \lec \int_0^t (t-\tau)^{-\frac12-\frac{n}{2q}} \tau^{-\frac{n}{2q}} \bke{\tau^{\frac{n}{2q}}\norm{\nb\td d(\tau)}_{L^\infty}} \bkt{\norm{u(\tau)}_{L^q}  + \norm{\nb\td d(\tau)}_{L^q}  \bke{\norm{\td d(\tau)}_{L^\infty} + 1} } d\tau\\
&\ \lec t^{\frac12-\frac{n}{q}} \norm{(u,\td d)}_{X_T}^2 \bke{\norm{(u,\td d)}_{X_T} + 1},
}
and, by using 
$
\nb\bkt{ -(u\cdot\nb)\td d + |\nb \td d|^2 (\td d + d_\infty)} \sim \nb u\cdot\nb\td d + u\cdot\nb^2\td d + \nb\td d\cdot\nb^2\td d(\td d+d_\infty) + |\nb\td d|^2\nb\td d$,
we have
\EQN{
&\norm{\nb^2 F_2(u,\td d)(t)}_{L^q}\\
&\ \lec \int_0^t (t-\tau)^{-\frac12} \bkt{\norm{\nb u\cdot\nb\td d}_{L^q} + \norm{u\cdot\nb^2\td d}_{L^q} + \norm{(\nb\td d\cdot\nb^2\td d)(\td d + d_\infty)}_{L^q} + \norm{|\nb\td d|^3}_{L^q} }(\tau)\, d\tau\\
&\ \le \int_0^t (t-\tau)^{-\frac12} \tau^{-\frac12-\frac{n}{2q}} \Bigg[ \tau^{\frac12}\norm{\nb u(\tau)}_{L^q} \tau^{\frac{n}{2q}} \norm{\nb\td d(\tau)}_{L^\infty} + \tau^{\frac{n}{2q}}\norm{u(\tau)}_{L^\infty} \tau^{\frac12} \norm{\nb^2\td d(\tau)}_{L^\infty} \\
&\qquad\qquad\qquad\qquad\qquad\quad  + \tau^{\frac{n}{2q}}\norm{\nb\td d(\tau)}_{L^\infty} \tau^{\frac12}\norm{\nb^2\td d(\tau)}_{L^q} \bke{\norm{\td d(\tau)}_{L^\infty} + 1} \Bigg]\, d\tau\\
&\qquad + \int_0^t (t-\tau)^{-\frac12} \tau^{-\frac{n}q} \bke{\tau^{\frac{n}{2q}} \norm{\nb\td d(\tau)}_{L^\infty}}^2 \norm{\nb\td d(\tau)}_{L^q}\, d\tau\\
&\ \lec \bke{t^{-\frac{n}{2q}} + t^{\frac12-\frac{n}{q}}} \norm{(u,\td d)}_{X_T}^2 \bke{\norm{(u,\td d)}_{X_T} + 1}.
} 
Note that we've used the assumption $q>n$.
The above estimates imply
\EQ{\label{eq-Lq-mild-nonlinear}
\norm{F(u,\td d)}_{X_T} \lec \bke{T^{\frac12-\frac{n}{2q}} + T^{1-\frac{n}{q}} } \norm{(u,\td d)}_{X_T}^2 \bke{1 + \norm{(u,\td d)}_{X_T}}.
}

Let $\{(u^m,\td d^m)\}_m$ be the sequence defined by the iteration scheme \eqref{eq-NLCF-iteration}.
It follows from \eqref{eq-Lq-mild-linear} and \eqref{eq-Lq-mild-nonlinear} that
\EQS{\label{eq-NLCF-N-Lq-induction}
\norm{(u^{m+1},\td d^{m+1})}_{X_T} &\le \norm{(e^{-tA}u_0, e^{t\De^N}\td d_0)}_{X_T} + \norm{F(u^m,\td d^m)}_{X_T}\\
& \le C_0 \bke{ \norm{u_0}_{L^q} + \norm{\td d_0}_{L^\infty} + \norm{\nb d_0}_{L^q} }\\
&\quad + C_1\bke{T^{\frac12-\frac{n}{2q}} + T^{1-\frac{n}{q}} } \norm{(u^m,\td d^m)}_{X_T}^2 \bke{1 + \norm{(u^m,\td d^m)}_{X_T}}.
}
Suppose $\norm{(u^m,\td d^m)}_{X_T}\le 2\e$, where 
\EQ{\label{eq-NLCF-data-size}
\e := C_0\bke{ \norm{u_0}_{L^q} + \norm{\td d_0}_{L^\infty} + \norm{\nb d_0}_{L^q} }.
}
Then, by \eqref{eq-NLCF-N-Lq-induction} and the induction hypothesis, 
\[
\norm{(u^{m+1},\td d^{m+1})}_{X_T} \le \e + 4C_1T^{\frac12-\frac{n}{2q}}\e^2(1+2\e)
\le 2\e
\]
provided $T^{\frac12-\frac{n}{2q}}<\min(1,4C_1\e(1+2\e))$.
Therefore, we have $\norm{(u^m,\td d^m)}_{X_T}\le 2\e$ for all $m$.

By a similar computation as above, we obtain the difference estimate
\EQ{\label{eq-diff-XT}
\norm{F(u,\td d) - F(v,\td e)}_{X_T}
 \lec T^{\frac12-\frac{n}{2q}}  \bke{1 + \norm{(u,\td d)}_{X_T} + \norm{(v,\td e)}_{X_T}} ^2
\norm{(u,\td d) - (v,\td e)}_{X_T}.
}
Hence,
\EQN{
&\norm{(u^{m+1},\td d^{m+1}) - (u^m,\td d^m)}_{X_T} = \norm{F(u^m,\td d^m) - F(u^{m-1},\td d^{m-1})}_{X_T}\\
&\quad \le C_2 T^{\frac12-\frac{n}{2q}} \bke{1 + \norm{(u^m,\td d^m)}_{X_T} + \norm{(u^{m-1},\td d^{m-1})}_{X_T}} ^2
\norm{(u^m,\td d^m) - (u^{m-1},\td d^{m-1})}_{X_T}\\
&\quad \le C_2 T^{\frac12-\frac{n}{2q}} (1+4\ve)^2 \norm{(u^m,\td d^m) - (u^{m-1},\td d^{m-1})}_{X_T}.
}
By choosing $T$ even smaller such that $C_2 T^{\frac12-\frac{n}{2q}} (1+4\ve)^2 <1$, we find that $(u^m,\td d^m)$ converges to a limit $(u,\td d)\in X_T$, which is the desired solution.
For the uniform-in-$t$ limit $\td d(x,t)  \to 0$ as $|x|\to \infty$, see Remark \ref{rem6.3}.
This completes the proof of \thref{mild-NLCF-N-Lq} for $q>n$.

\subsubsection{Construction of mild solutions in $L^n$}\label{sec:NLCF-N-Ln}
We now consider the critical case $q=n$ of \thref{mild-NLCF-N-Lq}.
Define 
\EQN{
X_T &= \Big\{(u,\td d)\in L^\infty(0,T;L^n\times L^\infty): 
\nb\td d \in L^\infty(0,T;L^n),\\ 
&\qquad t^{\frac12}u,\,t^{\frac12}\nb\td d \in L^\infty(0,T;L^\infty),\ 
t^{\frac12} \nb u,\, t^{\frac12}\nb^2\td d\in L^\infty(0,T;L^n)
\Big\},\\
Y_T &= \Big\{(u,\td d): 
\td d \in L^\infty(0,T;L^\infty),\, t^{\frac14}u,\, t^{\frac14}\nb\td d\in L^\infty(0,T;L^{2n})
\Big\}
}
with the norm
\EQN{
\norm{(u,\td d)}_{X_T} 
&:= \sup_{0\le t\le T} \bke{\norm{u(\cdot,t)}_{L^n} + \norm{\td d(\cdot,t)}_{L^\infty} + \norm{\nb d(\cdot,t)}_{L^n} }\\
&\quad + \sup_{0\le t\le T} t^{\frac12} \bke{\norm{u(\cdot,t)}_{L^\infty} + \norm{\nb\td d(\cdot,t)}_{L^\infty}  + \norm{\nb u(\cdot,t)}_{L^n } + \norm{\nb^2 d(\cdot,t)}_{L^n} },\\
\norm{(u,\td d)}_{Y_T} &:= \sup_{0\le t\le T} t^{\frac14} \bke{\norm{u(\cdot,t)}_{L^{2n}} + \norm{\nb d(\cdot,t)}_{L^{2n}} }.
}
By $t^{1/4}\norm{f}_{L^{2n}} \le (t^{1/2}\norm{f}_{L^\infty})^{1/2} \norm{f}_{L^n}^{1/2}$,
it is readily seen that $X_T\subset Y_T$ and $\norm{(u,\td d)}_{Y_T}\le \norm{(u,\td d)}_{X_T}$.

Similar to \eqref{eq-Lq-mild-linear}, we have
 \EQ{\label{eq-Lq-mild-linear-critical-XT}
\norm{(e^{-tA}u_0, e^{t\De^N}\td d_0)}_{X_T} \lec \norm{u_0}_{L^n} + \norm{\td d_0}_{L^\infty} + \norm{\nb d_0}_{L^n}.
}
By $X_T\subset Y_T$, the $Y_T$-norm of $(e^{-tA}u_0, e^{t\De^N}\td d_0)$ is bounded by the same right side.
Since $C^\infty_c$ is dense in $L^n$, 
by a density argument, one can show that
\EQ{\label{eq-Lq-mild-linear-critical-YT}
\norm{(e^{-tA}u_0, e^{t\De^N}\td d_0)}_{Y_T} = \sup_{0\le t\le T} t^{\frac14} \bke{\norm{e^{-tA}u_0}_{L^{2n}} + \norm{\nb e^{t\De^N}\td d_0}_{L^{2n}} } \to 0,\ \text{ as }T\to0.
}

Recall $F = (F_1,F_2)$, where $F_1$ and $F_2$ are defined as in \eqref{eq-F1-def} and \eqref{eq-F2-def}.
Using \cite[Lemma 9.1]{Green} and Lemma \ref{lem-heat-est-Lq}, we have for $F_1$
\EQN{
\norm{F_1(u,\td d)(t)}_{L^n} 
&\lec \int_0^t (t-\tau)^{-\frac12-\frac14} \norm{\bke{u\otimes u+\nb\td d\odot\nb\td d}(\tau)}_{L^{\frac{2n}3}} d\tau\\
&\le \int_0^t (t-\tau)^{-\frac34} \tau^{-\frac14} \bkt{\tau^{\frac14}\norm{u(\tau)}_{L^{2n}}\norm{u(\tau)}_{L^n} + \tau^{\frac14}\norm{\nb\td d(\tau)}_{L^{2n}}\norm{\nb\td d(\tau)}_{L^n} } d\tau\\
&\lec  \norm{(u,\td d)}_{X_T} \norm{(u,\td d)}_{Y_T},
}
\EQN{
&\norm{F_1(u,\td d)(t)}_{L^\infty} 
\lec \int_0^t (t-\tau)^{-\frac12-\frac14} \norm{\bke{u\otimes u+\nb\td d\odot\nb\td d}(\tau)}_{L^{2n}} d\tau\\
&\quad\le \int_0^t (t-\tau)^{-\frac34} \tau^{-\frac34} \bkt{ \tau^{\frac12}\norm{u(\tau)}_{L^\infty} \tau^{\frac14}\norm{u(\tau)}_{L^{2n}}  + \tau^{\frac12}\norm{\nb\td d(\tau)}_{L^\infty}\tau^{\frac14}\norm{\nb\td d(\tau)}_{L^{2n}}  }  d\tau\\
&\quad\lec t^{-\frac12} \norm{(u,\td d)}_{X_T} \norm{(u,\td d)}_{Y_T},
}
\EQN{
&\norm{\nb F_1(u,\td d)(t)}_{L^n} 
\lec \int_0^t (t-\tau)^{-\frac12-\frac14} \norm{\bke{u\cdot\nb u+\nb\td d\cdot\nb^2\td d}(\tau)}_{L^{\frac{2n}3}} d\tau\\
&\quad\le \int_0^t (t-\tau)^{-\frac34} \tau^{-\frac34} \bkt{\tau^{\frac14}\norm{u(\tau)}_{L^{2n}}{\tau^{\frac12}\norm{\nb u(\tau)}_{L^n} } + {\tau^{\frac14}\norm{\nb\td d(\tau)}_{L^{2n}}} {\tau^{\frac12}\norm{\nb^2\td d(\tau)}_{L^n} } }  d\tau\\
&\quad\lec t^{-\frac12} \norm{(u,\td d)}_{X_T} \norm{(u,\td d)}_{Y_T},
}
and we have for $F_2$
\begin{align}
\nonumber
&\norm{F_2(u,\td d)(t)}_{L^\infty} 
\lec \int_0^t (t-\tau)^{-\frac14} \norm{ {u\cdot\nb\td d + |\nb\td d|^2(\td d+d_\infty)}(\tau)}_{L^{2n}} d\tau\\ \nonumber
&\ \le \int_0^t (t-\tau)^{-\frac14} \tau^{-\frac34} 
{\tau^{\frac12}\norm{\nb\td d(\tau)}_{L^\infty} } 
\bkt{  {\tau^{\frac14}\norm{u(\tau)}_{L^{2n}}}   +  {\tau^{\frac14}\norm{\nb\td d(\tau)}_{L^{2n}}}   \bke{\norm{\td d(\tau)}_{L^\infty} + 1} } d\tau\\
&\ \lec \norm{(u,\td d)}_{X_T} \norm{(u,\td d)}_{Y_T}  \bke{ \norm{(u,\td d)}_{X_T} + 1},
\label{F2-Loo-q=n}
\end{align}
\begin{align}
\nonumber
&\norm{\nb F_2(u,\td d)(t)}_{L^n} 
\lec \int_0^t (t-\tau)^{-\frac12-\frac14} \norm{ {u\cdot\nb\td d + |\nb\td d|^2(\td d+d_\infty)}(\tau)}_{L^{\frac{2n}3}} d\tau\\ \nonumber
&\ \le \int_0^t (t-\tau)^{-\frac34} \tau^{-\frac14} 
{\norm{\nb\td d(\tau)}_{L^n} }
\bkt{  {\tau^\frac14\norm{u(\tau)}_{L^{2n}}} +  {\tau^\frac14\norm{\nb\td d(\tau)}_{L^{2n}}}  \bke{\norm{\td d(\tau)}_{L^\infty} + 1} } d\tau\\ 
&\ \lec \norm{(u,\td d)}_{X_T} \norm{(u,\td d)}_{Y_T}  \bke{ \norm{(u,\td d)}_{X_T} + 1},
\label{F2-Ln}
\end{align}
\EQN{
&\norm{\nb F_2(u,\td d)(t)}_{L^\infty} 
\lec \int_0^t (t-\tau)^{-\frac34} \norm{ {u\cdot\nb\td d + |\nb\td d|^2(\td d+d_\infty)}(\tau)}_{L^{2n}} d\tau\\
&\ \le \int_0^t (t-\tau)^{-\frac34} \tau^{-\frac34} 
{\tau^{\frac14}\norm{\nb\td d(\tau)}_{L^{2n}}}
\bkt{  {\tau^{\frac12}\norm{u(\tau)}_{L^\infty}}   +  {\tau^{\frac12}\norm{\nb\td d(\tau)}_{L^\infty}}    \bke{\norm{\td d(\tau)}_{L^\infty} + 1} } d\tau\\
&\ \lec t^{-\frac12} \norm{(u,\td d)}_{X_T} \norm{(u,\td d)}_{Y_T}  \bke{ \norm{(u,\td d)}_{X_T} + 1},
}
and using $\nb\bkt{ -(u\cdot\nb)\td d + |\nb \td d|^2 (\td d + d_\infty)} \sim \nb u\cdot\nb\td d + u\cdot\nb^2\td d + \nb\td d\cdot\nb^2\td d(\td d+d_\infty) + |\nb\td d|^2\nb\td d$,
\EQN{
&\norm{\nb^2 F_2(u,\td d)(t)}_{L^n} \lec \int_0^t  (t-\tau)^{-\frac34} \cdot
\\
&\hspace{12mm}\cdot \bkt{\norm{\nb u\cdot\nb\td d}_{L^{\frac{2n}3}} + \norm{u\cdot\nb^2\td d}_{L^{\frac{2n}3}} + \norm{(\nb\td d\cdot\nb^2\td d)(\td d + d_\infty)}_{L^{\frac{2n}3}} + \norm{|\nb\td d|^3}_{L^{\frac{2n}3}} }(\tau)\, d\tau\\
&\ \le \int_0^t (t-\tau)^{-\frac34} \tau^{-\frac34} \Bigg[  {\tau^{\frac12}\norm{\nb u(\tau)}_{L^n}}  {\tau^{\frac14}\norm{\nb\td d(\tau)}_{L^{2n}}} + 
 {\tau^{\frac14}\norm{u(\tau)}_{L^{2n}}}  {\tau^{\frac12}\norm{\nb^2\td d(\tau)}_{L^n}}\\
&\qquad\qquad\qquad\qquad +
 {\tau^{\frac14}\norm{\nb\td d(\tau)}_{L^{2n}}}  {\tau^{\frac12}\norm{\nb^2\td d(\tau)}_{L^n} }  \bke{\norm{\td d(\tau)}_{L^\infty} + 1} +
 \bke{\tau^{\frac14}\norm{\nb\td d}_{L^{2n}}}^3 \Bigg] d\tau\\
&\lec t^{-\frac12} \norm{(u,\td d)}_{X_T} \norm{(u,\td d)}_{Y_T}  \bke{ \norm{(u,\td d)}_{X_T} + 1},
}
using $\norm{(u,\td d)}_{Y_T}\le \norm{(u,\td d)}_{X_T}$.
The above estimates yield
\EQ{\label{eq-Lq-mild-nonlinear-critical-XT}
\norm{F(u,\td d)}_{X_T} \lec \norm{(u,\td d)}_{X_T} \norm{(u,\td d)}_{Y_T} \bke{\norm{(u,\td d)}_{X_T} + 1}.
}

Using \cite[Lemma 9.1]{Green} and Lemma \ref{lem-heat-est-Lq} again,
\EQN{
\norm{F_1(u,\td d)(t)}_{L^{2n}} 
&\lec \int_0^t (t-\tau)^{-\frac12-\frac14} \norm{\bke{u\otimes u+\nb\td d\odot\nb\td d}(\tau)}_{L^n} d\tau\\
&\le \int_0^t (t-\tau)^{-\frac34} \tau^{-\frac12} \bkt{ \bke{\tau^{\frac14}\norm{u(\tau)}_{L^{2n}}}^2 + \bke{\tau^{\frac14}\norm{\nb\td d(\tau)}_{L^{2n}}}^2} d\tau\\
&\lec t^{-\frac14} \norm{(u,\td d)}_{Y_T}^2,
}
and
\EQN{
&\norm{\nb F_2(u,\td d)(t)}_{L^{2n}}
\lec \int_0^t (t-\tau)^{-\frac12-\frac14} \norm{\bke{u\cdot\nb\td d + |\nb\td d|^2 (\td d + d_\infty)}(\tau)}_{L^n} d\tau\\
&\ \le \int_0^t (t-\tau)^{-\frac34} \tau^{-\frac12} \bkt{ {\tau^{\frac14} \norm{u(\tau)}_{L^{2n}}} {\tau^{\frac14} \norm{\nb\td d(\tau)}_{L^{2n}}} + \bke{\tau^{\frac14} \norm{\nb\td d(\tau)}_{L^{2n}} }^2 \bke{\norm{\td d(\tau)}_{L^\infty} + 1 } } d\tau\\
&\ \lec t^{-\frac14} \norm{(u,\td d)}_{Y_T}^2 \bke{\norm{(u,\td d)}_{X_T} + 1}.
}
The above two estimates imply
\EQ{\label{eq-Lq-mild-nonlinear-critical-YT}
\norm{F(u,\td d)}_{Y_T} \lec \norm{(u,\td d)}_{Y_T}^2 \bke{\norm{(u,\td d)}_{X_T} + 1}.
}

Recall that $(u^m,\td d^m)$ is defined in \eqref{eq-NLCF-iteration}.
Using the same induction argument below \eqref{eq-NLCF-N-Lq-induction} with the estimates \eqref{eq-Lq-mild-linear-critical-XT}, \eqref{eq-Lq-mild-linear-critical-YT}, \eqref{eq-Lq-mild-nonlinear-critical-XT}, \eqref{eq-Lq-mild-nonlinear-critical-YT}, we have that for sufficiently small $T$ we have
\[
\norm{(u^m,\td d^m)}_{X_T} \le 2\ve,\quad
\norm{(u^m,\td d^m)}_{Y_T} \le 2\ve,
\]
where $\ve$ is given in \eqref{eq-NLCF-data-size}.
More importantly, due to the convergence \eqref{eq-Lq-mild-linear-critical-YT} of the $Y_T$ norm, as $T\to0$, the above can be improved as
\[
\norm{(u^m,\td d^m)}_{Y_T} \le \ve_1(T),
\]
where $\ve_1(T)\to0$ as $T\to0$.

By a similar computation yielding \eqref{eq-Lq-mild-nonlinear-critical-YT}, we obtain the difference estimate
\EQN{
&\norm{F(u,\td d) - F(v,\td e)}_{Y_T}
\\
&\quad\lec \bke{\norm{(u,\td d)}_{Y_T} + \norm{(v,\td e)}_{Y_T} } \bke{\norm{(u,\td d)}_{X_T} + \norm{(v,\td e)}_{X_T} + 1} \norm{(u,\td d) - (v,\td e)}_{Y_T}.
}
Hence,
\EQN{
&\norm{(u^{m+1},\td d^{m+1}) - (u^m,\td d^m)}_{Y_T}
 = \norm{F(u^m,\td d^m) - F(u^{m-1},\td d^{m-1})}_{Y_T}\\
&\quad \le C_2 \bke{\norm{(u^m,\td d^m)}_{Y_T} + \norm{(u^{m-1},\td d^{m-1})}_{Y_T}} \bke{\norm{(u^m,\td d^m)}_{X_T} + \norm{(u^{m-1},\td d^{m-1})}_{X_T} + 1} \\
&\qquad\cdot \norm{(u^m,\td d^m) - (u^{m-1},\td d^{m-1})}_{Y_T}\\
&\quad \le C_2  (4\ve_1(T))(4\ve + 1) \norm{(u^m,\td d^m) - (u^{m-1},\td d^{m-1})}_{Y_T}.
}
Therefore, for by choosing sufficiently small $T$ such that $4C_2\ve_1(T)(4\ve+1)<1$, we get that $(u^m,\td d^m)$ converges to a limit $(u,\td d)\in Y_T$, which is the desired solution. As $(u^m,\td d^m)$ are uniformly bounded in $X_T$, $(u,\td d)$ has the same bound in $X_T$.
For the uniform-in-$t$ limit $\td d(x,t)  \to 0$ as $|x|\to \infty$, see Remark \ref{rem6.3}.
This completes the proof of \thref{mild-NLCF-N-Lq} for $q=n$.
\qed

\begin{remark}\label{XT-converge}
In the proof of \thref{mild-NLCF-N-Lq} for $q=n$ above, the sequence $\{(u^m,\td d^m)\}$ also converges to $(u,\td d)$ in $X_T$.
In fact, a similar calculation deriving \eqref{eq-Lq-mild-nonlinear-critical-XT} yields the difference estimate in $X_T$:
\EQN{
&\norm{F(u,\td d) - F(v,\td e)}_{X_T}\\
&\qquad\lec \bke{\norm{(u,\td d)}_{X_T} + \norm{(v,\td e)}_{X_T} } \bke{\norm{(u,\td d)}_{X_T} + \norm{(v,\td e)}_{X_T} + 1} \norm{(u,\td d) - (v,\td e)}_{Y_T}.
}
Thus,
\EQN{
&\norm{(u^{m+1},\td d^{m+1}) - (u^m,\td d^m)}_{X_T} 
= \norm{F(u^m,\td d^m) - F(u^{m-1},\td d^{m-1}) }_{X_T}\\
&\lec \bke{\norm{(u^m,\td d^m)}_{X_T} + \norm{(u^{m-1},\td d^{m-1})}_{X_T} + 1}^2 \cdot \norm{(u^m,\td d^m) - (u^{m-1},\td d^{m-1})}_{Y_T}\\
&\lec (4\ve+1)^2 \norm{(u^m,\td d^m) - (u^{m-1},\td d^{m-1})}_{Y_T}.
}
This implies that $\{(u^m,\td d^m)\}$ is a Cauchy sequence in $X_T$ and hence converges to $(u,\td d)$ in $X_T$.
This fact is useful for Remark \ref{rem6.3}.\hfill \qed
\end{remark}

\begin{remark}\label{rem6.3}
To show $\td d(x,t)  \to 0$ as $|x|\to \infty$ for mild solutions in $L^q$, denote 
$\td d_\ell(x,t)=I(x,t)+J(x,t)$ where $I$ and $J$ are
the two integrals in \eqref{d-mild-NLCF-N}. The same argument of Remark \ref{rem6.1} shows $I(x,t)  \to 0$ as $|x|\to \infty$, uniformly in $t\in(0,T)$. Consider now
\[
J(x,t)=\int_0^t \int_{\R^n_+} G^N(x,y,t-s)f_\ell(u,\nb \td d)(y,s)\, dyds, \quad
f=
\bkt{ -(u\cdot\nb)\td d_\ell + |\nb \td d|^2 ({d_\infty}+\td d_\ell)}.
\]
Let $\{(u^m,\td d^m)\}_m$ be the sequence defined by the iteration scheme \eqref{eq-NLCF-iteration}. We will show by induction
\begin{equation}\label{LooLq-decay}
\lim_{R \to \infty} \sup_{0<t<T} \int_{|x|>R} | {\nb \td d}^m(x,t)|^q \,dx=0.
\end{equation}
For $m=1$, $\td d^1(t)=e^{t\De^N}\td d_0$, by \eqref{eq-nb-d-linear},
\[
\pd_j \td d^1(x,t)=\int_{\R^n_+} \bket{\Ga(x-y,t) + \ep_j \Ga(x-y^*,t)} \pd_j\td d_0(y)\, dy = \int_{|y|<L} + \int_{|y|>L} =I_1+I_2 .
\]
Let $1 \ll L < |x|/2$.
We have pointwise bound for $I_1$,
\[
|I_1| \le \int_{|y|<L} C |x|^{-n} |\nb \td d_0|(y)dy \le C(L) \norm{\nb\td d_0}_{L^q} |x|^{-n}.
\]
By Lemma \ref{lem-heat-est-Lq} with $f(y) = \pd_j\td d_0(y) \one_{|y|>L}(y)$, 
\[
\norm{I_2}_{L^q} \le C  \norm{\nb\td d_0}_{L^q(|y|>L)} \to 0,\quad \text{as}\quad L \to \infty.
\]
For any small $\e>0$ we first choose $L$ to make $\norm{I_2}_{L^\infty_T  L^q}\le \e$ and then take $R>2L$ sufficiently large to make  $\norm{I_1}_{L^\infty_T L^q(|x|>R)}\le \e$. This shows  \eqref{LooLq-decay} for $m=1$.

Suppose \eqref{LooLq-decay} is valid for $\nb{\td d}^{m}(x,t)$ and we now consider $m+1$.  Denote $f^m=f(u^m,\nb \td d^m)$. We have
\[
\pd_j (\td d^{m+1}-\td d^1)  (x,t) 
=\bke{\int_0^t \!\!\int_{|y|<L} + \int_0^t\!\!\int_{|y|>L}} \pd_{x_j}G^N(x,y,t-s) f^m(y,s) \, dyds =:I_3+I_4 .
\]
Let $1 \ll L < |x|/2$.
We have pointwise bound for $I_3$,
\[
|I_3| \le \int_0^t \!\!\int_{|y|<L} C |x|^{-n-1} |f^m|(y,s)dy ds\le C(L) T\norm{f^m}_{L^\infty_T L^{q/2}} |x|^{-n-1}.
\]
By \eqref{F2-Lq} for $q>n$ and \eqref{F2-Ln} for $q=n$, 
\EQN{
\norm{I_4}_{L^\infty_T L^q} &\le C_T  \norm{f^m}_{L^\infty_T L^{q/2}(|y|>L)} \to 0,\quad \text{as}\quad L \to \infty,\ \text{for $q>n$, and }\\
\norm{I_4}_{L^\infty_T L^n} & \le C_T  \norm{f^m}_{L^\infty_T L^{2n/3}(|y|>L)} \to 0,\quad \text{as}\quad L \to \infty,
}
thanks to \eqref{LooLq-decay}. For any small $\e>0$ we first choose $L$ to make $\norm{I_4}_{L^\infty_T L^q}\le \e$ and then take $R>2L$ sufficiently large to make  $\norm{I_3}_{L^\infty_T L^q(|x|>R)}\le \e$. This shows  \eqref{LooLq-decay} for general $m\in \NN$.

As $\td d^m$ converges to $\td d$ in $L^\infty_{t,x}$ see Remark \ref{XT-converge} for $q=n$ case),
for any small $\e>0$ we can choose $m$ so that $\norm{\td d^m-\td d}_{L^\infty_{t,x}} \le \e$.
Decompose the $J$ function for $(u^m,\td d^m)$ as
\[
J^m(x,t)=\bke{\int_0^t \!\!\int_{|y|<L} + \int_0^t\!\!\int_{|y|>L}} G^N(x-y,t-s) f^m(y,s) \, dyds =:I_5+I_6 .
\]
By \eqref{F2-Loo} for $q>n$, \eqref{F2-Loo-q=n} for $q=n$, and  \eqref{LooLq-decay}, we can choose $L$ so that $|I_6|\le \e$. We then take $|x|$ sufficiently large to make $|I_5|\le \e$. 
This proves that $J(x,t)\to0$ as $|x|\to\infty$, and hence that $\td d(x,t)\to0$ as $|x|\to\infty$, uniformly in $t\in(0,T)$. \hfill \qed
\end{remark}

\subsection{Construction of mild solutions for the Dirichlet problem}\label{sec:mild-NLCF-D}

In this subsection, we prove Theorems \ref{mild-NLCF-D-Ya}\,--\,\ref{mild-NLCF-D-Lq-uloc}.
We consider the nematic liquid crystal flow with Dirichlet boundary condition for orientation field, \eqref{eq-NLCF-D}-\eqref{eq-NLCF-D-compatible}.
Assuming the constant boundary condition for the orientation field, \eqref{eq-NLCF-D-compatible}, we construct mild solutions in the corresponding frameworks.

Let $\td d(x,t) = d(x,t) - d_*$ and $\td d_0(x) = d(x,0) - d_*$.
By \eqref{d-mild-NLCF-D}, we have, for $\ell=1,\ldots,L$,
\EQN{
\td d_\ell(x,t) &= \int_{\R^n_+} G^D(x,y,t) (\td d_0)_\ell(y)\, dy + \int_0^t \!\! \int_{\R^n_+} G^D(x,y,t-s)\!\bkt{ |\nb \td d|^2 {d_\ell}-(u\cdot\nb)\td d_\ell }\! (y,s)\, dyds.
}

\subsubsection{Construction of mild solutions in $Y_a$}
For $0\le a\le n$,
define the Banach space
\[
X = \bket{(u,\td d) :\R^n_+ \to \R^n \times S^{L-1} \ \big| \ 
u\in Y_a,\ \td d\in L^\infty,\, \nb\td d\in Y_a},
\]
with the norm
\[
\norm{(u,\td d)}_X = \norm{u}_{Y_a} + \norm{\td d}_{L^\infty} + \norm{\nb\td d}_{Y_a}.
\]
For a given $(u_0,\td d_0)\in X$ define inductively a sequence $(u^m,\td d^m)$ by
\[
(u^{m+1},\td d^{m+1})(t) = (e^{-tA}u_0, e^{t\De^D}\td d_0) + F(u^m,\td d^m)(t),\quad m \in \NN_0,
\]
where $A=-\mathbb P\De$ is the Stokes operator, $F = (F_1,F_2)$ and for $i=1,\ldots,n$,
\EQ{\label{eq-F1-db}
(F_1(u,\td d))_i(x,t) =  \sum_{j,k=1}^n \sum_{\ell=1}^L \int_0^t\int_{\R^n_+} \pd_{y_k}G_{ij}(x,y,t-s)\bke{u_ku_j + \pd_k\td d_\ell \pd_j\td d_\ell}(y,s)\, dyds,
}
and for $\ell = 1,\ldots,L$,
\EQ{\label{eq-F2-db}
(F_2(u,\td d))_\ell(x,t) = \int_0^t \int_{\R^n_+} G^D(x,y,t-s) \bkt{ -(u\cdot\nb)\td d_\ell + |\nb \td d|^2 d_\ell}(y,s)\, dyds.
}
The main difference from \eqref{eq-F1-def}--\eqref{eq-F2-def} is that $G^N$ in \eqref{eq-F2-def} is replaced by $G^D$ in \eqref{eq-F2-db}.

Let $T>0$ small and to be determined.
By \cite[Lemma 9.2]{Green}, we have, for $t\in(0,T)$, that
\EQ{\label{eq-NLCF-D-Ya-linear-1}
\norm{e^{-tA}u_0}_{Y_a} \lec (1+\de_{an}\log_+t)\norm{u_0}_{Y_a},
}
and (when $n=2$ and $a=0$, we also use Lemma \ref{lem-9.1}) 
\EQ{\label{eq-NLCF-D-Ya-nonlinear-1}
\norm{F_1(u,\td d)(t)}_{Y_a} \lec t^{1/2} \sup_{0\le t\le T} \bke{ \norm{u(t)}_{Y_a}^2 + \norm{\nb\td d(t)}_{Y_a}^2 }
\lec T^{1/2} \norm{(u,\td d)}_{L^\infty_TX}^2.
}
Moreover, by \thref{lem-heat-est-Ya}, we have, for $t\in(0,T)$, that
\EQ{\label{eq-NLCF-D-Ya-linear-2.1}
\norm{e^{t\De^D}\td d_0}_{L^\infty} \lec \norm{\td d_0}_{L^\infty}.
}
To bound $\nb e^{t\De^D}\td d_0$ in $Y_a$, note that
\EQS{\label{eq-db-d-linear}
\pd_{x_j} &e^{t\De^D}\td d_0(x,t)
= \int_{\R^n_+} \bke{\pd_j\Ga(x-y,t) - \pd_j\Ga(x-y^*,t)} \td d_0(y)\, dy\\
&= \int_{\R^n_+} \bke{- \pd_{y_j}\Ga(x-y,t) +\ep_j \pd_{y_j}\Ga(x-y^*,t)} \td d_0(y)\, dy\\
&= \int_{\R^n_+} \bke{\Ga(x-y,t) - \ep_j \Ga(x-y^*,t)} \pd_j\td d_0(y)\, dy + 2\de_{jn}\int_\Si \Ga(x-y',t) \td d_0(y',0)\, dy'\\
&= \int_{\R^n_+} \bke{\Ga(x-y,t) - \ep_j \Ga(x-y^*,t)} \pd_j\td d_0(y)\, dy,
}
where we've used the compatibility condition \eqref{eq-NLCF-D-compatible} that  $\td d_0(y',0) = d_0(y',0) - d_* = 0$.
By \eqref{eq-linear-Ya}, $\norm{ \int  (\Ga(x-y,t) - \ep_j\Ga(x-y^*,t)) f(y) dy }_{Y_a} \lec (1+\de_{an}\log_+ t) \norm{f}_{Y_a}$.
So,
\EQ{\label{eq-NLCF-D-Ya-linear-2.3}
\norm{\nb e^{t\De^D}\td d_0}_{Y_a} \lec (1+\de_{an}\log_+t)\norm{\nb\td d_0}_{Y_a}.
}
The rest of the proof of \thref{mild-NLCF-D-Ya} is exactly the same as what follows \eqref{eq-NLCF-N-Ya-linear-2.3} in the proof of \thref{mild-NLCF-N-Ya} for the Neumann problem because the nonlinear terms $F_1$ and $F_2$ for the Dirichlet problem have the same estimate as in \eqref{eq-NLCF-N-Ya-nonlinear-2.1}-\eqref{eq-NLCF-N-Ya-nonlinear-diff-2}.
It is worth mentioning that $a=0$ is allowed (unlike \thref{mild-NLCF-N-Ya}) since the formula \eqref{d-mild-NLCF-D} in Proposition \ref{prop-NLCF-sol-formula-D} for the Dirichlet problem does not require the spatial decay assumption in Proposition \ref{prop-NLCF-sol-formula-N}.
We leave out the details.

\subsubsection{Construction of mild solution in $L^q$}
The proof of Theorem \ref{mild-NLCF-D-Lq} follows a similar argument as presented in Sections \ref{sec:NLCF-N-Lq} and \ref{sec:NLCF-N-Ln}.
To establish this, we redefine $\td d$ as $\td d = d-d_*$, replace $\De^N$ with $\De^D$, and employ the formula \eqref{eq-db-d-linear} for $\nb e^{t\De^D}\td d_0$.
Note that, since the solution formula \eqref{d-mild-NLCF-D} for the Dirichlet problem does not require the spatial decay assumption in Proposition \ref{prop-NLCF-sol-formula-N}, the case $q=\infty$ is allowed.
For brevity, we omit the detailed steps of the proof.

\subsubsection{Construction of mild solution in $L^q_\uloc$}
Now, we turn to the $L^q_\uloc$ setting and prove Theorem \ref{mild-NLCF-D-Lq-uloc}.
Let $\td d= d - d_*$ and
\EQS{\label{eq-banach-Lq-uloc}
X_T &= \Big\{(u,\td d)\in L^\infty(0,T;L^q_\uloc\times L^\infty): 
\nb\td d \in L^\infty(0,T;L^q_\uloc),\\ 
&\qquad t^{\frac{n}{2q}}u,\,t^{\frac{n}{2q}}\nb\td d \in L^\infty(0,T;L^\infty),\ 
t^{\frac12} \nb u,\, t^{\frac12}\nb^2\td d\in L^\infty(0,T;L^q_\uloc)
\Big\} ,
}
with the norm
\EQN{
&\norm{(u,\td d)}_{X_T} 
:= \sup_{0\le t\le T} \bke{\norm{u(\cdot,t)}_{L^q_\uloc} + \norm{\td d(\cdot,t)}_{L^\infty} + \norm{\nb d(\cdot,t)}_{L^q_\uloc} }\\
&\quad + \sup_{0\le t\le T} \bket{ t^{\frac{n}{2q}} \bke{\norm{u(\cdot,t)}_{L^\infty} + \norm{\nb \td d(\cdot,t)}_{L^\infty}  } 
+t^{\frac12} \bke{\norm{\nb u(\cdot,t)}_{L^q_\uloc} + \norm{\nb^2 d(\cdot,t)}_{L^q_\uloc} }}.
}
Applying \cite[Lemma 9.4]{Green}, \cite[Proposition 5.3 (5-2)]{MMP1},
 and Lemma \ref{lem-heat-est-Lq-uloc}, we derive the following estimates%
\EQN{
\norm{e^{-tA}u_0}_{L^q_\uloc} &\lec \norm{u_0}_{L^q_\uloc},\quad
\norm{e^{-tA}u_0}_{L^\infty} \lec (1+t^{-\frac{n}{2q}}) \norm{u_0}_{L^q_\uloc},\quad
\\
\norm{\nb e^{-tA}u_0}_{L^q_\uloc} &\lec t^{-\frac12} \norm{u_0}_{L^q_\uloc},\\
\norm{e^{t\De^D}\td d_0}_{L^\infty} &\lec \norm{\td d_0}_{L^\infty},\quad
\norm{\nb e^{t\De^D}\td d_0}_{L^q_\uloc} \lec \norm{\nb\td d_0}_{L^q_\uloc},\\
\norm{\nb e^{t\De^D}\td d_0}_{L^\infty} &\lec (1+t^{-\frac{n}{2q}}) \norm{\nb\td d_0}_{L^q_\uloc},\quad
\norm{\nb^2 e^{t\De^D}\td d_0}_{L^q_\uloc} \lec t^{-\frac12} \norm{\nb\td d_0}_{L^q_\uloc}.
}
Note that the estimate  $\norm{\nb e^{-tA}u_0}_{L^q_\uloc} \lec t^{-\frac12} \norm{u_0}_{L^q_\uloc}$ for $1<q\le \infty$ is shown in \cite[Proposition 5.3 (5-2)]{MMP1}. It also
follows from \cite[Lemma 9.4, (9.20)]{Green} for $1\le q<\infty$ including $q=1$, and  \cite[(1.8)]{BJ2012} for $q=\infty$. Also note that we've used the formula \eqref{eq-db-d-linear} for $\nb e^{t\De^D}\td d_0$.

These estimates imply
\EQ{\label{eq-Lq-uloc-mild-linear}
\norm{(e^{-tA}u_0, e^{t\De^D}\td d_0)}_{X_T} \lec \bke{ 1+T^{\frac{n}{2q}} } \bke{\norm{u_0}_{L^q_\uloc} + \norm{\td d_0}_{L^\infty} + \norm{\nb d_0}_{L^q_\uloc} }.
}

Denote $F_1$ and $F_2$ the nonlinear terms defined as in \eqref{eq-F1-db} and \eqref{eq-F2-db}, and let $F = (F_1,F_2)$.
Using the estimates in \cite[Lemma 9.4]{Green} and Lemmas \ref{lem-heat-est-Lq-uloc} and \ref{lem-9.1}, we follow a similar procedure deriving \eqref{eq-Lq-mild-nonlinear} and \eqref{eq-diff-XT} to get
\EQS{\label{eq-Lq-uloc-mild-nonlinear}
\norm{F(u,\td d)}_{X_T} 
&\lec \bkt{ \bke{T^{\frac12} + T^{\frac12-\frac{n}{2q}} } + \bke{ T+ T^{1-\frac{n}{q}} } } \norm{(u,\td d)}_{X_T}^2 \bke{1 + \norm{(u,\td d)}_{X_T}}\\
&\lec \bke{ T+ T^{1-\frac{n}{q}} }\norm{(u,\td d)}_{X_T}^2 \bke{1 + \norm{(u,\td d)}_{X_T}},\quad q>n.
}
and 
\EQ{
\norm{F(u,\td d) - F(v,\td e)}_{X_T}
 \lec T^{1-\frac{n}{q}}  \bke{1 + \norm{(u,\td d)}_{X_T} + \norm{(v,\td e)}_{X_T}} ^2
\norm{(u,\td d) - (v,\td e)}_{X_T}.
}

With the above estimates, it readily follows the same procedure as in the proof of \thref{mild-NLCF-N-Lq} to construct a mild solution $(u,\td d)\in X_T$ for $q>n$, which proves Part (a) of Theorem \ref{mild-NLCF-D-Lq-uloc}.

For the critical case $q=n$, we use the auxiliary space
\[
Y_T = \bket{(u,\td d): 
\td d \in L^\infty(0,T;L^\infty),\, t^{\frac14}u,\, t^{\frac14}\nb\td d\in L^\infty(0,T;L^{2n}_\uloc)}
\]
with the norm 
\[
\norm{(u,\td d)}_{Y_T} := \sup_{0\le t\le T} t^{\frac14} \bke{\norm{u(\cdot,t)}_{L^{2n}_\uloc} + \norm{\nb d(\cdot,t)}_{L^{2n}_\uloc} }.
\]
Following the same argument as in the second part of the proof of \thref{mild-NLCF-N-Lq} for $q=n$,
we obtain a mild solution $(u,\td d)\in X_T$ provided $\norm{u_0}_{L^n_\uloc} + \norm{\td d_0}_{L^\infty} + \norm{\nb d_0}_{L^n_\uloc}$ is sufficiently small.
The smallness of initial data is required in the $L^n_\uloc$ framework since, in contrast to \eqref{eq-Lq-mild-linear-critical-YT}, the convergence $\lim_{T\to0} \norm{(e^{-tA}u_0, e^{t\De^N}\td d_0)}_{Y_T}=0$ may not hold true since $C^\infty_c$ is not dense in $L^n_\uloc$. (It will be fine if we further assume that
$u_0$ and $\nb d_0$ are in the closure of smooth functions of compact supports in $L^n_\uloc$-norm.)
This completes the proof of Part (b) of Theorem \ref{mild-NLCF-D-Lq-uloc}.
\qed

\section*{Acknowledgments}
We thank Yifu Zhou for pointing out useful references and critical discussions on nematic liquid crystal flow.
The research of KK was partially supported by NRF-2019R1A2C1084685.
The research of BL was partially supported by NSFC-12371202 and Hunan provincial NSF-
2022jj10032,22A0057.
The research of CL is partially supported by Simons Foundation Math + X Investigator Award \#376319 (Michael I. Weinstein).
The research of TT was partially supported by the NSERC grant 
 RGPIN-2023-04534.

\addcontentsline{toc}{section}{\protect\numberline{}{References}}
\small

\def\cprime{$'$}


\begin{thebibliography}{10}

\bibitem{B2006}
H.-O. Bae.
\newblock Temporal decays in {$L^1$} and {$L^\infty$} for the {S}tokes flow.
\newblock {\em J. Differential Equations}, 222(1):1--20, 2006.

\bibitem{BJ2012}
H.-O. Bae and B.~J. Jin.
\newblock Existence of strong mild solution of the {N}avier-{S}tokes equations
  in the half space with nondecaying initial data.
\newblock {\em J. Korean Math. Soc.}, 49(1):113--138, 2012.

\bibitem{Batchelor-book1967}
G.~K. Batchelor.
\newblock {\em An introduction to fluid dynamics}.
\newblock Cambridge university press, 1967.

\bibitem{B2022}
T.~A. Bui.
\newblock Temporal-spatial decays for the {S}tokes flows in half spaces.
\newblock {\em NoDEA Nonlinear Differential Equations Appl.}, 29(3):Paper No.
  23, 22, 2022.

\bibitem{CJ-NA2017}
T.~Chang and B.~J. Jin.
\newblock Pointwise decay estimate of {N}avier-{S}tokes flows in the half space
  with slowly decreasing initial value.
\newblock {\em Nonlinear Anal.}, 157:167--188, 2017.

\bibitem{CWQ-IJMMS1998}
F.~Chen, P.~Wang, and C.~Qu.
\newblock On the differential system governing flows in magnetic field with
  data in {$L^p$}.
\newblock {\em Internat. J. Math. Math. Sci.}, 21(2):299--305, 1998.

\bibitem{CY-ARMA2017}
Y.~Chen and Y.~Yu.
\newblock Global {$m$}-equivariant solutions of nematic liquid crystal flows in
  dimension two.
\newblock {\em Arch. Ration. Mech. Anal.}, 226(2):767--808, 2017.

\bibitem{CM2004}
F.~Crispo and P.~Maremonti.
\newblock On the {$(x,t)$} asymptotic properties of solutions of the
  {N}avier-{S}tokes equations in the half-space.
\newblock {\em Zap. Nauchn. Sem. S.-Peterburg. Otdel. Mat. Inst. Steklov.
  (POMI)}, 318(Kraev. Zadachi Mat. Fiz. i Smezh. Vopr. Teor. Funkts. 36
  [35]):147--202, 311, 2004.

\bibitem{DR-Proyec1997}
P.~D. Dam\'{a}zio and M.~A. Rojas-Medar.
\newblock On some questions of the weak solutions of evolution equations for
  magnetohydrodynamic type.
\newblock {\em Proyecciones}, 16(2):83--97, 1997.

\bibitem{Davidson-book2001}
P.~A. Davidson.
\newblock {\em An introduction to magnetohydrodynamics}.
\newblock Cambridge Texts in Applied Mathematics. Cambridge University Press,
  Cambridge, 2001.

\bibitem{DHP}
W.~Desch, M.~Hieber, and J.~Pr\"{u}ss.
\newblock {$L^p$}-theory of the {S}tokes equation in a half space.
\newblock {\em J. Evol. Equ.}, 1(1):115--142, 2001.

\bibitem{DXX-SciChinaMath2022}
Q.~Duan, Y.~Xiao, and Z.~Xin.
\newblock On the vanishing dissipation limit for the incompressible {MHD}
  equations on bounded domains.
\newblock {\em Sci. China Math.}, 65(1):31--50, 2022.

\bibitem{DL-ARMA1972}
G.~Duvaut and J.-L. Lions.
\newblock In\'{e}quations en thermo\'{e}lasticit\'{e} et
  magn\'{e}tohydrodynamique.
\newblock {\em Arch. Rational Mech. Anal.}, 46:241--279, 1972.

\bibitem{Ericksen-ARMA1962}
J.~L. Ericksen.
\newblock Hydrostatic theory of liquid crystals.
\newblock {\em Arch. Rational Mech. Anal.}, 9:371--378, 1962.

\bibitem{GMS-MZ1999}
Y.~Giga, S.~Matsui, and Y.~Shimizu.
\newblock On estimates in {H}ardy spaces for the {S}tokes flow in a half space.
\newblock {\em Math. Z.}, 231(2):383--396, 1999.

\bibitem{H2014}
P.~Han.
\newblock Long-time behavior for the nonstationary {N}avier-{S}tokes flows in
  {$L^1(\Bbb{R}_+^n)$}.
\newblock {\em J. Funct. Anal.}, 266(3):1511--1546, 2014.

\bibitem{H2018}
P.~Han.
\newblock Decay results of the nonstationary {N}avier-{S}tokes flows in
  half-spaces.
\newblock {\em Arch. Ration. Mech. Anal.}, 230(3):977--1015, 2018.

\bibitem{Hide-geophyics1971}
R.~Hide.
\newblock On planetary atmospheres and interiors.
\newblock {\em Mathematical Problems in the Geophisical Sciences}, 1, 1971.

\bibitem{HNPS-AnnIHPoincareAN2016}
M.~Hieber, M.~Nesensohn, J.~Pr\"{u}ss, and K.~Schade.
\newblock Dynamics of nematic liquid crystal flows: the quasilinear approach.
\newblock {\em Ann. Inst. H. Poincar\'{e} C Anal. Non Lin\'{e}aire},
  33(2):397--408, 2016.

\bibitem{HWW-JDE2019}
J.~Huang, C.~Wang, and H.~Wen.
\newblock Time decay rate of global strong solutions to nematic liquid crystal
  flows in {$\Bbb{R}_+^3$}.
\newblock {\em J. Differential Equations}, 267(3):1767--1804, 2019.

\bibitem{HLLW-ARMA2016}
T.~Huang, F.~Lin, C.~Liu, and C.~Wang.
\newblock Finite time singularity of the nematic liquid crystal flow in
  dimension three.
\newblock {\em Arch. Ration. Mech. Anal.}, 221(3):1223--1254, 2016.

\bibitem{Hynd-SIMA2013}
R.~Hynd.
\newblock Partial regularity of weak solutions of the viscoelastic
  {N}avier-{S}tokes equations with damping.
\newblock {\em SIAM J. Math. Anal.}, 45(2):495--517, 2013.

\bibitem{KK-JDE2012}
K.~Kang and J.-M. Kim.
\newblock Regularity criteria of the magnetohydrodynamic equations in bounded
  domains or a half space.
\newblock {\em J. Differential Equations}, 253(2):764--794, 2012.

\bibitem{KK-JFA2014}
K.~Kang and J.-M. Kim.
\newblock Boundary regularity criteria for suitable weak solutions of the
  magnetohydrodynamic equations.
\newblock {\em J. Funct. Anal.}, 266(1):99--120, 2014.

\bibitem{KLLT-TAMS2022}
K.~Kang, B.~Lai, C.-C. Lai, and T.-P. Tsai.
\newblock Finite energy {N}avier-{S}tokes flows with unbounded gradients
  induced by localized flux in the half-space.
\newblock {\em Trans. Amer. Math. Soc.}, 375(9):6701--6746, 2022.

\bibitem{Green}
K.~Kang, B.~Lai, C.-C. Lai, and T.-P. Tsai.
\newblock The {G}reen tensor of the nonstationary {S}tokes system in the half
  space.
\newblock {\em Comm. Math. Phys.}, 399(2):1291--1372, 2023.

\bibitem{Green-errata}
K.~Kang, B.~Lai, C.-C. Lai, and T.-P. Tsai.
\newblock Errata to: The {G}reen tensor of the nonstationary {S}tokes system in
  the half space.
\newblock 2024.
\newblock submitted to Comm. Math. Phys.

\bibitem{Kato1984}
T.~Kato.
\newblock Strong {$L^{p}$}-solutions of the {N}avier-{S}tokes equation in
  {${\bf R}^{m}$}, with applications to weak solutions.
\newblock {\em Math. Z.}, 187(4):471--480, 1984.

\bibitem{Kim-ActaMathSci2017}
J.-M. Kim.
\newblock Local regularity criteria of a suitable weak solution to {MHD}
  equations.
\newblock {\em Acta Math. Sci. Ser. B (Engl. Ed.)}, 37(4):1033--1047, 2017.

\bibitem{Kim-AML2017}
J.-M. Kim.
\newblock On regularity criteria of weak solutions to the 3{D} viscoelastic
  {N}avier-{S}tokes equations with damping.
\newblock {\em Appl. Math. Lett.}, 69:153--160, 2017.

\bibitem{Kim-AdvMathPhys2022}
J.-M. Kim.
\newblock On regularity criteria via pressure for the 3{D} {MHD} equations in a
  half space.
\newblock {\em Adv. Math. Phys.}, pages Art. ID 6954802, 7, 2022.

\bibitem{MR191213}
G.~H. Knightly.
\newblock On a class of global solutions of the {N}avier-{S}tokes equations.
\newblock {\em Arch. Rational Mech. Anal.}, 21:211--245, 1966.

\bibitem{LS-1960}
O.~A. Ladyzhenskaya and V.~A. Solonnikov.
\newblock Mathematical problems of hydrodynamics and magnetohydrodynamics of a
  viscous incompressible fluid.
\newblock {\em Proceedings of VA Steklov Mathematical Institute},
  59(1):115--173, 1960.

\bibitem{LLW-SIMA2017}
B.~Lai, J.~Lin, and C.~Wang.
\newblock Forward self-similar solutions to the viscoelastic {N}avier-{S}tokes
  equation with damping.
\newblock {\em SIAM J. Math. Anal.}, 49(1):501--529, 2017.

\bibitem{Lai-JMFM2019}
C.-C. Lai.
\newblock Forward discretely self-similar solutions of the {MHD} equations and
  the viscoelastic {N}avier-{S}tokes equations with damping.
\newblock {\em J. Math. Fluid Mech.}, 21(3):Paper No. 38, 28, 2019.

\bibitem{lai-thesis2021}
C.-C. Lai.
\newblock {\em Study and analysis of some incompressible fluid PDEs: the
  Navier--Stokes equations in the half space, the MHD and the viscoelastic
  Navier--Stokes equations, and coupled Keller--Segel-fluid models}.
\newblock PhD thesis, University of British Columbia, 2021.

\bibitem{LLWWZ-CPAM2022}
C.-C. Lai, F.~Lin, C.~Wang, J.~Wei, and Y.~Zhou.
\newblock Finite time blowup for the nematic liquid crystal flow in dimension
  two.
\newblock {\em Comm. Pure Appl. Math.}, 75(1):128--196, 2022.

\bibitem{Lassner-ARMA1967}
G.~Lassner.
\newblock \"{U}ber ein {R}and-{A}nfangswertproblem der {M}agnetohydrodynamik.
\newblock {\em Arch. Rational Mech. anal.}, pages 388--405, 1967.

\bibitem{LR-book2002}
P.~G. Lemari\'{e}-Rieusset.
\newblock {\em Recent developments in the {N}avier-{S}tokes problem}, volume
  431 of {\em Chapman \& Hall/CRC Research Notes in Mathematics}.
\newblock Chapman \& Hall/CRC, Boca Raton, FL, 2002.

\bibitem{Leslie-ARMA1968}
F.~M. Leslie.
\newblock Some constitutive equations for liquid crystals.
\newblock {\em Arch. Rational Mech. Anal.}, 28(4):265--283, 1968.

\bibitem{LW-JDE2012}
X.~Li and D.~Wang.
\newblock Global solution to the incompressible flow of liquid crystals.
\newblock {\em J. Differential Equations}, 252(1):745--767, 2012.

\bibitem{LLW-ARMA2010}
F.~Lin, J.~Lin, and C.~Wang.
\newblock Liquid crystal flows in two dimensions.
\newblock {\em Arch. Ration. Mech. Anal.}, 197(1):297--336, 2010.

\bibitem{LL-JPDE2001}
F.~Lin and C.~Liu.
\newblock Static and dynamic theories of liquid crystals.
\newblock {\em J. Partial Differential Equations}, 14(4):289--330, 2001.

\bibitem{LW-ChineseAnnMath2010}
F.~Lin and C.~Wang.
\newblock On the uniqueness of heat flow of harmonic maps and hydrodynamic flow
  of nematic liquid crystals.
\newblock {\em Chinese Ann. Math. Ser. B}, 31(6):921--938, 2010.

\bibitem{LW-CPAM2016}
F.~Lin and C.~Wang.
\newblock Global existence of weak solutions of the nematic liquid crystal flow
  in dimension three.
\newblock {\em Comm. Pure Appl. Math.}, 69(8):1532--1571, 2016.

\bibitem{LZ-CPAM2014}
F.~Lin and P.~Zhang.
\newblock Global small solutions to an {MHD}-type system: the three-dimensional
  case.
\newblock {\em Comm. Pure Appl. Math.}, 67(4):531--580, 2014.

\bibitem{Lin-CPAM1989}
F.-H. Lin.
\newblock Nonlinear theory of defects in nematic liquid crystals; phase
  transition and flow phenomena.
\newblock {\em Comm. Pure Appl. Math.}, 42(6):789--814, 1989.

\bibitem{LL-CPAM1995}
F.-H. Lin and C.~Liu.
\newblock Nonparabolic dissipative systems modeling the flow of liquid
  crystals.
\newblock {\em Comm. Pure Appl. Math.}, 48(5):501--537, 1995.

\bibitem{MMP1}
Y.~Maekawa, H.~Miura, and C.~Prange.
\newblock Estimates for the {N}avier--{S}tokes equations in the half-space for
  nonlocalized data.
\newblock {\em Anal. PDE}, 13(4):945--1010, 2020.

\bibitem{MaTe}
Y.~Maekawa and Y.~Terasawa.
\newblock The {N}avier-{S}tokes equations with initial data in uniformly local
  {$L^p$} spaces.
\newblock {\em Differential Integral Equations}, 19(4):369--400, 2006.

\bibitem{Maremonti-JMFM2011}
P.~Maremonti.
\newblock A remark on the {S}tokes problem with initial data in {$L^1$}.
\newblock {\em J. Math. Fluid Mech.}, 13(4):469--480, 2011.

\bibitem{NY-JMAA2021}
J.~Neustupa and M.~Yang.
\newblock A new sufficient condition for local regularity of a suitable weak
  solution to the {MHD} equations.
\newblock {\em J. Math. Anal. Appl.}, 502(2):Paper No. 125258, 28, 2021.

\bibitem{Pikelner-book1966}
S.~Pikelner.
\newblock Grundlangen der kosmischen elektrohydrodynamik, 1966.

\bibitem{PW-book1967}
M.~H. Protter and H.~F. Weinberger.
\newblock {\em Maximum principles in differential equations}.
\newblock Prentice-Hall, Inc., Englewood Cliffs, NJ, 1967.

\bibitem{QKC-AML2000}
C.~Qu, K.-H. Kwek, and F.~Chen.
\newblock The global asymptotic stability for {F}-{M} equation.
\newblock {\em Appl. Math. Lett.}, 13(2):35--41, 2000.

\bibitem{QSW-JMAA1994}
C.~S. Qu, S.~G. Song, and P.~Wang.
\newblock On the equations for the flow and the magnetic field within the
  {E}arth.
\newblock {\em J. Math. Anal. Appl.}, 187(3):1003--1018, 1994.

\bibitem{RB-Proyec1994}
M.~A. Rojas-Medar and J.~L. Boldrini.
\newblock The weak solutions and reproductive property for a system of
  evolution equations of magnetohydrodynamic type.
\newblock {\em Proyecciones}, 13(2):85--97, 1994.

\bibitem{RB-JAustMathSocB1997}
M.~A. Rojas-Medar and J.~L. Boldrini.
\newblock Global strong solutions of equations of magnetohydrodynamic type.
\newblock {\em J. Austral. Math. Soc. Ser. B}, 38(3):291--306, 1997.

\bibitem{Schluter-plasma1950}
A.~Schl{\"u}ter.
\newblock Dynamik des plasmas i: Grundgleichungen, plasma in gekreuzten
  feldern.
\newblock {\em Zeitschrift f{\"u}r Naturforschung A}, 5(2):72--78, 1950.

\bibitem{Schluter-plasma1951}
A.~Schl{\"u}ter.
\newblock Dynamik des plasmas ii: Plasma mit neutralgas.
\newblock {\em Zeitschrift f{\"u}r Naturforschung A}, 6(2):73--78, 1951.

\bibitem{ST-CPAM1983}
M.~Sermange and R.~Temam.
\newblock Some mathematical questions related to the {MHD} equations.
\newblock {\em Comm. Pure Appl. Math.}, 36(5):635--664, 1983.

\bibitem{Serrin-ARMA1962}
J.~Serrin.
\newblock On the interior regularity of weak solutions of the {N}avier-{S}tokes
  equations.
\newblock {\em Arch. Rational Mech. Anal.}, 9:187--195, 1962.

\bibitem{Shercliff-book1965}
J.~A. Shercliff.
\newblock Textbook of magnetohydrodynamics.
\newblock 1965.

\bibitem{MR0415097}
V.~A. Solonnikov.
\newblock Estimates of the solutions of the nonstationary {N}avier-{S}tokes
  system.
\newblock {\em Zap. Nau\v{c}n. Sem. Leningrad. Otdel. Mat. Inst. Steklov.
  (LOMI)}, 38:153--231, 1973.
\newblock In Russian; English translation in J. Soviet Math. 8(4):467-529,
  1977.

\bibitem{MR0460931}
V.~A. Solonnikov.
\newblock Estimates of the solution of a certain initial-boundary value problem
  for a linear nonstationary system of {N}avier-{S}tokes equations.
\newblock {\em Zap. Nau\v cn. Sem. Leningrad. Otdel Mat. Inst. Steklov.
  (LOMI)}, 59:178--254, 257, 1976.
\newblock English transl., J. Soviet Math. 10 (1978), 336-393.

\bibitem{MR1992567}
V.~A. Solonnikov.
\newblock Estimates for solutions of the nonstationary {S}tokes problem in
  anisotropic {S}obolev spaces and estimates for the resolvent of the {S}tokes
  operator.
\newblock {\em Uspekhi Mat. Nauk}, 58(2(350)):123--156, 2003.
\newblock {I}n Russian; English translation in \emph{Russian Math. Surveys} 58
  (2003), no. 2, 331–365.

\bibitem{Solonnikov-RMS2003}
V.~A. Solonnikov.
\newblock On estimates of solutions of the non-stationary stokes problem in
  anisotropic sobolev spaces and on estimates for the resolvent of the stokes
  operator.
\newblock {\em Russian Mathematical Surveys}, 58(2):331, 2003.

\bibitem{TWZ-CMS2020}
Z.~Tan, W.~Wu, and J.~Zhou.
\newblock Existence of mild solutions and regularity criteria of weak solutions
  to the viscoelastic {N}avier-{S}tokes equation with damping.
\newblock {\em Commun. Math. Sci.}, 18(1):205--226, 2020.

\bibitem{VS-POMI2010}
V.~Vyalov and T.~Shilkin.
\newblock On the boundary regularity of weak solutions to the {MHD} system.
\newblock {\em Zap. Nauchn. Sem. S.-Peterburg. Otdel. Mat. Inst. Steklov.
  (POMI)}, 385:18--53, 234, 2010.

\bibitem{MR4109418}
Z.~Zeng.
\newblock Mild solutions of the stochastic {MHD} equations driven by fractional
  {B}rownian motions.
\newblock {\em J. Math. Anal. Appl.}, 491(1):124296, 18, 2020.

\bibitem{Zhao-MMAS2003}
C.~Zhao.
\newblock Initial boundary value problem for the evolution system of {MHD} type
  describing geophysical flow in three-dimensional domains.
\newblock {\em Math. Methods Appl. Sci.}, 26(9):759--781, 2003.

\end{thebibliography}
\end{document}